\documentclass[11pt,reqno]{amsart}

\usepackage[english]{babel}
\usepackage{amsmath, amsthm, amssymb, color}
\usepackage[colorlinks=true,linkcolor=blue,urlcolor=blue]{hyperref}
\usepackage{graphicx}
\usepackage{blkarray}
\usepackage{caption}
\usepackage{subcaption}
\usepackage{mathtools}
\usepackage{enumerate}
\usepackage{verbatim}
\usepackage{tikz,tikz-cd,tikz-3dplot}
\usepackage{amssymb}
\usetikzlibrary{matrix}
\usetikzlibrary{arrows}
\usepackage{longtable}
\usepackage{algorithm}
\usepackage[noend]{algpseudocode}
\usepackage[normalem]{ulem}
\usepackage{color}

\oddsidemargin \evensidemargin
\usepackage{makecell}
\usepackage{array}
\usetikzlibrary{decorations.pathmorphing}

\DeclareCaptionLabelFormat{subfiglabel}{(#2)}
\newtheorem{theorem}{Theorem}
\newtheorem{proposition}[theorem]{Proposition}
\newtheorem{lemma}[theorem]{Lemma}

\theoremstyle{definition}
\newtheorem{definition}[theorem]{Definition}
\newtheorem{remark}[theorem]{Remark}
\newtheorem{example}[theorem]{Example}

\numberwithin{theorem}{section}
\theoremstyle{plain}

\newcommand{\PP}{\mathbb{P}}
\newcommand{\RR}{\mathbb{R}}

\DeclareMathOperator{\trop}{trop}
\DeclareMathOperator{\SGr}{SGr}
\DeclareMathOperator{\Gr}{Gr}
\DeclareMathOperator{\Dr}{Dr}

\title{Mustafin degenerations of six points in the projective plane and self-duality}

\author{Alheydis Geiger}
\address{Max-Planck-Institute for Mathematics in the Sciences, Inselstr. 22, 04103 Leipzig}
\email{geiger@mis.mpg.de}

\author{Luca Schaffler}
\address{Dipartimento di Matematica e Fisica, Universit\`a degli Studi Roma Tre, Largo San Leonardo Murialdo 1, 00146, Roma, Italy}
\email{luca.schaffler@uniroma3.it}

\begin{document}

\begin{abstract}
Mustafin varieties provide a framework for degenerating the projective plane together with marked points, governed by the combinatorics of Bruhat--Tits buildings. Such limit objects are parametrized by the boundary points of a projective compactification of the moduli space of $n$ points in $\mathbb{P}^2$ constructed by Schaffler--Tevelev. In this paper, we study the locus of this compactification generically parametrizing self-dual configurations, or equivalently, six points lying on a conic. We prove that this locus is a tropical compactification of the moduli space of self-dual configurations, with associated fan a refinement of the self-dual tropical Grassmannian introduced by Geiger--Hashimoto--Sturmfels--Vlad. We show that this compactification is a blow up of $\overline{\mathrm{M}}_{0,6}$ and we classify the objects it parametrizes: degenerations of $\mathbb{P}^2$ equipped with a stable $6$-pointed rational curve that records the limiting conic.
\end{abstract}

\maketitle

\section{Introduction}

Degenerations of point configurations on algebraic varieties give rise to compact moduli spaces~\cite{FM94,CGK09,GR17}. In this paper, we turn to the case of the moduli space $\mathbf{B}(3,n)$ of $n$ marked points in $\mathbb{P}^2$ up to automorphism, and more specifically to the compactification $\overline{\mathbf{X}}_{\mathrm{GP}}(3,n)$ constructed in \cite{ST22} --- originally proposed by Gerritzen and Piwek in \cite{GP91} --- that arises by degenerating the $n$-pointed projective plane using \emph{Mustafin varieties}. Roughly, these are one-parameter degenerations of projective space that are controlled by the combinatorics of the Bruhat--Tits building \cite{CHSW11}. Such degenerations can be appropriately endowed with $n$ marked points by considering the Mustafin variety associated to the \emph{stable lattices} of a one-parameter family of $n$ linearly general points in $\mathbb{P}^2$ (see Section~\ref{sec:preliminaries}).

The compactification $\overline{\mathbf{X}}_{\mathrm{GP}}(3,n)$ is a birational modification of Kapranov's Chow quotient compactification of the moduli space of $n$ lines in $\mathbb{P}^2$ \cite{Kap93,KT06}. Moreover, for $n=6$, it was shown in \cite{ST22} that $\overline{\mathbf{X}}_{\mathrm{GP}}(3,6)^\nu$ (for a reduced scheme $X$, let $X^{\nu}$ be its normalization) is the tropical compactification \cite{Tev07} associated to a specific refinement of the tropical Grassmannian $\trop(\mathrm{Gr}(3,6)^\circ)$ \cite{SS04}. The case of six points in $\PP^2$ is particularly rich because the points may be self-dual (\cite{Cob22,DO88}), or, equivalently, lie on a smooth conic. In the current paper we focus on the intersection of both these perspectives and investigate how self-duality degenerates at the boundary of $\overline{\mathbf{X}}_{\mathrm{GP}}(3,6)$.

Specifically, let $\mathbf{S}(3,6)\subseteq\mathbf{B}(3,6)$ be the closed subset of self-dual point configurations and denote by $\overline{\mathbf{S}}_{\mathrm{GP}}(3,6)$ its Zariski closure in $\overline{\mathbf{X}}_{\mathrm{GP}}(3,6)$. The next statement is a summary of our main results.

\begin{theorem}[{Theorem~\ref{thm:SGP36} and Theorem~\ref{thm:main-theorem-about-family-of-degenerate-conics}}]
The following hold.
\begin{enumerate}

\item The compactification $\overline{\mathbf{S}}_{\mathrm{GP}}(3,6)^\nu$ is tropical, with associated fan $\widehat{\Sigma}^{\mathrm{sd}}(3,6)$: see Definition~\ref{def:sd_refined_fan} and Proposition~\ref{prop:sd_refined_fan}. The correspondence between the boundary strata of $\overline{\mathbf{S}}_\mathrm{GP}(3,6)^\nu$ and the cones of $\widehat{\Sigma}^{\mathrm{sd}}(3,6)$ is spelled out in Table~\ref{tab:correspondencekapranov'sstrataandcones_selfdual}.

\item There exists a proper flat family of surfaces $\overline{\mathcal{M}}^{\mathrm{sd}}(3,6)\rightarrow\overline{\mathbf{S}}_{\mathrm{GP}}(3,6)^\nu$ with a divisor $\overline{\mathcal{D}}^{\mathrm{sd}}(3,6)\subseteq\overline{\mathcal{M}}^{\mathrm{sd}}(3,6)$ relative over $\overline{\mathbf{S}}_{\mathrm{GP}}(3,6)^\nu$ together with disjoint sections $s_i^{\mathrm{sd}}\colon\overline{\mathbf{S}}_{\mathrm{GP}}(3,6)^\nu\rightarrow\overline{\mathcal{D}}^{\mathrm{sd}}(3,6)$, $i=1,\ldots,6$, such that, over a point $x$, the fiber
\[
(\overline{\mathcal{M}}^{\mathrm{sd}}(3,6)_x,\overline{\mathcal{D}}^{\mathrm{sd}}(3,6)_x;s_1^{\mathrm{sd}}(x),\ldots,s_6^{\mathrm{sd}}(x))=(X,C;p_1,\ldots,p_6)
\]
is described as follows:

\begin{itemize}

\item If $x\in\mathbf{S}(3,6)$, then $X=\mathbb{P}^2$ and $C$ is the smooth conic passing through $p_1,\ldots,p_6$;

\item Otherwise, $(X;p_1,\ldots,p_6)$ is a $6$-pointed Mustafin degeneration of $\mathbb{P}^2$ (see Definition~\ref{def:n-pointed-Mustafin-degeneration}) and $C\subseteq X$ is a reduced curve interpolating $p_1,\ldots,p_6$ such that $(C;p_1,\ldots,p_6)$ is a stable $6$-pointed rational curve. The possibilities for $(X,C;p_1,\ldots,p_6)$ are listed in Tables~\ref{tbl:GPdegenerationsself-dual-1} and \ref{tbl:GPdegenerationsself-dual-2}.
\end{itemize}

\end{enumerate}
\end{theorem}

To prove the first part of the theorem, the starting point is understanding the tropical fan describing $\overline{\mathbf{S}}_{\mathrm{GP}}(3,6)^\nu$. On the one hand, we have the tropical fan $\widehat{\Sigma}(3,6)$ associated to $\overline{\mathbf{X}}_{\mathrm{GP}}(3,6)^\nu$ in \cite[Section~12]{ST22}, whose cones are listed explicitly in Table~\ref{tab:correspondencekapranov'sstrataandcones_refinement}. On the other hand, we have the tropical self-dual Grassmannian $\mathrm{trop}(\mathrm{SGr(3,6)}^\circ)$ studied in \cite{GHSV24}. We define $\widehat{\Sigma}^{\mathrm{sd}}(3,6)$ as the common refinement of these two fans, which induces a normal tropical compactification $\widehat{\mathbf{S}}(3,6)$ of $\mathbf{S}(3,6)$ (see Proposition~\ref{prop:widehatS(3,6)-is-normal-and-tropical}). Then, we prove that $\widehat{\mathbf{S}}(3,6)$ is isomorphic to the normalization of $\overline{\mathbf{S}}_{\mathrm{GP}}(3,6)$ in Theorem~\ref{thm:SGP36}.

Although the existence of a flat family of conics through the six marked points is clear over the interior $\mathbf{S}(3,6)$, it is not clear a priori how to extend it to the compactification $\overline{\mathbf{S}}_{\mathrm{GP}}(3,6)$. One possible solution is to take the Zariski closure of the generic family of conics inside the family of $6$-pointed Mustafin degenerations $\overline{\mathcal{M}}(3,6)\rightarrow\overline{\mathbf{X}}_{\mathrm{GP}}(3,6)$ (see Theorem~\ref{thm:main-on-XbarGP(3,n)}), but this may not give expected properties such as flatness and reduced fibers. Instead, we first control how the closure behaves for one-parameter families: see Section~\ref{sec:one-parameter-Mustafin-degens-and-conics}. In Figure~\ref{fig:example_degeneration} we showcase an example of these one-parameter degenerations. With a full understanding of the degenerating behavior in one-parameter families, we construct $\overline{\mathcal{D}}^{\mathrm{sd}}(3,6)\rightarrow\overline{\mathbf{S}}_{\mathrm{GP}}(3,6)^\nu$ globally by appropriately using the theory of \emph{multigraded Hilbert schemes} \cite{HS04}. This theory was already employed in \cite{ST22} to construct, more in general, $\overline{\mathcal{M}}(3,n)\rightarrow\overline{\mathbf{X}}_{\mathrm{GP}}(3,n)$. The need for normalization is technical and linked to the use of multigraded Hilbert schemes: see Section~\ref{sec:construction-of-the-global-family-of-degenerate-conics}, and especially Proposition~\ref{prop:barSGPc(3,6)-and-barSGP(3,6)-have-iso-normalizations}.

\begin{figure}
\begin{center}
\begin{tikzpicture}[scale=0.5]

\draw[line width=1pt] (-9,6) -- (-3,6);
\draw[line width=1pt] (-3,6) -- (-3,0);
\draw[line width=1pt] (-3,0) -- (-9,6);

\begin{scope}[rotate around={-30:(-5,4.2)}]

\draw[magenta,line width=1.5pt]
(-5,4.2) ellipse (1.4 and 0.9);

\fill   (5.4-9,4.2)  circle (4pt);
\fill	(4.7-9,4.979) circle (4pt);
\fill	(3.3-9,4.979) circle (4pt);
\fill	(2.6-9,4.2) circle (4pt);
\fill	(3.3-9,3.421) circle (4pt);
\fill	(4.7-9,3.421)circle (4pt);

\node[right] at (5.4-9,4.2) {\footnotesize{$4$}};
\node[above right] at (4.7-9,4.979) {\footnotesize{$3$}};
\node[above left] at (3.3-8.65,5.1)  {\footnotesize{$2$}};
\node[left] at (2.6-9,4.2)   {\footnotesize{$1$}};
\node[below left] at (3.3-8.7,3.621)  {\footnotesize{$6$}};
\node[below right] at (4.7-9.2,3.2) {\footnotesize{$5$}};
\end{scope}

\draw[->,thick,decorate,decoration={snake,amplitude=1mm,segment length=8mm}] (-1.7,3) -- (2.2,3) node[midway, above] {\footnotesize{degenerate}};

\draw[line width=1pt] (0+2,6) -- (6+2,6);
\draw[line width=1pt] (6+2,6) -- (6+2,0);
\draw[line width=1pt] (6+2,0) -- (0+2,6);
\draw[line width=1pt] (4.5+2,6) -- (4.5+2,1.5);
\draw[line width=1pt] (1.5+2,4.5) -- (6+2,4.5);
\fill[gray!50,nearly transparent] (1.5+2,4.5) -- (4.5+2,4.5) -- (4.5+2,1.5) -- cycle;
\fill[gray!50,nearly transparent] (1.5+2,4.5) -- (4.5+2,4.5) -- (4.5+2,6) -- (0+2,6) -- cycle;
\fill[gray!50,nearly transparent] (4.5+2,1.5) -- (4.5+2,4.5) -- (6+2,4.5) -- (6+2,0) -- cycle;

\node at (-0.4+2,6) {\footnotesize{$1$}};
\node at (1.8+2,5.1) {\footnotesize{$2$}};
\node at (3.1+2,5.1) {\footnotesize{$3$}};
\node at (6.4+2,0) {\footnotesize{$4$}};
\node at (5.1+2,1.6) {\footnotesize{$5$}};
\node at (5.2+2,3.2) {\footnotesize{$6$}};

\draw[magenta,line width=1.5pt] (0+2,6) to[out=355, in=180] (1.5+2,5.5) to[out=45, in=160] (3+2,5.6) to[out=350, in=180] (4.5+2,4.5) to[out=90, in=355] (4+2,5);

\draw[magenta,line width=1.5pt] (5+2,4) to[out=160, in=270] (4.5+2,4.5) to[out=0, in=75] (5.6+2,3) to[out=255, in=115] (5.5+2,1.5)  to[out=270, in=90] (6+2,0);

\fill (6+2,0) circle (4pt);
\fill (0+2,6) circle (4pt);
\fill (5.5+2,1.5) circle (4pt);
\fill (5.6+2,3) circle (4pt);
\fill (1.5+2,5.5) circle (4pt);
\fill (3+2,5.6) circle (4pt);
\end{tikzpicture}
\end{center}
\caption{Mustafin degeneration of $\mathbb{P}^2$ with $6$ self-dual marked points with the limit conic. For an explanation of the picture, see Remark~\ref{rmk:pictures-explained} and Example~\ref{ex:deg_conic}.}
\label{fig:example_degeneration}
\end{figure}
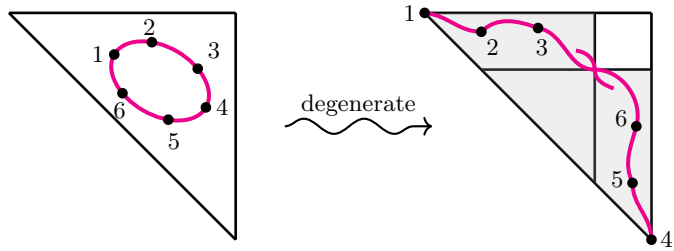

The family of stable $6$-pointed rational curves $\overline{\mathcal{D}}^{\mathrm{sd}}(3,6)\rightarrow\overline{\mathbf{S}}_{\mathrm{GP}}(3,6)^\nu$ gives a morphism $\overline{\mathbf{S}}_{\mathrm{GP}}(3,6)^\nu\rightarrow\overline{\mathrm{M}}_{0,6}$, which we describe.

\begin{theorem}[{Theorem~\ref{thm:birat-morph-to-M06bar-and-exc-loci}}]
There exists a birational morphism  $g\colon\overline{\mathbf{S}}_{\mathrm{GP}}(3,6)^\nu\rightarrow\overline{\mathrm{M}}_{0,6}$ extending the natural isomorphism on the interiors $\mathbf{S}(3,6)\cong\mathrm{M}_{0,6}$. The exceptional loci are the boundary strata of $\overline{\mathbf{S}}_{\mathrm{GP}}(3,6)^\nu$ of type \#12, \#20, \#23, \#20.2, \#20.6 in Table~\ref{tbl:GPdegenerationsself-dual-1}.
\end{theorem}

Thus, the compact moduli space $\overline{\mathbf{S}}_{\mathrm{GP}}(3,6)^\nu$ yields a birational modification of $\overline{\mathrm{M}}_{0,6}$ with geometric interpretation in terms of algebraic surfaces marked by a curve together with six marked~points.

\begin{remark}
We started from the point of view of compact moduli spaces of pointed algebraic varieties. From a different perspective, compact moduli spaces also arise by considering the extra datum of a divisor on the parametrized algebraic varieties: for instance, this is the approach taken by the Koll\'ar, Shepherd-Barron, Alexeev compactification \cite{KSB88,Ale96,Kol23}. The compact moduli space $\overline{\mathbf{S}}_{\mathrm{GP}}(3,6)^\nu$ provides an example that interpolates between the marked points and marked divisor perspectives, as the objects $(X,C;p_1,\ldots,p_6)$ parametrized by $\overline{\mathbf{S}}_{\mathrm{GP}}(3,6)$ realize a flag $p_1,\ldots,p_6\in C\subseteq X$. (We clarify that in our case $(X,C)$ is never KSBA-stable.)
\end{remark}

\subsection*{Acknowledgements}
We would like to thank Bernd Sturmfels for raising the question that motivated the writing of this paper.

\subsection*{Funding}
A. Geiger was funded by the Deutsche Forschungsgemeinschaft (DFG, German Research Foundation) – project number 539846931, within the SPP 2458~\emph{Combinatorial Synergies}. L. Schaffler was supported by the projects ``Programma per Giovani Ricercatori Rita Levi Montalcini'', PRIN2020KKWT53 ``Curves, Ricci flat Varieties and their Interactions'', and PRIN 2022 ``Moduli Spaces and Birational Geometry'' – CUP E53D23005790006. L. Schaffler is a member of the INdAM group GNSAGA and is supported by the INdAM - GNSAGA Project ``Birational geometry of compact moduli spaces via stable pairs and foliations'' – CUP E53C25002010001.

\section{Preliminaries}
\label{sec:preliminaries}

We recall the necessary background on Mustafin varieties, stable lattices, and the needed results from \cite{ST22}. At the same time, we expand more on some of these results as we need the details for our purposes. More precisely, in Section~\ref{sec:Mustafin-comp-for-5-points} we extrapolate from the proof of \cite[Theorem~7.9]{ST22} the family of $5$-pointed degenerate conics over $\overline{\mathbf{X}}_{\mathrm{GP}}(3,5)$.

\subsection{Mustafin varieties}

Let us fix $R$ to be the discrete valuation ring $\mathbb{C}[[t]]$. Denote by $K$ and by $\Bbbk$ the quotient field and the residue field of $R$, respectively. By a \emph{lattice} we mean a free $R$-submodule $L\subseteq K^3$ of rank $3$. For a lattice $L$, we define $\mathbb{P}(L)=\mathrm{Proj}(\mathrm{Sym}(L^\vee))$, with $L^\vee=\mathrm{Hom}_R(L,R)$. We have that $\mathbb{P}(L)\cong\mathbb{P}_R^2$, hence $\mathbb{P}(L)$ should be thought of as a $\mathbb{P}^2$-bundle over the curve $\mathrm{Spec}(R)$. We have a natural open embedding $\mathbb{P}_K^2\subseteq\mathbb{P}(L)$ that comes from the identification $\mathbb{P}_K^2\cong\mathbb{P}(L)_K$ induced by the isomorphism of $K$-algebras $K^3\cong L\otimes_RK$. 

For our purposes, we consider lattices up to rescaling by a nonzero element of $K$. We denote by $\mathfrak{B}_3^0$ the set of equivalence classes of lattices, which constitute the $0$-skeleton of the so-called \emph{Bruhat--Tits building}. By choosing finitely many lattice classes, it is possible to construct one-parameter families of degenerations of $\mathbb{P}^2$ as follows.

\begin{definition}
Let $\Sigma=\{[L_1],\ldots,[L_m]\}$ be a finite subset of lattice classes of $\mathfrak{B}_3^0$. Consider the diagonal embedding
\[
\mathbb{P}_K^2\subseteq\mathbb{P}(L_1)\times_R\ldots\times_R\mathbb{P}(L_m),
\]
which component-wisely is the natural open embedding $\mathbb{P}_K^2\subseteq\mathbb{P}(L_i)_K$, $i\in\{1,\ldots,m\}$. Then, the Zariski closure of $\mathbb{P}_K^2$ under the above embedding is denoted by $\mathbb{P}(\Sigma)$ and it is called the \emph{Mustafin variety} associated to $\Sigma$ (or the \emph{Mustafin join} of $[L_1],\ldots,[L_m]$). The construction of $\mathbb{P}(\Sigma)$ is independent of the choice of representatives of the lattice classes as the resulting Mustafin varieties will be isomorphic.
\end{definition}

\begin{remark}
Let us collect some of the basic properties of a Mustafin variety $\mathbb{P}(\Sigma)$ as recounted in \cite{CHSW11}. First of all, $\mathbb{P}(\Sigma)$ is integral, normal, Cohen--Macaulay, the family $\mathbb{P}(\Sigma)\rightarrow\mathrm{Spec}(R)$ is flat, and its generic fiber $\mathbb{P}(\Sigma)_K$ isomorphic to $\mathbb{P}_K^2$. The central fiber $\mathbb{P}(\Sigma)_\Bbbk$ can possibly be reducible (see the next example). At the same time, $\mathbb{P}(\Sigma)_\Bbbk$ is reduced, connected, Cohen--Macaulay, and its irreducible components are rational varieties.
\end{remark}

\begin{example}
\label{ex:Mustafin-variety-example}
Let $e_1,e_2,e_3$ be the standard basis of the $K$-vector space $K^3$. Consider the lattices
\[
L_1=e_1R+e_2R+e_3R,~L_2=te_1R+e_2R+e_3R.
\]
Let $\Sigma=\{[L_1],[L_2]\}$. Then, the central fiber $\mathbb{P}(\Sigma)_\Bbbk$ is isomorphic to the transverse gluing of a copy of $\mathbb{P}^2$ and a copy of the Hirzebruch surface $\mathbb{F}_1$, the blow up of $\mathbb{P}^2$ at one point. The gluing is along a line of $\mathbb{P}^2$ and the $(-1)$-curve of $\mathbb{F}_1$. For a detailed calculation explaining this, see \cite[Remark~5.1]{ST22}. In Figure~\ref{degeneration5points1} we give a schematic visualization of such degeneration. More in general, in Remark~\ref{rmk:pictures-explained} we include the dictionary needed to read such pictures for the cases of interest in this paper.
\end{example}

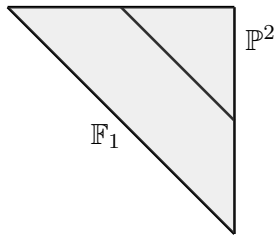
\begin{figure}[hbtp]
\centering
\begin{tikzpicture}[scale=0.5]
	\draw[line width=1pt] (0,6) -- (6,6);
	\draw[line width=1pt] (6,0) -- (6,6);
	\draw[line width=1pt] (0,6) -- (6,0);
	\draw[line width=1pt] (3,6) -- (6,3);
	\fill[gray!50,nearly transparent] (0,6) -- (6,6) -- (6,0) -- (0,6) -- cycle;

	\node at (2.6,2.6) {$\mathbb{F}_1$};

	\node at (6.7,5.2) {$\mathbb{P}^2$};
\end{tikzpicture}
\caption{Central fiber of the Mustafin join $\mathbb{P}(\Sigma)$ in Example~\ref{ex:Mustafin-variety-example}.}
\label{degeneration5points1}
\end{figure}

\subsection{Degenerations of \texorpdfstring{$n$}{n} points in \texorpdfstring{$\mathbb{P}^2$}{P2} from Mustafin varieties}
\label{sec:degen-P^2-n-points-via-Mustafin-var's}

We now use Mustafin varieties to construct specific degenerations of the projective plane with $n$ marked points. We will think of a one-parameter family of $n$-points in $\mathbb{P}^2$ as a collection $\mathbf{a}=(a_1,\ldots,a_n)$ where, for each $i\in\{1,\ldots,n\}$, $a_i$ is a morphism of $K$-schemes
\[
a_i\colon\mathrm{Spec}(K)\rightarrow\mathbb{P}_K^2.
\]
Using standard notation, we will simply write $\mathbf{a}\in(\mathbb{P}^2)^n(K)$. Given a one-parameter family $\mathbf{a}$ and a lattice class $[L]\in\mathfrak{B}_3^0$, we can view each $a_i$ also as a $K$-point of $\mathbb{P}(L)$ via the composition $\mathrm{Spec}(K)\xrightarrow{a_i}\mathbb{P}_K^2\subseteq\mathbb{P}(L)$. By the valuative criterion of properness, there exists a unique extension of $a_i$ as an $R$-point of $\mathbb{P}(L)$, which we denote by $\overline{a}_i^L$. Let $\overline{a}_i^L(0)$ be the image of the closed point of $\mathrm{Spec}(R)$ via $\overline{a}_i^L$ (note the slight abuse of notation here, as the maximal ideal of $R$ is generated by $t$, and it is not the zero ideal). So, the points $\overline{a}_1^L(0),\ldots,\overline{a}_n^L(0)\in\mathbb{P}(L)_\Bbbk$ are the limits of $a_1,\ldots,a_n$ in the central fiber of the family $\mathbb{P}(L)\rightarrow\mathrm{Spec}(R)$.

\begin{definition}
Let $\mathbf{a}=(a_1,\ldots,a_n)\in(\mathbb{P}^2)^n(K)$ be a one-parameter family of $n$ linearly general $K$-points of $\mathbb{P}^2$. Let $[L]\in\mathfrak{B}_3^0$ be arbitrary. We say that $[L]$ is \emph{stable} with respect to $\mathbf{a}$ provided there exist four points among $\overline{a}_1^L(0),\ldots,\overline{a}_n^L(0)$ which are in general linear position. If the four points are $\overline{a}_{i_1}^L(0),\ldots,\overline{a}_{i_4}^L(0)$, then we say that $[L]$ stabilizes the quadruple $(i_1,\ldots,i_4)$. We denote by $\Sigma_\mathbf{a}$ the set of stable lattice classes with respect to $\mathbf{a}$.
\end{definition}

As $\Sigma_\mathbf{a}$ is a finite set \cite{KT06}, we can consider the Mustafin variety $\mathbb{P}(\Sigma_\mathbf{a})$.

\begin{definition}
\label{def:n-pointed-Mustafin-degeneration}
Let $\mathbf{a}=(a_1,\ldots,a_n)\in(\mathbb{P}^2)^n(K)$ and consider the corresponding set $\Sigma_\mathbf{a}$ of stable lattice classes. Then, the $K$-points $a_1,\ldots,a_n\in\mathbb{P}(\Sigma_\mathbf{a})(K)$ uniquely extend to $R$-points of the Mustafin variety $\overline{a}_i\in\mathbb{P}(\Sigma_\mathbf{a})(R)$. Hence, the central fiber $\mathbb{P}(\Sigma_\mathbf{a})_\Bbbk$ comes equipped with the marked points $\overline{a}_1(0),\ldots,\overline{a}_n(0)$. By \cite[Lemma~6.1 and Theorem~6.9]{ST22}, we have that these are distinct smooth points of $\mathbb{P}(\Sigma_\mathbf{a})_\Bbbk$. We call $(\mathbb{P}(\Sigma_\mathbf{a})_\Bbbk;\overline{a}_1(0),\ldots,\overline{a}_n(0))$ an \emph{$n$-pointed Mustafin degeneration of $\mathbb{P}^2$}.
\end{definition}

\begin{example}
\label{ex:example-Mustafin-limit-5-points}
Let $\mathbf{a}=(a_1,\ldots,a_5) = ([1:0:0],[0:1:0],[0:0:1],[1:1:1],[t^2:1:t])$. Then, the stable lattice classes are given by
\[
\Sigma_\mathbf{a}=\{[e_1R+ e_2R+ e_3R], [t^2e_1R+ e_2R+ te_3R], [te_1R+ e_2R+ te_3R]\}.
\]
For instance, changing basis from $\{e_1,e_2,e_3\}$ to $\{te_1,e_2,te_3\}$ and after rescaling the projective coordinates, the arc becomes
\[
([1:0:0],[0:1:0],[0:0:1],[1:t:1],[t:1:1]).
\]
By setting $t=0$, we can see that in the limit has the points $1,2,4,5$ in general linear position. The other quadruples are linearly general in the limit after changing basis with respect to the other lattices. The central fiber $\mathbb{P}(\Sigma_\mathbf{a})_\Bbbk$ is pictured in Figure~\ref{fig:example-Mustafin-limit-5-points}. More details can be found in \cite[Example~6.2]{ST22}.
\end{example}

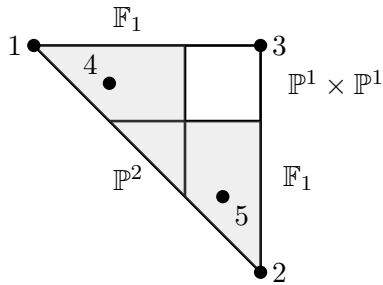
\begin{figure}[hbtp]
\centering
\begin{tikzpicture}[scale=0.5]
	\draw[line width=1pt] (0,6) -- (6,6);
	\draw[line width=1pt] (6,0) -- (6,6);
	\draw[line width=1pt] (0,6) -- (6,0);
	\draw[line width=1pt] (2,4) -- (6,4);
	\draw[line width=1pt] (4,6) -- (4,2);
    \fill[gray!50,nearly transparent] (0,6) -- (4,6) -- (4,4) -- (2,4) -- cycle;
    \fill[gray!50,nearly transparent] (6,0) -- (6,4) -- (4,4) -- (4,2) -- cycle;
    \fill[gray!50,nearly transparent] (2,4) -- (4,4) -- (4,2) -- cycle;

	\node at (2.5,6.7) {$\mathbb{F}_1$};

	\node at (7,2.5) {$\mathbb{F}_1$};

	\node at (8,5) {$\mathbb{P}^1\times\mathbb{P}^1$};
	
	\node at (2.5,2.5) {$\mathbb{P}^2$};

	\fill (0,6) circle (5pt);
	\fill (6,0) circle (5pt);
	\fill (6,6) circle (5pt);
	\fill (2,5) circle (5pt);
	\fill (5,2) circle (5pt);

	\node at (-0.5,6) {$1$};
	\node at (6.5,0) {$2$};
	\node at (6.5,6) {$3$};
	\node at (1.5,5.5) {$4$};
	\node at (5.5,1.5) {$5$};
\end{tikzpicture}
\caption{The $5$-pointed Mustafin degeneration of $\mathbb{P}^2$ in Example~\ref{ex:example-Mustafin-limit-5-points}.}
\label{fig:example-Mustafin-limit-5-points}
\end{figure}

The idea originally pursued by Gerritzen--Piwek in \cite{GP91} was to construct a compact moduli space parametrizing all possible $n$-pointed Mustafin degenerations $(\mathbb{P}(\Sigma_\mathbf{a});\overline{a}_1(0),\ldots,\overline{a}_n(0))$. This was carried out in \cite{ST22}. Next, we introduce this compact moduli space and review some of its main properties. In Section~\ref{sec:construction-of-the-global-family-of-degenerate-conics}, we recall the details of the construction, as they will be needed there.

\subsection{Moduli of points and lines in \texorpdfstring{$\mathbb{P}^2$}{P2}}
\label{sec:mod-pf-pts-and-lines-in-P2}

Let $n\geq5$ be an integer. Let $\mathbf{U}(3,n)\subseteq(\mathbb{P}^2)^n$ be the open subset corresponding to $n$-tuples of points in general linear position. Then, $\mathrm{SL}_3$ naturally acts on $\mathbf{U}(3,n)$ and the quotient
\[
\mathbf{B}(3,n):=\mathbf{U}(3,n)/\mathrm{SL}_3
\]
is the moduli space of isomorphism classes of $n$-tuples of linearly general points in $\mathbb{P}^2$.

By viewing points in $\mathbb{P}^2$ as lines in the dual projective plane $(\mathbb{P}^2)^\vee$, we have the following alternative construction of $\mathbf{B}(3,n)$. Consider the Grassmannian $\mathrm{Gr}(3,n)$ of $3$-dimensional vector subspaces $V$ in $\mathbb{C}^n$. In particular, given a point $[V]\in\mathrm{Gr}(3,n)$, we have a projective plane $\mathbb{P}(V)\subseteq\mathbb{P}(\mathbb{C}^n)=\mathbb{P}^{n-1}$, and the coordinate hyperplanes $H_i:=V(X_{i-1})\subseteq\mathbb{P}^{n-1}$ for $i=1,\ldots,n$, restrict to $\mathbb{P}(V)$ giving an arrangement of $n$ lines: $\sum_{i=1}^nH_i|_{\mathbb{P}(V)}$. Let $\mathrm{Gr}(3,n)^\circ\subseteq\mathrm{Gr}(3,n)$ be the open subset corresponding to points $[V]$ such that the line arrangement $H_1|_{\mathbb{P}(V)},\ldots,H_n|_{\mathbb{P}(V)}$ is linearly general. The torus $\mathbb{G}_m^n$ naturally acts on $\mathrm{Gr}(3,n)$ with the small diagonal $\mathrm{diag}(\mathbb{G}_m)\subseteq\mathbb{G}_m^n$ as stabilizer. Then, the quotient
\[
\mathbf{X}(3,n):=\mathrm{Gr}(3,n)^\circ/(\mathbb{G}_m^n/\mathrm{diag}(\mathbb{G}_m))=\mathrm{Gr}(3,n)^\circ/\mathbb{G}_m^{n-1}
\]
parametrizes isomorphism classes of $n$ lines in $(\mathbb{P}^2)^\vee$. Via the \emph{Gelfand--MacPherson correspondence} \cite{GM82}, we have an isomorphism
\[
\mathbf{B}(3,n)=\mathbf{U}(3,n)/\mathrm{SL}_3\cong\mathrm{Gr}(3,n)^\circ/\mathbb{G}_m^{n-1}=\mathbf{X}(3,n).
\]
This dual perspective will be useful in Section~\ref{sec:selfdualpoints} when we consider the sublocus of $\mathbf{B}(3,6)$ consisting of self-dual point configurations. From now on, we will identify $\mathbf{B}(3,n)$ and $\mathbf{X}(3,n)$ via the Gelfand--MacPherson correspondence, and always think of $\mathbf{X}(3,n)$ as a moduli space of points in $\mathbb{P}^2$ via this identification.

\begin{remark}
The $\mathrm{SL}_3$-orbit of a one-parameter family $\mathbf{a}=(a_1,\ldots,a_n) $ of $n$ linearly general points in $\mathbb{P}_K^2$ (i.e. $\mathbf{a}\in(\mathbb{P}^2)^n(K)$) is equivalent to considering a $K$-point of the moduli space $\mathbf{B}(3,n)$:
\[
\mathbf{a}\colon\mathrm{Spec}(K)\rightarrow\mathbf{B}(3,n),
\]
which we call an \emph{arc}. In the sequel, we will often use these two perspectives interchangeably.
\end{remark}

\subsection{Compact moduli of points in \texorpdfstring{$\mathbb{P}^2$}{P2}}

The compact moduli space parametrizing $n$-pointed Mustafin degenerations of $\mathbb{P}^2$ was constructed in \cite{ST22} building upon \cite{GP91} and the theory of multigraded Hilbert schemes \cite{HS04}. The next theorem summarizes its main properties.

\begin{theorem}[{\cite[Section~7]{ST22}}]
\label{thm:main-on-XbarGP(3,n)}
There exists a projective compactification $\mathbf{B}(3,n)\subseteq\overline{\mathbf{X}}_{\mathrm{GP}}(3,n)$ together with a proper flat family $\overline{\mathcal{M}}(3,n)\rightarrow\overline{\mathbf{X}}_{\mathrm{GP}}(3,n)$ with $n$ disjoint sections $s_1,\ldots,s_n$ satisfying the following property.
Let $x\in\overline{\mathbf{X}}_{\mathrm{GP}}(3,n)$ be a closed point. Consider an arc $\mathbf{a}\colon\mathrm{Spec}(K)\rightarrow\mathbf{B}(3,n)$  such that the unique extension of $\mathbf{a}$ as an $R$-point of $\overline{\mathbf{X}}_{\mathrm{GP}}(3,n)$ satisfies $x=\overline{\mathbf{a}}(0)$.
Let $\Sigma_\mathbf{a}$ be the set of stable lattices with respect to $\mathbf{a}$ and denote by $\mathbb{P}(\Sigma_\mathbf{a})$ their Mustafin join. 
 Then,
\begin{enumerate}

\item $\overline{\mathbf{a}}^*\overline{\mathcal{M}}(3,n)\cong\mathbb{P}(\Sigma_{\mathbf{a}})$, where $\overline{\mathbf{a}}^*\overline{\mathcal{M}}(3,n)$ denotes the fiber product of $\mathrm{Spec}(R)$ and $\overline{\mathcal{M}}(3,n)$ over $\overline{\mathbf{X}}_{\mathrm{GP}}(3,n)$;

\item Under the above isomorphism, the points $s_1(x),\ldots,s_n(x)$ in the central fiber of $\overline{\mathbf{a}}^*\overline{\mathcal{M}}(3,n)$ correspond to $\overline{a}_1(0),\ldots,\overline{a}_n(0)\in\mathbb{P}(\Sigma_{\mathbf{a}})_\Bbbk$, respectively.

\end{enumerate}
\end{theorem}

In \cite{ST22} is given a full description for $\overline{\mathbf{X}}_{\mathrm{GP}}(3,5)$ and $\overline{\mathbf{X}}_{\mathrm{GP}}(3,6)$ which we now recount.

\subsection{The compactification for \texorpdfstring{$5$}{5} points in \texorpdfstring{$\mathbb{P}^2$}{P2}}
\label{sec:Mustafin-comp-for-5-points}

For $n\geq3$, let $\overline{\mathrm{M}}_{0,n}$ be the Grothendieck--Knudsen moduli space of stable $n$-pointed rational curves. In particular, for the case $n=5$, the interiors of the moduli spaces $\overline{\mathbf{X}}_\mathrm{GP}(3,5)$ and $\overline{\mathrm{M}}_{0,5}$ are naturally identified:
\[
\overline{\mathbf{X}}_\mathrm{GP}(3,5)\supsetneq\mathbf{B}(3,5)\cong\mathbf{X}(3,5)\cong\mathbf{X}(2,5)=\mathrm{M}_{0,5}\subsetneq\overline{\mathrm{M}}_{0,5}.
\]

\begin{theorem}[{\cite[Theorem~7.9]{ST22}}]
\label{thm:XbarGP35-iso-M05bar}
The compactification $\overline{\mathbf{X}}_\mathrm{GP}(3,5)$ is isomorphic to $\overline{\mathrm{M}}_{0,5}$, and the isomorphism extends the identification of their interiors. The boundary of $\overline{\mathbf{X}}_\mathrm{GP}(3,5)$ consists of two different types of strata as described below.
\begin{itemize}

\item The $1$-dimensional boundary strata of $\overline{\mathbf{X}}_\mathrm{GP}(3,5)$ generically parametrize the $5$-pointed Mustafin degenerations shown on the left of Figure~\ref{fig:degenerations-for-5-points}. These are obtained by gluing transversely a copy of $\mathbb{P}^2$ and a copy of $\mathbb{F}_1$ along a line and the $(-1)$-curve, respectively. We call this type I. The five marked points are arranged as shown in the figure.

\item The $0$-dimensional strata of $\overline{\mathbf{X}}_\mathrm{GP}(3,5)$ parametrize the $5$-pointed Mustafin degenerations shown on the right of Figure~\ref{fig:degenerations-for-5-points}. These are obtained by gluing transversely a copy of $\mathbb{P}^2$, two copies of $\mathbb{F}_1$, and one copy of $\mathbb{P}^1\times\mathbb{P}^1$. The two $(-1)$-curves of $\mathbb{F}_1$ are glued to two distinct lines of $\mathbb{P}^2$, and two incident rulings of $\mathbb{P}^1\times\mathbb{P}^1$ are glued to one fiber of each copy of $\mathbb{F}_1$. We call this type II. The five marked points are arranged as shown in the figure.

\end{itemize}
\end{theorem}

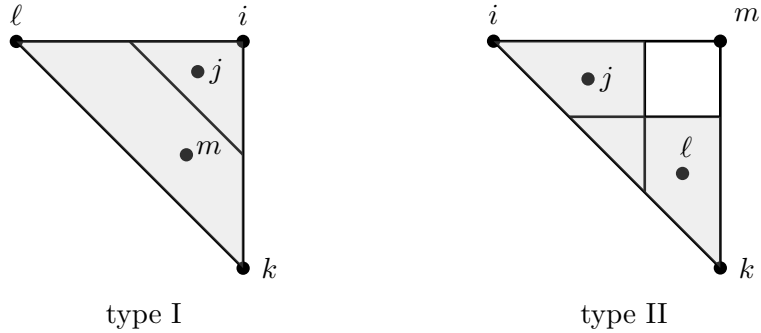
\begin{figure}[hbtp]
\centering
\begin{tabular}{>{\centering\arraybackslash}m{6cm}>{\centering\arraybackslash}m{6cm}}
\begin{tikzpicture}[scale=0.5]

	\draw[line width=1pt] (0,6) -- (6,6);
	\draw[line width=1pt] (6,0) -- (6,6);
	\draw[line width=1pt] (0,6) -- (6,0);
	\draw[line width=1pt] (3,6) -- (6,3);
	\fill (0,6) circle (5pt);
	\fill (6,0) circle (5pt);
	\fill (6,6) circle (5pt);
	\fill (4.5,3) circle (5pt);
	\fill (4.8,5.2) circle (5pt);
	\fill[gray!50,nearly transparent] (0,6) -- (6,6) -- (6,0) -- (0,6) -- cycle;

	\node at (6,6.7) {$i$};
	\node at (5.3,5.2) {$j$};
	\node at (6.7,0) {$k$};
    \node at (0,6.7) {$\ell$};
	\node at (5.1,3.2) {$m$};

\end{tikzpicture}
&
\begin{tikzpicture}[scale=0.5]

	\draw[line width=1pt] (0,6) -- (6,6);
	\draw[line width=1pt] (6,0) -- (6,6);
	\draw[line width=1pt] (0,6) -- (6,0);
	\draw[line width=1pt] (4,6) -- (4,2);
	\draw[line width=1pt] (6,4) -- (2,4);
	\fill (0,6) circle (5pt);
	\fill (2.5,5) circle (5pt);
	\fill (6,0) circle (5pt);
	\fill (5,2.5) circle (5pt);
	\fill (6,6) circle (5pt);
    \fill[gray!50,nearly transparent] (0,6) -- (4,6) -- (4,4) -- (2,4) -- cycle;
    \fill[gray!50,nearly transparent] (6,0) -- (6,4) -- (4,4) -- (4,2) -- cycle;
    \fill[gray!50,nearly transparent] (2,4) -- (4,4) -- (4,2) -- cycle;

	\node at (0,6.7) {$i$};
	\node at (3,5) {$j$};
	\node at (6.7,0) {$k$};
	\node at (5.1,3.2) {$\ell$};
	\node at (6.7,6.7) {$m$};

\end{tikzpicture}\\
type I & type II
\end{tabular}
\caption{$5$-pointed Mustafin degenerations parametrized by $\overline{\mathbf{X}}_{\mathrm{GP}}(3,5)$.}
\label{fig:degenerations-for-5-points}
\end{figure}

The proof of \cite[Theorem~7.9]{ST22} implicitly gives a flat family $\overline{\mathcal{D}}\rightarrow \overline{\mathbf{X}}_{\mathrm{GP}}(3,5)$ of $5$-pointed curves whose fibers over the open part $\mathbf{B}(3,5)$ correspond to the isomorphism classes of smooth plane conics passing through five marked points. The family is obtained as follows. First, we have a smooth conic bundle $\pi\colon\overline{\mathcal{C}}\rightarrow\overline{\mathrm{M}}_{0,5}$ which embeds via $-K_\pi$ into a $\mathbb{P}^2$ bundle thereby denoted $\overline{\mathcal{P}}'''$. Then, $\overline{\mathcal{P}}'''\rightarrow\overline{\mathrm{M}}_{0,5}$ is birationally modified to $\overline{\mathcal{P}}\rightarrow\overline{\mathrm{M}}_{0,5}$, which yields the family of Mustafin degenerations of $\mathbb{P}^2$ parametrized by $\overline{\mathbf{X}}_{\mathrm{GP}}(3,5)$ (see \cite[Figures~7 and 8]{ST22}). By tracing $\overline{\mathcal{C}}$ throughout the birational modification $\overline{\mathcal{P}}'''\dashrightarrow\overline{\mathcal{P}}$, we obtain the claimed family of $5$-pointed curves, which we denote by $\overline{\mathcal{D}}$. This set-up is summarized below:
\begin{center}
\begin{tikzpicture}[>=angle 90]
\matrix(a)[matrix of math nodes,
row sep=2em, column sep=2em,
text height=1.5ex, text depth=0.25ex]
{\overline{\mathcal{C}}&\overline{\mathcal{P}}'''&\overline{\mathcal{P}}&\overline{\mathcal{D}}\\
\overline{\mathrm{M}}_{0,5}&&\overline{\mathbf{X}}_{\mathrm{GP}}(3,5).&\\};
\path[right hook->] (a-1-1) edge node[]{}(a-1-2);
\path[->] (a-1-1) edge node[]{}(a-2-1);
\path[->] (a-1-2) edge node[]{}(a-2-1);
\path[dashed, ->] (a-1-2) edge node[]{}(a-1-3);
\path[->] (a-1-3) edge node[]{}(a-2-3);
\path[->] (a-2-1) edge node[above]{$\cong$}(a-2-3);
\path[left hook->] (a-1-4) edge node[]{}(a-1-3);
\path[->] (a-1-4) edge node[]{}(a-2-3);
\end{tikzpicture}
\end{center}
The fibers of $\overline{\mathcal{D}}\rightarrow\overline{\mathbf{X}}_{\mathrm{GP}}(3,5)$ over the $1$ and $0$-dimensional boundary strata are described in the next Propositions~\ref{prop:degenerate-5-pointed-conics-1-strata} and \ref{prop:degenerate-5-pointed-conics-0-strata} (again, this is just a summary of the construction in the proof of \cite[Theorem~7.9]{ST22}, so there is nothing to prove). As we will see, these fibers are reduced and $1$-dimensional. In particular, it follows that $\overline{\mathcal{D}}\rightarrow\overline{\mathbf{X}}_{\mathrm{GP}}(3,5)$ is flat by \cite[Theorem~10.12]{HKT09}.

\begin{proposition}
\label{prop:degenerate-5-pointed-conics-1-strata}
Let $x\in\overline{\mathbf{X}}_{\mathrm{GP}}(3,5)$ be a point parametrizing a $5$-pointed Mustafin degeneration of type I. Hence, it is given by $T\cup F$ where $T\cong\mathbb{P}^2$ and $F\cong\mathbb{F}_1$ (see Figure~\ref{fig:degenerations-for-5-points}). Then, the curve $\overline{\mathcal{D}}_x$ is described as follows:
\begin{enumerate}

\item On $T$, $\overline{\mathcal{D}}|_T$ is the unique line passing through $p_i$ and $p_j$;

\item On $F$, let us consider the blow up morphism $F\rightarrow\mathbb{P}^2$. Then, $\overline{\mathcal{D}}|_F$ is the strict transform of a conic passing through $p_k,p_\ell,p_m$ and the blown up point.

\end{enumerate}
\end{proposition}

\begin{proposition}
\label{prop:degenerate-5-pointed-conics-0-strata}
Let $x\in\overline{\mathbf{X}}_{\mathrm{GP}}(3,5)$ be a point parametrizing a $5$-pointed Mustafin degeneration of type II. Hence, it is given by $F_1\cup F_2\cup T\cup Z$ where $F_1,F_2\cong\mathbb{F}_1$, $T\cong\mathbb{P}^2$ and $Q\cong \PP^1\times\PP^1$ (see Figure~\ref{fig:degenerations-for-5-points}). Then, the curve $\overline{\mathcal{D}}_x$ is described as follows:
\begin{enumerate}

\item  On $F_h$, for $h = 1,2$, consider the blow up morphism $F_h\rightarrow\mathbb{P}^2$. Then $\overline{\mathcal{D}}|_{F_h}$ is the strict transform of the unique line passing through $i,j$, which does not pass through the blown up point;

\item On $T$, $\overline{\mathcal{D}}|_T$ is empty;

\item On $Q$, $\overline{\mathcal{D}}|_Q$ is the unique $(1,1)$-curve passing through $p_m$ and the two points given by the intersections of $\overline{\mathcal{D}}|_{F_1}$ and $\overline{\mathcal{D}}|_{F_2}$ with $Q$.

\end{enumerate}
\end{proposition}

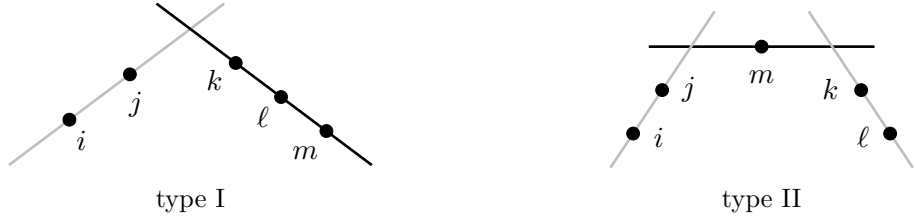
\begin{figure}[hbtp]
\centering
\begin{subfigure}{.45\textwidth}\centering
\begin{tikzpicture}[scale=0.8]

\draw [line width=1.pt,lightgray] (3.56,2.67)-- (0.,0.);
\draw [line width=1.pt] (2.44,2.67)-- (6.,0.);

\draw (1.2,0.4) node {$i$};
\draw (2.1,1.0) node {$j$};
\draw (3.4,1.4) node {$k$};
\draw (4.2,0.8) node {$\ell$};
\draw (4.9,0.2) node {$m$};

\draw [fill=black] (2.,1.5) circle (3pt);
\draw [fill=black] (1.,0.75) circle (3pt);
\draw [fill=black] (4.5,1.125) circle (3pt);
\draw [fill=black] (3.75,1.6875) circle (3pt);
\draw [fill=black] (5.25,0.5625) circle (3pt);

\end{tikzpicture}
    \caption*{type I}
\end{subfigure}
\begin{subfigure}{.45\textwidth}\centering
\begin{tikzpicture}[scale=0.8]

\draw [line width=1.pt] (8.633333333333335,2.)-- (12.36666666666668,2.);
\draw [line width=1.pt,lightgray] (11.278376529309,2.582435206036501)-- (13.,0.);
\draw [line width=1.pt,lightgray] (9.72162347069099,2.5824352060364837)-- (8.,0.);

\draw (8.8,0.5) node {$i$};
\draw (9.3,1.3) node {$j$};
\draw (11.64,1.3) node {$k$};
\draw (12.18,0.5) node {$\ell$};
\draw (10.5,1.5) node {$m$};

\draw [fill=black] (10.5,2.) circle (3pt);
\draw [fill=black] (8.375,0.5625) circle (3pt);
\draw [fill=black] (12.625,0.5625) circle (3pt);
\draw [fill=black] (8.854166666666668,1.28125) circle (3pt);
\draw [fill=black] (12.145833333333336,1.28125) circle (3pt);

\end{tikzpicture}
    \caption*{type II}
\end{subfigure}
\caption{The reducible curve $\overline{\mathcal{D}}_x$ inside the $5$-pointed Mustafin degeneration of $\mathbb{P}^2$. The coloring convention for the curve is explained in Remark~\ref{rem:coloring}.}
\label{fig:degenerate-conic-for-5-points}
\end{figure}

\begin{remark}
\label{rmk:typeI-intersection-fiber-strict-transform-conic}
In Proposition~\ref{prop:degenerate-5-pointed-conics-1-strata}, the curve $\overline{\mathcal{D}}|_F$ intersects each fiber $f\subseteq F\cong\mathbb{F}^1$ of the $\mathbb{P}^1$-fibration on $F$ in exactly one point, as the intersection number $f\cdot\overline{\mathcal{D}}|_F$ equals one.
\end{remark}

\begin{remark}
\label{rem:coloring}
We use the following coloring convention for drawing the curve $\overline{\mathcal{D}}_x$ in Propositions~\ref{prop:degenerate-5-pointed-conics-1-strata} and \ref{prop:degenerate-5-pointed-conics-0-strata}. Let $C$ be an irreducible component of $\overline{\mathcal{D}}_x$ and denote by $X$ the irreducible component of the Mustafin degeneration that contains it. We color $C$ in gray if $X\cong\mathbb{P}^2$ and $C$ is a line, or if $X$ is a blow up of $\mathbb{P}^2$ and $C$ is the strict transform of a general line. Otherwise, we color $C$ in black if $C$ is a smooth $(1,1)$-curve in $X\cong\PP^1\times\PP^1$, or if $X$ is a blow up of $\mathbb{P}^2$ and $C$ is the strict transform of a smooth conic. This convention will be used throughout the paper.
\end{remark}

\subsection{The compactification for \texorpdfstring{$6$}{6} points in \texorpdfstring{$\mathbb{P}^2$}{P2}}
\label{sec:refinedfan}

The description of $\overline{\mathbf{X}}_{\mathrm{GP}}(3,6)$ is more involved. An important insight about its geometry is provided by the theory of tropical compactifications: \cite{Tev07}, \cite[Section~6.4]{MS15}. If $X$ is a closed subvariety of a torus $T$, also called a \emph{very affine variety}, and $Y_\Sigma^T$ denotes the toric variety with dense open subtorus $T$ and fan $\Sigma$, then the closure $\overline{X}\subseteq Y_\Sigma^T$ is called a \emph{tropical compactification} if $\overline{X}$ proper and the multiplication map $T\times\overline{X}\rightarrow Y_\Sigma^T$ is faithfully flat (\cite[Definition~1.1]{Tev07}).

In our context, the very affine variety of interest is
\[
\mathbf{B}(3,6)\cong\mathbf{X}(3,6)\subseteq\mathbb{G}_m^{\binom{6}{3}-1}/\mathbb{G}_m^5=:T,
\]
and one of the main results in \cite{ST22} is that $\overline{\mathbf{X}}_{\mathrm{GP}}(3,6)^\nu$ is a tropical compactification. To understand this, the starting point is the description of the fan of the tropical Grassmannian $\trop(\Gr(3,6)^\circ)\subseteq\RR^{\binom{6}{3}}/\mathbf{1}\RR$, where $\mathbf{1}$ denotes the all ones vector. To be precise, we use the secondary fan structure  of $\trop(\Gr(3,6)^\circ)$ inherited from the Dressian $\Dr(3,6)$, which has the same support.
This fan is not pointed. Indeed, each cone contains the same $6$-dimensional vector space $L$, called \emph{lineality space} of the fan. This vector space is generated by the six vectors
\[
 \sum_{I\in\binom{[6]}{3},~i\in I} e_I,
\]
where $\binom{[6]}{3}=\{I \subseteq \{1,\ldots,6\} \,:\, |I| = 3 \}$ and $e_I$ with $I\in \binom{[6]}{3}$ denotes the standard basis vector of $\RR^{\binom{6}{3}}$ corresponding to the Plücker coordinate $p_I$. As $\mathbf{X}(3,6)\subseteq\mathbb{G}_m^{\binom{6}{3}-1}/\mathbb{G}_m^5$ and the maximal torus corresponds to the lineality space $L$ of the tropical Grassmannian, we define the fan $\Sigma(3,6):=\trop(\Gr^\circ(3,6))/L$, which is pointed. The number of its cones by increasing dimension is collected in its \emph{f-vector} $(1,65,535,1350,1005)$ \cite[Example~4.4.10]{SS04}. The $65$ rays can be distinguished into three types:
\begin{itemize}
    \item $20$ rays generated by the coordinate vectors of the form $e_{i_1i_2i_3}$ (type E),
    \item $15$ rays generated by $f_{i_1i_2i_3i_4} =e_{i_1i_2i_3} +e_{i_1i_2i_4} +e_{i_1i_3i_4} +e_{i_2i_3i_4}$ (type~F), and 
    \item $30$ rays generated by
$g_{i_1i_2i_3i_4i_5i_6} =e_{i_1i_2i_3} +e_{i_1i_2i_4} +e_{i_1i_3i_4} +e_{i_2i_3i_4} +e_{i_3i_4i_5} +e_{i_3i_4i_6}$ (type~G), where  $g_{i_1i_2 i_3 i_4 i_5 i_6
} =g_{i_3 i_4 i_5 i_6 i_1 i_2} =
g_{i_5 i_6 i_1i_2 i_3 i_4}\neq g_{i_1i_2i_5i_6i_3i_4}$. We point out that in this notation the a pairing $(i_1,i_2),(i_3,i_4),(i_5,i_6)$ is implicit.
\end{itemize} 
The $1005$ maximal cones come in $7$ flavors, described by the types of rays that span them: $30$ of type EEEE, $90$ of type EEFF(a), $90$ of type EEFF(b), $180$ of type EFFG, $240$ of type EEEG, $360$ of type EEFG, and $15$ of type FFFGG. The concrete combinatorics of their generating rays is listed in \cite[Section 5]{SS04}.

Luxton \cite{Lux08} showed that Kapranov's Chow quotient compactification $\overline{\mathbf{X}}(3,6)$  (\cite{Kap93}, \cite{HKT06}) is a tropical compactification of $\mathbf{X}(3,6)$ inside the toric variety $Y_{\Sigma(3,6)}^T$. This space has $15$ singular points corresponding to the non-simplicial maximal cones of $\Sigma(3,6)$ of type FFFGG. These are cones over the bipyramids given by a triangle of three vectors of type F with two tips that are vectors of type G: see the upper-left of Table~\ref{tab:refinementTropGr}.

\begin{definition}
\label{def:the-fan-hatSigma(3,6)}
Let $\widehat{\Sigma}(3,6)$ be the minimal refinement of $\Sigma(3,6)$ induced by splitting each cone of type FFFGG into $12$ subcones as illustrated in the upper left of Table~\ref{tab:refinementTropGr}. Let $\widehat{\mathbf{X}}(3,6)$ be the tropical compactification given by the closure of $\mathbf{X}(3,6)$ in $Y_{\widehat{\Sigma}(3,6)}$.
\end{definition}

\begin{theorem}[{\cite[Section~12]{ST22}}]

The following hold:

\begin{enumerate}

\item The normalization $\overline{\mathbf{X}}_\mathrm{GP}(3,6)^\nu$ is isomorphic to $\widehat{\mathbf{X}}(3,6)$. In particular, $\overline{\mathbf{X}}_\mathrm{GP}(3,6)^\nu$ is a tropical compactification.

\item The $6$-pointed Mustafin degenerations of $\mathbb{P}^2$ parametrized by $\overline{\mathbf{X}}_{\mathrm{GP}}(3,6)$ are listed in \cite[Tables~13--16]{ST22}.

\end{enumerate}
\end{theorem}

\begin{remark}
\label{rmk:pictures-explained}
We explain how to read the pictures of the $6$-pointed Mustafin degenerations in \cite[Tables~13--16]{ST22}. The different polygons correspond to irreducible toric surfaces glued transversely along boundary curves according to the following dictionary:
\begin{itemize}

\item A triangle means $\mathbb{P}^2$ and the sides are general lines;

\item A parallelogram means $\mathbb{P}^1\times\mathbb{P}^1$ and two opposite sides represent distinct fibers of the same $\mathbb{P}^1$-fibration;

\item A trapezoid means $\mathbb{F}_1$, where the two parallel edges represent the $(-1)$-curve (the short one) and the strict transform of a general line under the blow up $\mathbb{F}_1\rightarrow\mathbb{P}^1$. The remaining two sides correspond to two distinct fibers of the $\mathbb{P}^1$-fibration;

\item A pentagon means the blow up of $\mathbb{P}^2$ at two distinct points. The three $(-1)$-curves are represented by three consecutive sides (one has to determine which three of the five on a case by case basis);

\item A hexagon means the blow up of $\mathbb{P}^2$ at three linearly general points, where the six sides represent the six $(-1)$-curves.

\end{itemize}
The polygons are colored in gray or white to distinguish \emph{primary} (gray) and \emph{secondary components} (white). We review the meaning of this when needed at the beginning of Section~\ref{sec:one-parameter-Mustafin-degens-and-conics}.
\end{remark}

\section{Compact moduli of self-dual point configurations in \texorpdfstring{$\mathbb{P}^2$}{P2}}
\label{sec:selfdualpoints}

We now introduce the moduli space $\mathbf{S}(3,6)$ of configurations of six self-dual points in $\mathbb{P}^2$ and its compactification via Mustafin degenerations.

\subsection{Self-duality and the self-dual Grassmannian}

\begin{definition}
Let $r\geq2$. Let $p_1,\ldots,p_{2r}$ be a non-degenerate configuration of points in projective space $\mathbb{P}^{r-1}$. Let $A$ be the $r\times2r$ matrix whose $i$-th column is a choice of coordinates for the point $p_i$ within its projective equivalence class. Let $P\in\mathrm{Gr}(r,2r)$ be the point corresponding to the row-span of $A$. Let $[\ldots:P_I:\ldots]$, $I\in\binom{[2r]}{r}$, be the Pl\"ucker coordinates of $P$. Then, we say that the configuration $p_1,\ldots,p_{2n}$ is \emph{self-dual} provided there exist non-zero constants $\lambda_1,\ldots,\lambda_{2r}$ such that 
\[
P_I=\prod_{i\in I}\lambda_i\cdot\mathrm{sign}(I,I^c)\cdot P_{I^c},
\]
where $\mathrm{sign}(I,I^c)$ is the sign of the permutation in $S_{2r}$ that sends the sequence $(1,2,\ldots,2r)$ to the ordered sequence $(I,I^c)$. 
Such a point $P\in\Gr(r,2r)$ is called \emph{self-dual}.
This definition is independent from the choice of $A$ because for a different choice of $A$ we can rescale the constants $\lambda_1,\ldots,\lambda_{2r}$ appropriately.

Self-dual point configurations appeared in the literature already a hundred years ago as a special instance of associated points in work of Coble \cite{Cob22}. Many contributions have been made to their study, for example by Kapranov \cite{Kap93}, Dolgachev and Ortland \cite{DO88} or  Eisenbud and Popesku \cite{EP00}, among many others. A recent study of a generalization of self-duality can be found in \cite{GZ25}.

Following \cite{GHSV24}, the \emph{self-dual Grassmannian} $\SGr(r,2r)$ is defined as the Zariski closure of the set of self-dual points $P\in\Gr(r,2r)$. It is a rational variety of dimension $\binom{2r}{r}+2r-1$ \cite[Corollary 2.5]{GHSV24}. By the Gelfand--MacPherson correspondence it parametrizes self-dual point configurations in $\PP^{r-1}$. We denote by $\SGr(r,2r)^\circ$ the open subset of $\SGr(r,2r)$ where all Pl\"ucker coordinates are non-zero.
\end{definition}

\begin{definition}
We define the \emph{moduli space of configurations of $2r$ self-dual points in $\mathbb{P}^{r-1}$} as the quotient
\[
\mathbf{S}(r,2r)=\mathrm{SGr}(r,2r)^\circ/\mathbb{G}_m^{2r-1}\subseteq\mathbf{X}(r,2r).
\]
Recall that, although $\mathbf{S}(r,2r)\subseteq\mathbf{X}(r,2r)$, we think of it as a moduli space of points via the Gelfand--MacPherson correspondence (see Section~\ref{sec:mod-pf-pts-and-lines-in-P2}).
\end{definition}

From now on, we  consider the case $r=3$. In this case, six points in $\PP^2$ are self-dual if and only if they lie on a conic.  More precisely, a $6$-tuple $\mathbf{p}=(p_1,\ldots,p_6)$ of points in $\mathbb{P}^2$ lie on a conic if and only if their images under the degree two Veronese embedding $v\colon\mathbb{P}^2\hookrightarrow\mathbb{P}^5$ lie on a hyperplane. In other words, the determinant of the $6\times6$ matrix whose columns are $v(p_1),\ldots,v(p_6)$ vanishes. As this condition is $\mathrm{SL}_3$-invariant, by the first fundamental theorem on invariant theory, this determinant can be written as a polynomial in the maximal minors of the matrix $(p_1,\ldots,p_6)$ of coordinates of the points. After applying the straightening algorithm \cite{Stu08}, this simplifies to:
\begin{equation}
\label{eq:6-points-on-a-conic-equation}
\psi(\mathbf{p})=|p_1p_2p_3|\cdot|p_1p_4p_5|\cdot|p_3p_5p_6|\cdot|p_2p_4p_6| - |p_1p_2p_4|\cdot|p_1p_3p_5|\cdot|p_4p_5p_6|\cdot|p_2p_3p_6|.
\end{equation}

As self-dual point configurations can degenerate, $\mathbf{S}(3,6)$ is not compact. We will study the following compactification.

\begin{definition}
\label{def:compactification-for-moduli-of-6-self-dual-points-family-sections}
Let $\overline{\mathbf{S}}_{\mathrm{GP}}(3,6)$ be the Zariski closure of $\mathbf{S}(3,6)$ inside $\overline{\mathbf{X}}_{\mathrm{GP}}(3,6)$.
\end{definition}

\subsection{The normalization of \texorpdfstring{$\overline{\mathbf{S}}_{\mathrm{GP}}(3,6)$}{barSGP(3,6)} is tropical}
\label{sec:trop_comp_6pts}

We start by introducing the fan that will describe $\overline{\mathbf{S}}_{\mathrm{GP}}(3,6)^\nu$.

\begin{definition}
We denote by $\Sigma^{\rm sd}(3,6)$ the polyhedral fan arising from the tropicalized self-dual Grassmannian induced by the secondary fan structure of the tropical Grassmannian.
More explicitly, let $L^{\mathrm{sd}}\subseteq\mathbb{R}^{\binom{6}{3}}/L$
be the linear space spanned by the vectors of type $e_{i_1i_2i_3}+e_{i_4i_5i_6}$:
\[
L^{\mathrm{sd}}:=\raisebox{.2em}{$\Bigg(\mathrm{Span}_\mathbb{R}\bigg\{ e_I+e_{I^c} \;\bigg|\; I\in\binom{[6]}{3}\bigg\}+L\Bigg)$}\bigg/\raisebox{-.2em}{$L$.}
\]

Then, $\Sigma^{\rm sd}(3,6)= \{ \sigma \cap L^{\mathrm{sd}} \,|\, \sigma \in \Sigma(3,6) \}$ \cite[Theorem~6.5]{GHSV24}.
From \cite[Section~6]{GHSV24} we know that $\Sigma^{\rm sd}(3,6)$ has f-vector $(1,25,105,105)$. Its rays are generated by $e_{i_1i_2i_3}+e_{i_4i_5i_6}$ (type E$^{\rm sd}$) and $f_{i_1i_2i_3i_4}$ (type F). Its maximal cones come in $2$ flavors: $15$ of type FFF and $90$ of type E$^{\rm sd}$FF, which arise from the intersection of the maximal cones of $\Sigma(3,6)$ of type FFFGG and EEFF(a) with the linear space $L^{\rm sd}$.
\end{definition}

\begin{definition}
\label{def:sd_refined_fan}
We denote by $\widehat{\Sigma}^{\mathrm{sd}}(3,6)$ the polyhedral fan arising as the common refinement of $\Sigma^{\mathrm{sd}}(3,6)$ with  $\widehat{\Sigma}(3,6)$. 
Equivalently, the set of cones of $\widehat{\Sigma}^{\mathrm{sd}}(3,6)$ is given by
 \[
    \{ \sigma \cap \tau \,|\, \sigma \in \widehat{\Sigma}(3,6),\,\tau \in \Sigma^{\mathrm{sd}}(3,6) \}.
\]
\end{definition}

We now connect $\widehat{\Sigma}^{\mathrm{sd}}(3,6)$ to the moduli space $\overline{\mathbf{S}}_{\mathrm{GP}}(3,6)$.

\begin{definition}
Let $L_{\mathbb{Z}}^{\mathrm{sd}}$ be the $\mathbb{Z}$-span of the generators of $L^{\mathrm{sd}}$ and let $H$ be the torus $L_{\mathbb{Z}}^{\mathrm{sd}}\otimes\mathbb{G}_m$.
Let $Y_{\widehat{\Sigma}^{\mathrm{sd}}(3,6)}^H$ be the toric variety with dense open subtorus $H$ and fan $\widehat{\Sigma}^{\mathrm{sd}}(3,6)$. Then, we define $\widehat{\mathbf{S}}(3,6)$ as the closure of $\mathbf{S}(3,6)$ in $Y_{\widehat{\Sigma}^{\mathrm{sd}}(3,6)}^H$.

Recall from Section~\ref{sec:refinedfan} the torus $T=\mathbb{G}_m^{\binom{6}{3}-1}/\mathbb{G}_m^5$ such that the tropical compactification $\widehat{\mathbf{X}}(3,6)$ is a closed subset of the toric variety $Y_{\widehat{\Sigma}(3,6)}^T$. The inclusion of lattices $L_{\mathbb{Z}}^{\mathrm{sd}}\hookrightarrow N_T=\mathrm{Hom}(\mathbb{G}_m,T)$ realizes $H$ as a subtorus of $T$ and induces an inclusion of toric varieties $Y_{\widehat{\Sigma}^{\mathrm{sd}}(3,6)}^H\hookrightarrow Y_{\widehat{\Sigma}(3,6)}^T$. In particular, we have the following commutative diagram:
\begin{center}
\begin{tikzpicture}[>=angle 90]
\matrix(a)[matrix of math nodes,
row sep=2em, column sep=2em,
text height=2ex, text depth=1ex]
{\mathbf{S}(3,6)&\widehat{\mathbf{S}}(3,6)&Y_{\widehat{\Sigma}^{\mathrm{sd}}(3,6)}^H\\
\mathbf{X}(3,6)&\widehat{\mathbf{X}}(3,6)&Y_{\widehat{\Sigma}(3,6)}^T.\\};
\path[right hook->] (a-1-1) edge node[]{}(a-1-2);
\path[right hook->] (a-1-1) edge node[]{}(a-2-1);
\path[right hook->] (a-1-2) edge node[]{}(a-1-3);
\path[right hook->] (a-1-2) edge node[]{}(a-2-2);
\path[right hook->] (a-1-3) edge node[]{}(a-2-3);
\path[right hook->] (a-2-1) edge node[]{}(a-2-2);
\path[right hook->] (a-2-2) edge node[]{}(a-2-3);
\end{tikzpicture}
\end{center}
\end{definition}

\begin{proposition}
\label{prop:widehatS(3,6)-is-normal-and-tropical}
$\widehat{\mathbf{S}}(3,6)$ is a normal tropical compactification of $\mathbf{S}(3,6)$.
\end{proposition}

\begin{proof}
As in \cite{GHSV24}, let us consider the following monomial map in the Pl\"ucker coordinates, which was studied by Pachter--Speyer (see \cite[Section~3]{PS04}):
\begin{align*}
\mathrm{Gr}(2,6)&\dashrightarrow\mathrm{Gr}(3,6)\\
[\ldots:p_{ij}:\ldots]&\mapsto[\ldots:p_{ijk}:\ldots],
\end{align*}
where $p_{ijk}=p_{ij}p_{ik}p_{jk}$. This restricts to a morphism $\mathrm{Gr}(2,6)^\circ\rightarrow\mathrm{Gr}(3,6)^\circ$, which is birational onto the self-dual Grassmannian $\mathrm{SGr}(3,6)^\circ$ and descends to an isomorphism
\[
\mathbf{X}(2,6)=\mathrm{Gr}(2,6)^\circ/(\mathbb{G}_m^6/\mathrm{diag}(\mathbb{G}_m))\xrightarrow{\cong}\mathrm{Gr}(3,6)^\circ/(\mathbb{G}_m^6/\mathrm{diag}(\mathbb{G}_m))=\mathbf{S}(3,6).
\]
By \cite[Theorem~1.11]{HKT09}, we have that $\mathbf{X}(2,6)\cong\mathrm{M}_{0,6}$ is H\"ubsch, and hence also Sch\"on (see \cite[Definition~1.9]{HKT09}). It follows that the very affine variety $\mathbf{S}(3,6)$ is Sch\"on.

Now, let us consider the tropical compactification $\overline{\mathbf{X}}(2,6)$ associated to the fan $\mathrm{trop}(\mathrm{Gr}(2,6))$. By \cite[Theorem~6.5]{GHSV24} we have that $\mathrm{trop}(\mathrm{Gr}(2,6))\cong\Sigma^{\mathrm{sd}}(3,6)$. Hence, $\Sigma^{\mathrm{sd}}(3,6)$ induces a tropical compactification of $\mathbf{S}(3,6)$.

Since the fan $\widehat{\Sigma}^{\mathrm{sd}}(3,6)$ refines $\Sigma^{\mathrm{sd}}(3,6)$, we have that $\widehat{\mathbf{S}}(3,6)$ is tropical by \cite[Proposition~2.5]{Tev07} and normal by \cite[Theorem~1.4]{Tev07}.
\end{proof}

\begin{theorem}
\label{thm:SGP36}
The normalization $\overline{\mathbf{S}}_\mathrm{GP}(3,6)^\nu$ is isomorphic to $\widehat{\mathbf{S}}(3,6)$. In particular, $\overline{\mathbf{S}}_\mathrm{GP}(3,6)^\nu$ is a tropical compactification.
\end{theorem}

\begin{proof}
From \cite[Proof of Theorem~1.7]{ST22} we have a finite birational morphism $\varphi\colon\widehat{\mathbf{X}}(3,6)\rightarrow\overline{\mathbf{X}}_{\rm GP}(3,6)$ that is the identity on the common open subset $\mathbf{X}(3,6)$. Hence, $\varphi$ is also the identity when restricted to $\mathbf{S}(3,6)\subseteq\mathbf{X}(3,6)$. We have that the compactifications $\widehat{\mathbf{S}}(3,6)$ and $\overline{\mathbf{S}}_{\mathrm{GP}}(3,6)$ arise as the closure of $\mathbf{S}(3,6)$ inside $\widehat{\mathbf{X}}(3,6)$ and $\overline{\mathbf{X}}_{\rm GP}(3,6)$, respectively. Therefore, by continuity, $\varphi(\widehat{\mathbf{S}}(3,6))\subseteq\overline{\mathbf{S}}_{\mathrm{GP}}(3,6)$. Hence, $\varphi|_{\widehat{\mathbf{S}}(3,6)}\colon\widehat{\mathbf{S}}(3,6)\rightarrow\overline{\mathbf{S}}_{\mathrm{GP}}(3,6)$ is a finite birational morphism, which induces an isomorphism between the respective normalizations by Zariski's Main Theorem. The conclusion follows as $\widehat{\mathbf{S}}(3,6)$ is normal and tropical by Proposition~\ref{prop:widehatS(3,6)-is-normal-and-tropical}.
\end{proof}

Next, we study the fan $\widehat{\Sigma}^{\mathrm{sd}}(3,6)$. More precisely, we compute its f-vector and describe the bijection between $\widehat{\Sigma}^{\mathrm{sd}}(3,6)$ and the boundary strata of $\overline{\mathbf{S}}_{\mathrm{GP}}(3,6)^\nu$. To accomplish this, we first need to better understand the combinatorics of $\widehat{\Sigma}(3,6)$.

\subsection{The combinatorics of the fan \texorpdfstring{$\widehat{\Sigma}(3,6)$}{hatSigma(3,6)}}

The refinement $\widehat{\Sigma}(3,6)\preccurlyeq\Sigma(3,6)$ splits the $15$ bipyramid cones into $12\cdot 15$ simplicial cones over tetrahedra. To do this formally, we introduce two new types of rays.

\begin{definition}
In $\RR^{\binom{6}{3}}/L$, let us define the following vectors:
\begin{itemize}
\item type C:  $c_{i_3i_4i_5i_6}=f_{i_1i_2i_3i_4}+f_{i_1i_2i_5i_6}$,
\item type D: $d_{i_3i_4i_5i_6} = f_{i_1i_2i_3i_4}+f_{i_1i_2i_5i_6}+f_{i_3i_4i_5i_6} = g_{i_1i_2i_3i_4i_5i_6}+g_{i_1i_2i_5i_6i_3i_4}$.
\end{itemize}
For the above vectors, the pairing $(i_3,i_4),(i_5,i_6)$ is implicit.
\end{definition}

We now explain geometrically how the rays generated by these new vectors relate to the fan $\widehat{\Sigma}(3,6)$. A ray of type C is generated by the midpoint of an edges of the bipyramid FFFGG connecting two vertices of type F. A ray type D is generated by the midpoint of the line segment connecting the two vertices of type G of one bipyramid. Equivalently, it is generated by the barycenter of the triangle with vertices of type F.
Table~\ref{tab:refinementTropGr} illustrates the refinement of the bipyramid into $12$ tetrahedra and the role of the vectors of type C and D within it. In order to obtain the minimal refined fan, some neighboring cones over tetrahedra are also affected, which is also shown in Table~\ref{tab:refinementTropGr}.

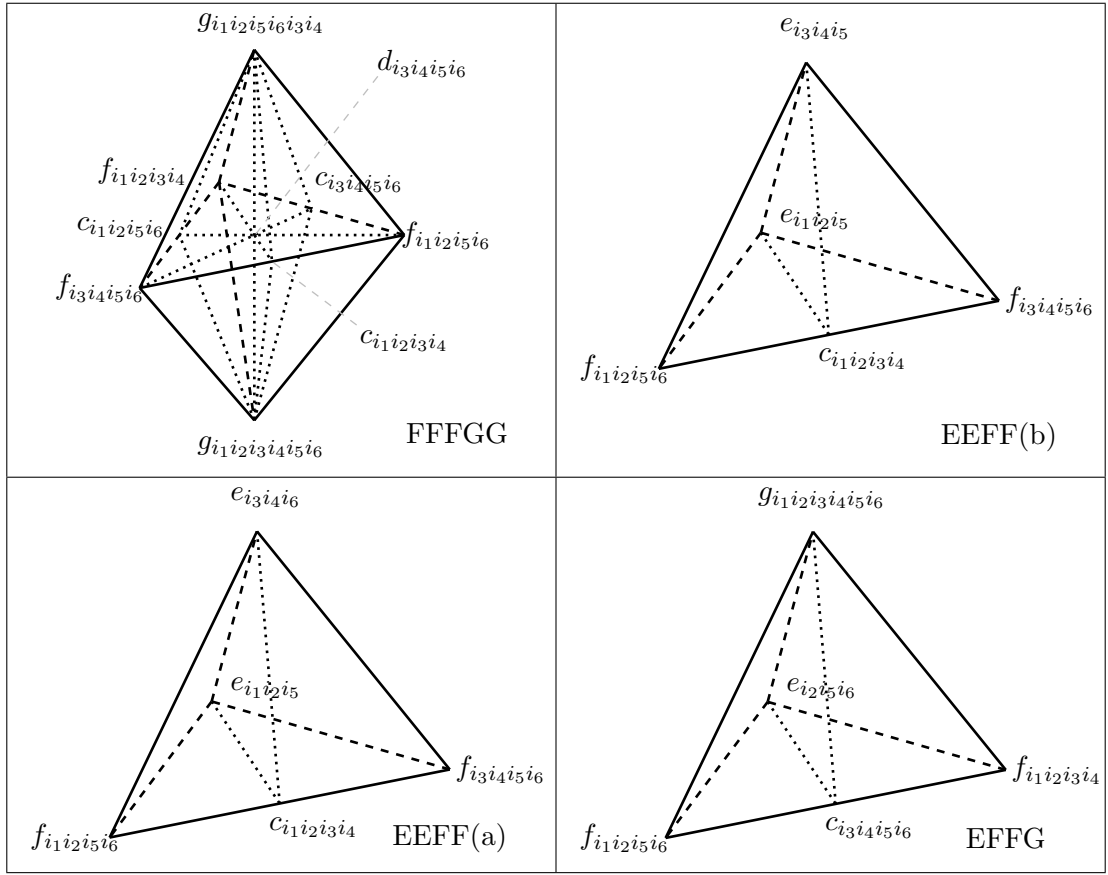
\begin{table}[hbtp]
\centering
\caption{The refinement $\widehat{\Sigma}(3,6)$ of the tropical Grassmannian.}
\label{tab:refinementTropGr}
\label{subdivisionofthebipyramid}
			\begin{tabular}{|>{\centering\arraybackslash}m{6.9cm}|>{\centering\arraybackslash}m{6.9cm}|}
				\hline
				\begin{tikzpicture}[scale=0.7]
					
					\draw[line width=1pt] (0,0) -- (5,1);
					\draw[dashed,line width=1pt] (5,1) -- (1.5,2);
					\draw[dashed,line width=1pt] (1.5,2) -- (0,0);
					\draw[dashed,line width=1pt] (1.5,2) -- (6.5/3,4.5);
					\draw[line width=1pt] (0,0) -- (6.5/3,4.5);
					\draw[line width=1pt] (5,1) -- (6.5/3,4.5);
					\draw[dashed,line width=1pt] (1.5,2) -- (6.5/3,-2.5);
					\draw[line width=1pt] (0,0) -- (6.5/3,-2.5);
					\draw[line width=1pt] (5,1) -- (6.5/3,-2.5);
					
					\draw[dotted,line width=1pt] (6.5/3,4.5) -- (6.5/3,-2.5);
					\draw[dotted,line width=1pt] (0,0) -- (6.5/2,3/2);
					\draw[dotted,line width=1pt] (1.5,2) -- (5/2,1/2);
					\draw[dotted,line width=1pt] (5,1) -- (1.5/2,2/2);
					
					\draw[dotted,line width=1pt] (6.5/3,4.5) -- (6.5/2,3/2);
					\draw[dotted,line width=1pt] (6.5/3,4.5) -- (5/2,1/2);
					\draw[dotted,line width=1pt] (6.5/3,4.5) -- (1.5/2,2/2);
					
					\draw[dotted,line width=1pt] (6.5/3,-2.5) -- (6.5/2,3/2);
					\draw[dotted,line width=1pt] (6.5/3,-2.5) -- (5/2,1/2);
					\draw[dotted,line width=1pt] (6.5/3,-2.5) -- (1.5/2,2/2);
					
					\node at (2.3,5) {$g_{i_1i_2i_5i_6i_3i_4}$};
					\node at (2.3,-3) {$g_{i_1i_2i_3i_4i_5i_6}$};
					\node at (-0.75,0) {$f_{i_3i_4i_5i_6}$};
					\node at (5.8,1) {$f_{i_1i_2i_5i_6}$};
					\node at (0.05,2.2) {$f_{i_1i_2i_3i_4}$};
					
					\node at (5.0,-1.0) {$c_{i_1i_2i_3i_4}$};
					\draw[lightgray, dashed,line width=0.5pt] (4.1,-0.7) -- (2.5,0.5);
					\node at (4.15,2) {$c_{i_3i_4i_5i_6}$};
					\node at (-0.35,1.15) {$c_{i_1i_2i_5i_6}$};
					
					\node at (5.35,4.25) {$d_{i_3i_4i_5i_6}$};
					\draw[lightgray, dashed,line width=0.5pt] (4.5,4) -- (2.18,1);
					
					\node at (6,-2.75) {FFFGG};
				\end{tikzpicture}
				&
				\begin{tikzpicture}[scale=0.9]
					
					\draw[line width=1pt] (0,0) -- (5,1);
					\draw[dashed,line width=1pt] (5,1) -- (1.5,2);
					\draw[dashed,line width=1pt] (1.5,2) -- (0,0);
					\draw[dashed,line width=1pt] (1.5,2) -- (6.5/3,4.5);
					\draw[line width=1pt] (0,0) -- (6.5/3,4.5);
					\draw[line width=1pt] (5,1) -- (6.5/3,4.5);
					
					\draw[dotted,line width=1pt] (1.5,2) -- (5/2,1/2);
					
					\draw[dotted,line width=1pt] (6.5/3,4.5) -- (5/2,1/2);
					
					\node at (2.3,5) {$e_{i_3i_4i_5}$ };
					\node at (2.3,2.2) {$e_{i_1i_2i_5}$};
					\node at (5.75,1) {$f_{i_3i_4i_5i_6}$};
					\node at (-0.5,0) {$f_{i_1i_2i_5i_6}$};
					\node at (3,0.15) {$c_{i_1i_2i_3i_4}$};
					
					\node at (5,-1) {EEFF(b)};
				\end{tikzpicture}
				\\
				\hline
				\begin{tikzpicture}[scale=0.9]
					
					\draw[line width=1pt] (0,0) -- (5,1);
					\draw[dashed,line width=1pt] (5,1) -- (1.5,2);
					\draw[dashed,line width=1pt] (1.5,2) -- (0,0);
					\draw[dashed,line width=1pt] (1.5,2) -- (6.5/3,4.5);
					\draw[line width=1pt] (0,0) -- (6.5/3,4.5);
					\draw[line width=1pt] (5,1) -- (6.5/3,4.5);
					
					\draw[dotted,line width=1pt] (1.5,2) -- (5/2,1/2);
					
					\draw[dotted,line width=1pt] (6.5/3,4.5) -- (5/2,1/2);
					
					\node at (2.3,5) {$e_{i_3i_4i_6}$};
					\node at (2.3,2.2) {$e_{i_1i_2i_5}$};
					\node at (5.75,1) {$f_{i_3i_4i_5i_6}$};
					\node at (-0.5,0) {$f_{i_1i_2i_5i_6}$};
					\node at (3,0.15) {$c_{i_1i_2i_3i_4}$};
					
					\node at (5,0) {EEFF(a)};
				\end{tikzpicture}
				&
				\begin{tikzpicture}[scale=0.9]
					
					\draw[line width=1pt] (0,0) -- (5,1);
					\draw[dashed,line width=1pt] (5,1) -- (1.5,2);
					\draw[dashed,line width=1pt] (1.5,2) -- (0,0);
					\draw[dashed,line width=1pt] (1.5,2) -- (6.5/3,4.5);
					\draw[line width=1pt] (0,0) -- (6.5/3,4.5);
					\draw[line width=1pt] (5,1) -- (6.5/3,4.5);
					
					\draw[dotted,line width=1pt] (1.5,2) -- (5/2,1/2);
					
					\draw[dotted,line width=1pt] (6.5/3,4.5) -- (5/2,1/2);
					
					\node at (2.3,5) {$g_{i_1i_2i_3i_4i_5i_6}$};
					\node at (5.75,1) {$f_{i_1i_2i_3i_4}$};
					\node at (-0.6,0) {$f_{i_1i_2i_5i_6}$};
					\node at (2.3,2.2) {$e_{i_2i_5i_6}$};
					\node at (3,0.15) {$c_{i_3i_4i_5i_6}$};
					\node at (5,0) {EFFG};
				\end{tikzpicture}
				\\
				\hline
			\end{tabular}
\end{table}

\begin{theorem}
\label{thm:f-vector-of-widehat-Sigma(3,6)}
The refined fan $\widehat{\Sigma}(3,6)$ has f-vector $(1,125,970,2250,1530)$. The number of cones of a given type and the correspondence between the strata of $\overline{\mathbf{X}}_\mathrm{GP}(3,6)^\nu$ and the cones of $\widehat{\Sigma}(3,6)$ can be found in Table~\ref{tab:correspondencekapranov'sstrataandcones_refinement}.
\end{theorem}

\begin{remark}
To compute the f-vector of $\widehat{\Sigma}^{\mathrm{sd}}(3,6)$ and its combinatorial types of cones, we do not need the full knowledge of the f-vector and cone types of $\widehat{\Sigma}(3,6)$. However, as the latter was not computed in \cite{ST22}, we take the chance to do it here.
\end{remark}

\begin{proof}[Proof of Theorem~\ref{thm:f-vector-of-widehat-Sigma(3,6)}]
As we observe from Table~\ref{tab:refinementTropGr}, we have two additional types of rays (Types C and D) to the $65$ rays of $\Sigma(3,6).$ As there are $15$ bipyramids of type FFFGG, there are $15$ rays of type D.
To count the rays of type C, we need to understand how the bipyramids FFFGG intersect each other. Assume two bipyramids share at least an edge, then this edge is either of type FF or of type FG. 
Given a vertex G there is exactly one other vertex G, such that the pair belongs to the same bipyramid FFFGG. 
Given an edge of type FF that is part of a bipyramid, e.g. $f_{1234}, f_{1256}$, then there is a unique third vertex of type F ($f_{3456}$) that will complete the triangular non-facet of a bipyramid. 
This follows from the combinatorics of the rays spanning the cone, as indicated in Table~\ref{tab:refinementTropGr}. Therefore, no two bipyramids can share an edge. Hence, the bipyramids intersect at most in vertices.
It follows that we have $3\cdot 15$ rays of type C. 
These considerations give the dimension $1$ entry of the f-vector: $65+15+3\cdot15=125$ rays.

In \cite[Section~5]{SS04} there is a list of all $2$-dimensional cones of the tropical Grassmannian $\trop(\Gr^\circ(3,6))$ with the Gr\"obner fan structure. This list contains $2$-dimensional cones of type GG, which do not appear in the secondary fan structure used in this paper, and which, therefore, need to be disregarded.
The $535$ $2$-dimensional cones of $\Sigma(3,6)$ come in five flavors: EE, FF, EF, EG, and FG. When considering the refinement as depicted in Table~\ref{tab:refinementTropGr}, we observe that the edges of type FF are split into two edges of type CF. There were $45$ edges of type FF before, which now give rise to $90$ edges of type CF. The refinement further adds new $2$-dimensional cones, which are not splits of existing $2$-dimensional cones. The new ones have types CD, CE, CG, DF, and DG. They are illustrated in Table~\ref{tab:refinementTropGr}. All $2$-dimensional cones containing a ray of type D are contained in the $15$ maximal cones of type FFFGG: There are $45 = 3\cdot 15 $ of type CD, $45 = 3\cdot 15 $ of type DF, and $30 = 2\cdot 15 $ of type DG. The cones of type CG are contained in faces of the cones of type FFFGG, which only intersect each other in rays of type F. Thus, we count $90 = 2\cdot 3\cdot 15$ cones of type CG. To obtain the count of the cones of type CE, we show that for a given ray of type C there exist four different rays of type E that form a cone with C. 
We know there are $45$ FF edges, which contain the $45$ type C rays. These edges are shared between $4$ types of cones: FFFGG (which do not contain E's, and therefore can be disregarded), EFFG, EEFF(a), and EEFF(b). Now, given one edge of type FF (e.g. $f_{1234}, f_{3456}$), one can build two maximal cones of type EEFF(a) using the pair $(e_{123},e_{456})$ or $(e_{124},e_{356})$,
two of type EEFF(b) using the pairs $(e_{123}, e_{356})$ or $(e_{124}, e_{456})$,
four of type EFFG using $g_{123456}$ with one of the following: $e_{123}$, $e_{456}$, $e_{124}$, and $e_{356}$. Thus, there are exactly four rays of type E ($e_{456}, e_{356}, e_{123}$, and $e_{124}$) that form a triangle with the given edge of type FF. So, there are $4\cdot 45 = 180$ cones of type CE.
In total, we obtain $535-45+90+45+45+30+90+180 = 970$ cones of dimension $2$.

Next, we compute the dimension $4$ entry of the f-vector. When refining $\Sigma(3,6)$, each of the $15$ cones of type FFFGG is split into $12$ cones of type CDFG. Additionally, the neighboring cones that intersect in the split faces of the bipyramid, which are of type EEFF(a), EEFF(b), and EFFG ($360$ in total) are each split into two. These splittings are shown in Table~\ref{tab:refinementTropGr}. The other maximal cones are not affected. In total, this gives $1005-15-360+15\cdot12+360\cdot2=1530$ maximal cones for $\widehat{\Sigma}(3,6)$.

Lastly, to prove the count of the $3$-dimensional cones in $\widehat{\Sigma}(3,6)$, we preliminarily need to count those of the $1350$ cones of dimension three in $\Sigma(3,6)$ that are affected by the refinement. These are only the cones of type FFG and type EFF. To obtain their number, we start from the cones of type FFFGG, which only intersect each other in rays of type F. There are $15$ cones of type FFFGG, each with six facets of type FFG (and no other facets). So, there are at least $90$ FFG cones. The only other type of maximal cones that have FFG facets are the EFFG cones. Since each FFG facet of an FFFGG bipyramid is shared by exactly two tetrahedra of type EFFG, and considering that there are exactly $180$ EFFG cones, we can conclude that there are exactly $90$ FFG cones. Next, let us consider the EFF cones. As different EFFG cones do not share the same EFF facet, we obtain at least $180$ cones of type EFF. Finally, no new EFF cones arise by considering the maximal cones of type EEFF(a)/(b), so we have exactly $180$ EFF cones.

We are now ready to compute the number of $3$-dimensional cones in $\widehat{\Sigma}(3,6)$. Consider Table~\ref{tab:refinementTropGr}. Each facet of type FFG is split into two facets of type CFG, each facet of type EFF is split into two of type CEF. Additional facets arising from the refinement are of type CDF, CDG, DFG, CEG, CEE(a), and CEE(b). There are $15$ rays of type D (one in each tetrahedron FFFGG), and for each one there are six choices of triangles CDF $6\cdot15 = 90$. The same count applies to the cones of type CDG and DFG. There is one facet of type CEG (resp. CEE(a), CEE(b)) for each cone of type EFFG (resp. EEFF(a), EEFF(b)), which yields $180$ (resp. $90$, $90$) $3$-dimensional cones. In conclusion, we obtain the following count for the $3$-dimensional cones of $\widehat{\Sigma}(3,6)$:
\begin{align*}
1350&-90\,\text{\scriptsize{\rm{(FFG)}}\normalsize}-180\,\text{\scriptsize{\rm{(EFF)}}\normalsize}+180\,\text{\scriptsize{\rm{(CFG)}}\normalsize}+360\,\text{\scriptsize{\rm{(CEF)}}\normalsize}+ 90\,\text{\scriptsize{\rm{(CDF)}}\normalsize}\\& +90\,\text{\scriptsize{\rm{(CDG)}}\normalsize}+90\,\text{\scriptsize{\rm{(DFG)}}\normalsize}+180\,\text{\scriptsize{\rm{(CEG)}}\normalsize}+90\,\text{\scriptsize{\rm{(CEE(a))}}\normalsize}+90\,\text{\scriptsize{\rm{(CEE(b))}}\normalsize} = 2250.
\end{align*} 
All types and numbers of the different cones of the fan $\widehat{\Sigma}(3,6)$ are listed in Table~\ref{tab:correspondencekapranov'sstrataandcones_refinement}.

Note that the count for the $3$-dimensional cones of types EEF(a)/(b), EFG, and EEG does not follow from the above proof and can also be not directly found in \cite[Section~5]{SS04}, but needs to be deduced from the count and combinatorics of the maximal cones as well as their intersection behavior.
\end{proof}

\begin{center}
\begin{longtable}{|>{\centering\arraybackslash}m{4.0cm}|>{\centering\arraybackslash}m{2.3cm}|>{\centering\arraybackslash}m{5.5cm}|>{\centering\arraybackslash}m{1cm}|}
\caption{Correspondence between strata of $\overline{\mathbf{X}}_{\mathrm{GP}}(3,6)^\nu$ and $\widehat{\Sigma}(3,6)$. The extra decoration (i)/(ii) of some of the combinatorial types is explained in Remark~\ref{rem:notation}.}
\label{tab:correspondencekapranov'sstrataandcones_refinement}
\\
\hline
stratum of $\overline{\mathbf{X}}_{\mathrm{GP}}(3,6)^\nu$  & comb. type & generators of the cone & count\\ 
\hline
\endfirsthead

\multicolumn{4}{c}
{{\bfseries \tablename\ \thetable{} -- continued from previous page}} \\
\hline 
stratum of $\overline{\mathbf{X}}_{\mathrm{GP}}(3,6)^\nu$  & comb. type & generators of the cone & count
\\ \hline 
\endhead

\hline \multicolumn{4}{|r|}{{Continued on next page}} \\ \hline
\endfoot
\hline 
\endlastfoot

$\#1$ & E & $e_{456}$
 & 20 \\
\hline
$\#2$ & EE(a) & $e_{123},\,e_{456}$
 & 10\\
\hline
$\#3$& EE(b) & $e_{124},\,e_{456}$ & 90
\\
\hline
$\#4$ & EEE & $e_{456},\,e_{124},\,e_{135}$ &120
\\
\hline
$\#5$ & F & $f_{1234}$ &15
\\
\hline
$\#6$ & EF(i) & $e_{123},\,f_{1234}$ & 60
\\
\hline
$\#7$ & EEEE & $e_{126},\,e_{235},\,e_{145},\,e_{346}$ & 30
\\
\hline
$\#8$ & EF(ii) & $e_{156},\,f_{1234}$ & 60
\\
\hline
$\#9$ & EEF(b) &$e_{123},\,e_{156},\,f_{1234}$ & 180
\\
\hline
$\#10$ & EEF(a)  &$e_{234},\,e_{156},\,f_{1234}$ & 60
\\
\hline
$\#11$ & G & $g_{125634}$ & 30
\\
\hline
$\#12$ & C  &$c_{3456}$ &45
\\
\hline
$\#12.1$ & CF  &$c_{3456},\,f_{1234}$ & 90
\\
\hline
$\#13$ & FG& $f_{1234},\,g_{123456}$ & 90 
\\
\hline
$\#14$ & EG & $e_{356},\,g_{125634}$ & 180 
\\
\hline
$\#15$ & CG & $c_{1256},\,g_{125634}$ &90
\\
\hline
$\#15.1$ & CFG(i)& $c_{1256},\,f_{3456},\,g_{125634}$ &90
\\
\hline
$\#15.2$ & CFG(ii)& $c_{1256},\,f_{1234},\,g_{125634}$ &90
\\
\hline
$\#16$ & CE & $c_{1256},\,e_{124}$ & 180
\\
\hline
$\#16.1$ & CEF(ii) & $c_{1256},\,e_{124},\,f_{3456}$ &180
\\
\hline
$\#16.2$ & CEF(i) & $c_{1256},\,e_{124},\,f_{1234}$ & 180
\\
\hline
$\#17$ & EFG(ii)& $e_{125},\,f_{3456},\,g_{125634}$ & 180
\\
\hline
$\#18$ & EFG(i) & $e_{356},\,f_{3456},\,g_{125634}$ & 180
\\
\hline
$\#19$ & EEG & $ e_{125},\,e_{356},\,g_{125634}$ & 360 
\\
\hline
$\#20$ & D & $d_{3456}$ & 15
\\
\hline
$\#20.1$ & DG & $d_{3456},\,g_{123456}$ &30
\\
\hline
$\#20.2$ & DF & $d_{3456},\,f_{1234}$ &45
\\
\hline
$\#20.3$ & CDF & $c_{1256},\,d_{3456},\,f_{1234}$ &90\\
\hline
$\#20.4$ & DFG & $d_{3456},\,f_{1234},\,g_{123456}$ &90
\\
\hline
$\#20.5$ & CDFG & $c_{1256},\,d_{3456},\,f_{1234},\,g_{123456}$ &180
\\
\hline
$\#20.6$ & CD & $c_{1256},\,d_{3456}$ &45
\\
\hline
$\#20.7$ & CDG &$c_{1256},\,d_{3456},\,g_{123456}$ &90
\\
\hline
$\#21$ & CEG & $c_{3456},\,e_{256},\,g_{123456}$ & 180
\\
\hline
$\#21.1$ & CEFG(i)& $c_{3456},\,e_{256},\,f_{1256},\,g_{123456}$ & 180
\\
\hline
$\#21.2$ & CEFG(ii) & $c_{3456},\,e_{256},\,f_{1234},\,g_{123456}$ & 180
\\
\hline
$\#22$ & EEFG & $e_{256},\,e_{346},\,f_{1256},\,g_{123456}$ &360
\\
\hline
$\#23$ & CEE(a) & $c_{3456},\,e_{134},\,e_{256}$ & 90
\\
\hline
$\#23.1$ & CEEF(a) & $c_{3456},\,e_{134},\,e_{256},\,f_{1256}$ & 180
\\
\hline
$\#24$ & CEE(b) & $c_{3456},\,e_{134},\,e_{156}$ & 90
\\
\hline
$\#24.1$ & CEEF(b) & $c_{3456},\,e_{134},\,e_{156},\,f_{1256}$ &180
\\
\hline
$\#25$ & EEEG & $e_{125},\,e_{234},\,e_{456},\,g_{125634}$ &240
\\
\hline
\end{longtable} 
\end{center}
\begin{remark}\label{rem:notation}
For lighter notation, in the third column of Table~\ref{tab:correspondencekapranov'sstrataandcones_refinement} we denote the vectors generating the cones up to permutation of the indices (instead of using notation like $e_{i_1i_2i_3}$). 
The choice of representatives follows the depiction of the corresponding $6$-pointed Mustafin degenerations of the strata as in \cite[Tables~13-16]{ST22}.

We also note that some combinatorial types of cones are generated by the same types of rays, and thus are counted together in the proof of Theorem \ref{thm:f-vector-of-widehat-Sigma(3,6)}. However, the combinatorics of the rays depends not just on the type but on the concrete generators of the cones. For example, the edge EF comes in two combinatorial flavors: EF(i) generated by $e_{i_1i_2i_3}, f_{i_1i_2i_3i_4}$ and EF(ii) generated by $e_{i_1i_5i_6}, f_{i_1i_2i_3i_4}$. This later also influences cones like CEF, depending on whether one chooses the left or right triangle in the figure at the bottom-right of Table~\ref{tab:refinementTropGr}.

Finally, we point out a subtlety when comparing Luxton's notation for the ray generators of the tropical Grassmannian $\mathrm{trop}(\mathrm{Gr}(3,6)^\circ)$ in \cite{Lux08} (this convention is also used in \cite{ST22}) with the original notation of Speyer--Sturmfels in \cite{SS04}. More specifically, Luxton's ray generator $(123)$ corresponds to Speyer--Sturmfels' $e_{456}$ because they induce the same matroid polytope subdivision of the hypersimplex
\[
\Delta(3,6)=\{(x_1,\ldots,x_6)\in\mathbb{R}^6\mid0\leq x_i\leq1,~x_1+\ldots+x_6=3\}.
\]
Indeed, the subdivision associated to $(123)$ is the splitting by the hyperplane $x_1+x_2+x_3=1$: see \cite[Section~4.2.4]{Lux08}. This matches the subdivision associated to $e_{456}$, which is the one induced by the height function with value $1$ on the vertex $(0,0,0,1,1,1)\in\Delta(3,6)$ and zero on the other vertices. For the other types of ray generators, Luxton's $(12)$ corresponds to Speyer--Sturmfels' $f_{3456}$, and lastly Luxton's $(12)(34)(56)$ is associated to $g_{125634}$.
\end{remark}

\subsection{The combinatorics of the fan \texorpdfstring{$\widehat{\Sigma}^{\mathrm{sd}}(3,6)$}{hatSigmasd(3,6)}}
Now we have all the ingredients to fully describe the fan $\widehat{\Sigma}^{\mathrm{sd}}(3,6)$, which gives us the tropical compactification $\overline{\mathbf{S}}_{\mathrm{GP}}(3,6)^\nu$. A complete understanding of the cones of this fan gives us a full view of the stratification of the boundary of $\overline{\mathbf{S}}_{\mathrm{GP}}(3,6)^\nu$ \cite{Tev07}.
\begin{proposition}\label{prop:sd_refined_fan}
The fan $\widehat{\Sigma}^{\mathrm{sd}}(3,6)$ has f-vector $(1,85,330,270)$. The number of cones of a given type can be found in Table~\ref{tab:correspondencekapranov'sstrataandcones_selfdual}.
\end{proposition}

\begin{proof}
To obtain the f-vector of $\widehat{\Sigma}^{\rm sd}(3,6)$,  we can either apply the refinement explained in Section~\ref{sec:refinedfan} to $\Sigma^{\rm sd}(3,6)$ or intersect $\widehat{\Sigma}(3,6)$ with $L^{\rm sd}$. We use the latter approach. Preliminarily, recall that the different types of cones of $\widehat{\Sigma}(3,6)$ and their count can be found in Table~\ref{tab:correspondencekapranov'sstrataandcones_refinement}.

The only rays in $\widehat{\Sigma}(3,6)$ that are contained in $L^{\mathrm{sd}}$ are the ones of type F, C, and D. Additionally, we obtain the rays of type E$^{\rm sd}$ (generated by $e_{i_1i_2i_3}+e_{i_4i_5i_6}$) by intersecting $L^{\rm sd}$ with the EE(a). So, there are $10$ rays of type E$^{\rm sd}$. In conclusion, the dimension one entry of the f-vector is given by
\[
15\,\text{\scriptsize{\rm{(F)}}\normalsize}+10\,\text{\scriptsize{\rm{(E$^{\rm sd}$)}}\normalsize}+45\,\text{\scriptsize{\rm{(C)}}\normalsize}+15\,\text{\scriptsize{\rm{(D)}}\normalsize}=85.
\]

For the $2$-dimensional cones, we obtain $90$ CE$^{\rm sd}$ cones from intersecting the CEE(a) cones with $L^{\rm sd}$ and $60$ E$^{\rm sd}$F cones from intersecting the EEF(a) cones with $L^{\rm sd}$. Only the $2$-dimensional cones CF, CD, and DF are contained in $L^{\rm sd}$. Moreover, no higher dimensional cone intersects with $L^{\rm sd}$ in a way that would give an additional $2$-dimensional cones of $\widehat{\Sigma}^{\mathrm{sd}}(3,6)$. So, in total these are
\[
90\,\text{\scriptsize{\rm{(CE$^{\rm sd}$)}}\normalsize}+60\,\text{\scriptsize{\rm{(E$^{\rm sd}$F)}}\normalsize}+90\,\text{\scriptsize{\rm{(CF)}}\normalsize}+45\,\text{\scriptsize{\rm{(CD)}}\normalsize}+45\,\text{\scriptsize{\rm{(DF)}}\normalsize}=330.
\]

All $3$-dimensional cones $\sigma\in\widehat{\Sigma}^{\mathrm{sd}}(3,6)$ arise as the intersection of a maximal cone of $\sigma\in\widehat{\Sigma}(3,6)$ with $L^{\mathrm{sd}}$. Let us go through all the possibilities. The cones of type EEEE and EEEG intersect $L^{\rm sd}$ only in zero.  Cones of type EEFG intersect $L^{\mathrm{sd}}$ in an F ray, while CEFG and CEEF(b) cones intersect it giving a CF cone. The cones of type CDFG intersect $L^{\mathrm{sd}}$ in the CDF cones, while the CEEF(a) cones give rise to the CE$^{\rm sd}$F cones. By the enumeration in Table~\ref{tab:correspondencekapranov'sstrataandcones_refinement}, we obtain $90\,\text{\scriptsize{\rm{(CDF)}}\normalsize}+180\,\text{\scriptsize{\rm{(CE$^{\rm sd}$F)}}\normalsize}=270$ cones of dimension three in $\widehat{\Sigma}^{\rm sd}(3,6)$.

All the types and numbers of the different cones of $\widehat{\Sigma}^{\rm sd}(3,6)$ are summarized in Table~\ref{tab:correspondencekapranov'sstrataandcones_selfdual}.
\end{proof}

\begin{definition}
Let us consider the inclusion
\[
\overline{\mathbf{S}}_{\mathrm{GP}}(3,6)^\nu\cong\widehat{\mathbf{S}}(3,6)\hookrightarrow\widehat{\mathbf{X}}(3,6)\cong\overline{\mathbf{X}}_{\mathrm{GP}}(3,6)^\nu.
\]
As the boundary $\overline{\mathbf{S}}_{\mathrm{GP}}(3,6)^\nu\setminus\mathbf{S}(3,6)$ is contained in the boundary of $\overline{\mathbf{X}}_{\mathrm{GP}}(3,6)^\nu$, we endow it with the stratification induced by the latter. More precisely, if $B\subseteq\overline{\mathbf{X}}_{\mathrm{GP}}(3,6)^{\nu}$ is a boundary stratum of a certain type (by this we mean closed, and generically parametrizing $6$-pointed Mustafin degeneration of the given type as indicated in \cite[Tables~13--16]{ST22}), then $B\cap\overline{\mathbf{S}}_{\mathrm{GP}}(3,6)^\nu$ is a boundary stratum of the same type as $B$, and all strata of $\overline{\mathbf{S}}_{\mathrm{GP}}(3,6)^\nu$ arise in this way. 
\end{definition}

\begin{proposition}
\label{prop:strataofselfdualcompactification}
The correspondence between the strata of $\overline{\mathbf{S}}_\mathrm{GP}(3,6)^\nu$ and the cones of $\widehat{\Sigma}^{\mathrm{sd}}(3,6)$ is illustrated in Table~\ref{tab:correspondencekapranov'sstrataandcones_selfdual}. In particular, if $(X;p_1,\ldots,p_6)$ is a $6$-pointed Mustafin degeneration of $\mathbb{P}^2$ parametrized by the boundary of $\overline{\mathbf{S}}_{\mathrm{GP}}(3,6)^\nu$, then it can only be of the types in Table~\ref{tab:correspondencekapranov'sstrataandcones_selfdual}.
\end{proposition}

\begin{proof}
Table~\ref{tab:correspondencekapranov'sstrataandcones_refinement} shows the correspondence between strata of $\overline{\mathbf{X}}_{\mathrm{GP}}(3,6)^\nu$ and $\widehat{\Sigma}(3,6)$  from \cite[Section~12]{ST22} explicitly. Every cone in $\widehat{\Sigma}^{\mathrm{sd}}(3,6)$ is the intersection of a cone in $\widehat{\Sigma}(3,6)$ with $L^\mathrm{sd}$. As boundary strata of $\overline{\mathbf{S}}_\mathrm{GP}(3,6)^\nu$ are restrictions of boundary strata of $\overline{\mathbf{X}}_\mathrm{GP}(3,6)^\nu$, we can fill out the first column of Table~\ref{tab:correspondencekapranov'sstrataandcones_selfdual} by using Table~\ref{tab:correspondencekapranov'sstrataandcones_refinement}. We illustrate this in two examples: A cone of type EE(a) of $\widehat{\Sigma}(3,6)$ corresponds to a stratum of type \#2 in $\overline{\mathbf{X}}_\mathrm{GP}(3,6)^\nu$. Intersecting such a cone with $L^{\rm sd}$ yields a ray of type $\mathrm{E}^{\rm sd}$ of $\widehat{\Sigma}^{\rm sd}(3,6)$. Thus, a ray of type $\mathrm{E}^{\rm sd}$ corresponds to a stratum of $\overline{\mathbf{S}}_\mathrm{GP}(3,6)^\nu$ of type \#2. If one looks at a stratum of $\overline{\mathbf{X}}_\mathrm{GP}(3,6)^\nu$ of type \#3, then it corresponds to a cone of type EE(b). Such a cone intersects $L^{\rm sd}$ only in zero, so there is no stratum of $\overline{\mathbf{S}}_\mathrm{GP}(3,6)^\nu$ of this type.
\end{proof}

\begin{table}[hbtp]
\centering
\caption{Correspondence between $\widehat{\Sigma}^{\mathrm{sd}}(3,6)$ and the strata of $\overline{\mathbf{S}}_{\mathrm{GP}}(3,6)^\nu$.}
\label{tab:correspondencekapranov'sstrataandcones_selfdual}
\begin{tabular}{|>{\centering\arraybackslash}m{2cm}|>{\centering\arraybackslash}m{2cm}|>{\centering\arraybackslash}m{8cm}|>{\centering\arraybackslash}m{2cm}|}
\hline
stratum of $\overline{\mathbf{S}}_{\mathrm{GP}}(3,6)^\nu$  & comb. type & generators of the cone & count\\ \hline
$\#2$ & ${\rm E}^{\rm sd} $ & $e_{123}+e_{456}$ & 10
\\
\hline
$\#5$ & F & $f_{1234}$ & 15
\\
\hline
$\#10$ & ${\rm E}^{\rm sd} $F  &$e_{234}+e_{156},\, f_{1234}$ & 60
\\
\hline
$\#12$ & C  &$c_{3456}$ & 45
\\
\hline
$\#12.1$ & CF  &
$c_{3456},\,f_{1234}$ & 90
\\
\hline
$\#20$ & D & $d_{3456}$ & 15
\\
\hline
$\#20.2$ & DF & 
$d_{3456},\,f_{1234} $ & 45
\\
\hline
$\#20.3$ & CDF & 
$c_{1256},\,d_{3456},\,f_{1234} $ & 90\\
\hline
$\#20.6$ & CD & 
$c_{1256},\,d_{3456} $ & 45
\\
\hline
$\#23$ & C${\rm E}^{\rm sd} $ & $c_{3456},\,e_{134}+e_{256}$
& 90
\\
\hline
$\#23.1$ & C${\rm E}^{\rm sd} $F  & $c_{3456},\,e_{134}+e_{256},\,f_{1256}$ 
& 180
\\
\hline
\end{tabular}
\end{table}

\section{One-parameter Mustafin degenerations and conics}
\label{sec:one-parameter-Mustafin-degens-and-conics}

If $(X;p_1,\ldots,p_6)$ is a $6$-pointed Mustafin degeneration parametrized by $\overline{\mathbf{S}}_{\mathrm{GP}}(3,6)^\nu$, then extra information is encoded in the marked points, as $(X;p_1,\ldots,p_6)$ is a flat degeneration of $\mathbb{P}^2$ with six marked points on a smooth conic. In this section, we study how the conic degenerates in one-parameter families.

We start by recalling some standard definitions from \cite{CHSW11}. Let $\Sigma$ be a finite set of lattices. An irreducible component of the central fiber of the Mustafin variety $\PP(\Sigma)_\Bbbk$ is called a \emph{primary component} if it maps birationally to the special fiber of the factor $\PP(L_i)$ for some $[L_i]\in\Sigma$. Otherwise it is called a \emph{secondary component}. There is a bijection between primary components of $\mathbb{P}(\Sigma)_\Bbbk$ and the lattice classes in $\Sigma$ \cite[Corollary~2.5]{CHSW11}. Next, we need a technical proposition that will be used in the proof of the main result of this section: Theorem~\ref{thm:1parameterfamilyconics}.

\begin{proposition}
\label{prop:forgetfulmorphism}
Let $\Sigma$ be a finite set of lattice classes. If $\Sigma'\subseteq\Sigma$, then there is a morphism $\rho\colon\mathbb{P}(\Sigma)\to \mathbb{P}(\Sigma')$ satisfying the following:
\begin{enumerate}

\item $\rho$ is an isomorphism in the complement of the central fibers;

\item If $X\subseteq\mathbb{P}(\Sigma)_\Bbbk$ is a primary component corresponding to a lattice class in $\Sigma\setminus\Sigma'$, then $X$ is either contracted by $\rho$ or mapped to a secondary component of $\mathbb{P}(\Sigma')_\Bbbk$.

\end{enumerate}
\end{proposition}

\begin{proof}
Let us write $\Sigma = \{[L_1],\ldots,[L_r]\}$ and $\Sigma'= \{[L_1],\ldots,[L_m]\}$ with $r\geq m$.
We have the following commutative diagram:
\begin{center}
\begin{tikzpicture}[>=angle 90]
\matrix(a)[matrix of math nodes,
row sep=2em, column sep=4em,
text height=2ex, text depth=1ex]
{\mathbb{P}(\Sigma)& \mathbb{P}(\Sigma')\\
\mathbb{P}(L_1)\times_R\ldots\times_R\mathbb{P}(L_r) &\mathbb{P}(L_1)\times_R\ldots\times_R\mathbb{P}(L_m)\\
&\mathbb{P}(L_i),\\};
\path[->] (a-1-1) edge node[above]{$\rho$}(a-1-2);
\path[right hook->] (a-1-1) edge node[]{}(a-2-1);
\path[right hook->] (a-1-2) edge node[]{}(a-2-2);
\path[->] (a-2-1) edge node[above]{$\pi$}(a-2-2);
\path[->] (a-2-1) edge node[below left]{$\mathrm{pr}_i$}(a-3-2);
\path[->] (a-2-2) edge node[right]{$\mathrm{pr}_i'$}(a-3-2);
\end{tikzpicture}
\end{center}
where $\mathrm{pr}_i,\mathrm{pr}_i'$ are the projections on the $i$-th factor, $i\in\{1,\ldots,m\}$, and $\rho\colon\mathbb{P}(\Sigma)\rightarrow\mathbb{P}(\Sigma')$ is the restriction of the projection $\pi$ on the first $m$ factors.

Now, let us consider a primary component $X\subseteq\PP(\Sigma)_\Bbbk$ corresponding to a lattice class $[L] \in\Sigma\setminus\Sigma'$ and suppose by contradiction that $\rho(X)\subseteq\mathbb{P}(\Sigma')_\Bbbk$ is a primary component. Then, there exists $i\in\{1,\ldots,m\}$ such that $\mathrm{pr}_i'(\rho(X))=\mathbb{P}(L_i)_\Bbbk$. Therefore, we have that
\[
\mathrm{pr}_i(X)=\mathrm{pr}_i'(\pi(X))=\mathrm{pr}_i'(\rho(X))=\mathbb{P}(L_i)_\Bbbk.
\]
Hence, $X$ is the primary component of $\mathbb{P}(\Sigma)_\Bbbk$ corresponding to a lattice class $[L']\in\Sigma'$. As $[L]\neq[L']$ and both correspond to $X$, this contradicts the bijection between $\Sigma$ and the primary components of $\mathbb{P}(\Sigma)_\Bbbk$. In conclusion, either $\rho$ contracts $X$, or $\rho(X)$ is a secondary component of $\mathbb{P}(\Sigma')_\Bbbk$, proving~(2).

To prove (1) we can consider, say, $\mathrm{pr}_1$ and $\mathrm{pr}_1'$ restricted to $\mathbb{P}(\Sigma)_K$ and $\mathbb{P}(\Sigma')_K$ and observe that they give isomorphisms with $\mathbb{P}(L_1)_K$. Hence, also $\rho$ is an isomorphism away from the central fibers.
\end{proof}

\begin{definition}
We call $\rho$ as in Proposition~\ref{prop:forgetfulmorphism} a \emph{forgetful morphism}, as it ``forgets'' the lattice classes in $\Sigma\setminus\Sigma'$. We will be particularly interested in the following setting. Let $\mathbf{a}=(a_1,\ldots,a_6)\colon\mathrm{Spec}(K)\rightarrow\mathbf{B}(3,6)$ be an arc and let $i\in\{1,\ldots,6\}$. Let $\pi_i(\mathbf{a})=(a_1,\ldots,a_{i-1},a_{i+1},\ldots,a_6)$. Then, $\Sigma_{\pi_i(\mathbf{a})}\subseteq\Sigma_{\mathbf{a}}$. We denote by $\rho_i$ the forgetful morphism $\mathbb{P}(\Sigma_{\mathbf{a}})\rightarrow\mathbb{P}(\Sigma_{\pi_i(\mathbf{a})})$.
\end{definition}

\begin{remark}
There are examples where a primary component $X$ of $\PP(\Sigma)_\Bbbk$ is contracted under a forgetful morphism (see $\rho_6$ as depicted in Figure~\ref{fig:conic-in-case-2}), as well as examples where $X$ is mapped to a secondary component (see $\rho_2$ as depicted in Figure~\ref{fig:conic-in-case-12}).
\end{remark}

\begin{remark}
In all the cases we consider in this section, the central fibers of the Mustafin varieties $\PP(\Sigma_{\pi_i(\mathbf{a})})$ are constructed using the method described in \cite[Remark~5.3]{ST22}, which follows \cite[Section~4]{CHSW11}.
\end{remark}

\begin{theorem}
\label{thm:1parameterfamilyconics}
Let $x\in\overline{\mathbf{S}}_{\mathrm{GP}}(3,6)\setminus\mathbf{S}(3,6)$ be a point and let $(X;p_1,\ldots,p_6)$ be the $6$-pointed Mustafin degeneration of $\mathbb{P}^2$ parametrized by $x$. Let $\mathbf{a}\colon\mathrm{Spec}(K)\rightarrow\mathbf{S}(3,6)$ be an arc such that the unique extension $\overline{\mathbf{a}}\colon\mathrm{Spec}(R)\rightarrow\overline{\mathbf{S}}_{\mathrm{GP}}(3,6)$ satisfies $\overline{\mathbf{a}}(0)=x$. In particular,
\[
(\mathbb{P}(\Sigma_\mathbf{a})_\Bbbk;\overline{a}_1(0),\ldots,\overline{a}_6(0))=(X;p_1,\ldots,p_6).
\]
Let $\mathcal{C}\subseteq\PP(\Sigma_{\mathbf{a}})_K$ be the one-parameter family of smooth conics such that for any $t\neq0$, the six points $\mathbf{a}(t)$ lie on $\mathcal{C}_t$. Define $\overline{\mathcal{C}}$ as the closure of $\mathcal{C}\subseteq\PP(\Sigma_{\mathbf{a}})$. Then,
\begin{enumerate}
    \item the central fiber $\overline{\mathcal{C}}_0\subseteq\PP(\Sigma_{\mathbf{a}})_\Bbbk$ is a reduced curve containing the points $\overline{\mathbf{a}}(0)$, 
    \item the closure $\overline{\mathcal{C}}$ is flat over $\mathrm{Spec}(R)$, and
    \item the isomorphism class of the tuple $(\mathbb{P}(\Sigma_{\mathbf{a}})_\Bbbk;\overline{\mathcal{C}}_0;\overline{a}_1(0),\ldots,\overline{a}_6(0))$ is independent from the choice of $\mathbf{a}$ and only depends on $x$.
\end{enumerate}
\end{theorem}

\begin{proof}
By continuity, the limiting markings $(p_1,\ldots,p_6)=(\overline{a}_1(0),\ldots,\overline{a}_6(0))$ are contained in $\overline{\mathcal{C}}_0$. So, for (1) the main thing to prove is that $\overline{\mathcal{C}}_0$ is reduced and of dimension $1$.

The general strategy is as follows. For each $i\in\{1,\ldots,6\}$, let $\pi_i(\mathbf{a})=(a_1,\ldots,\widehat{a}_i,\ldots,a_6)$. Consider the morphism $\rho_i\colon\mathbb{P}(\Sigma_{\mathbf{a}})\rightarrow\mathbb{P}(\Sigma_{\pi_i(\mathbf{a})})$ guaranteed by Proposition~\ref{prop:forgetfulmorphism}. Since $\rho_i$ is proper and continuous, we have that
\[
\rho_i\left(\overline{\mathcal{C}}\right)=\overline{\rho_i(\mathcal{C})}.
\]
Now, the central fiber $\overline{\rho_i(\mathcal{C})}_0$ of $\overline{\rho_i(\mathcal{C})}$ in $\mathbb{P}(\Sigma_{\pi_i(\mathbf{a})})$ can be described via Propositions~\ref{prop:degenerate-5-pointed-conics-1-strata} and~\ref{prop:degenerate-5-pointed-conics-0-strata}. We show that $\overline{\mathcal{C}}_0\subseteq\mathbb{P}(\Sigma_{\mathbf{a}})$ can be reconstructed from the different $\overline{\rho_i(\mathcal{C})}_0$. We do this by considering all possible cases for $(X;p_1,\ldots,p_6)$, which by Proposition~\ref{prop:strataofselfdualcompactification} can only be as in Table~\ref{tab:correspondencekapranov'sstrataandcones_selfdual}. (The order in which we consider the types reflects the amount of work needed to understand $\overline{\mathcal{C}}_0$, starting from the simpler cases.)

\begin{figure}[h]
    \centering
    \begin{subfigure}{.3\textwidth}
            \begin{tikzpicture}[scale=0.5]

	\draw[line width=1pt] (0,6) -- (6,6);
	\draw[line width=1pt] (6,6) -- (6,0);
	\draw[line width=1pt] (6,0) -- (0,6);
	\draw[line width=1pt] (3,6) -- (6,3);
	\fill[gray!50,nearly transparent] (0,6) -- (6,6) -- (6,0) -- cycle;
	\fill (6,0) circle (5pt);
	\fill (6,6) circle (5pt);
	\fill (0,6) circle (5pt);
	\fill (2,5) circle (5pt);
	\fill (5.5,5) circle (5pt);
	\fill (4,3.5) circle (5pt);

	\node at (-.5,6) {$1$};
	\node at (2.5,5) {$2$};
	\node at (3.5,3.5) {$3$};
	\node at (6.5,0) {$4$};
	\node at (6.5,6) {$5$};
	\node at (5.5,4.4) {$6$};

	\node at (1,2) {\#5};

\end{tikzpicture}
\caption{$(\mathbb{P}(\Sigma_\mathbf{a})_\Bbbk;\overline{a}_1(0),\ldots,\overline{a}_6(0))$.}\label{fig:5}
    \end{subfigure}
    \begin{subfigure}{.3\textwidth}\centering
            \begin{tikzpicture}[scale=0.5]

	\draw[line width=1pt] (0,6) -- (6,6);
	\draw[line width=1pt] (6,6) -- (6,0);
	\draw[line width=1pt] (6,0) -- (0,6);
	\draw[line width=1pt] (3,6) -- (6,3);
	\fill[gray!50,nearly transparent] (0,6) -- (6,6) -- (6,0) -- cycle;
	\fill (6,6) circle (5pt);
	\fill (0,6) circle (5pt);
	\fill (2,5) circle (5pt);
	\fill (5.5,5) circle (5pt);
	\fill (4,3.5) circle (5pt);

	\node at (-.5,6) {$1$};
	\node at (2.5,5) {$2$};
	\node at (3.5,3.5) {$3$};
	\node at (6.5,-0.25) {};
	\node at (6.5,6) {$5$};
	\node at (5.5,4.4) {$6$};

	\node at (1,2) {};

\end{tikzpicture}
\caption{$\rho_4(\mathbb{P}(\Sigma_{\mathbf{a}})_\Bbbk)=\mathbb{P}(\Sigma_{\pi_4(\mathbf{a})})_\Bbbk$.}
    \end{subfigure}
   \begin{subfigure}{.3\textwidth}\centering
\begin{tikzpicture}[scale=0.5]
\draw [line width=1.pt] (0.,0.)-- (5.,3.);
\draw [line width=1.pt,color=lightgray] (3.,3.)-- (8.,0.);
\draw (3.5,1.4) node {$1$};
\draw (2.7,1.0) node {$2$};
\draw (1.9,0.5) node {$3$};
\draw (1.0,0.1) node {$4$};
\draw (5,1.1) node {$5$};
\draw (6.3,0.4) node {$6$};

\draw (0,-2) node {};

\draw [fill=black] (3.2,1.92) circle (4.0pt);
\draw [fill=black] (5.333333333333331,1.6) circle (4.0pt);
\draw [fill=black] (2.4,1.44) circle (4.0pt);
\draw [fill=black] (1.6,0.96) circle (4.0pt);
\draw [fill=black] (6.666666666666666,0.8) circle (4.0pt);
\draw [fill=black] (0.8,0.48) circle (4.0pt);

\end{tikzpicture}
\caption{$(\overline{\mathcal{C}}_0;p_1,\ldots,p_6)$.}\label{fig:curve-case-5}
   \end{subfigure}

    \caption{Computation of the limit curve $\overline{\mathcal{C}}_0$ in type \#5.}
    \label{fig:conic-in-case-5}
\end{figure}
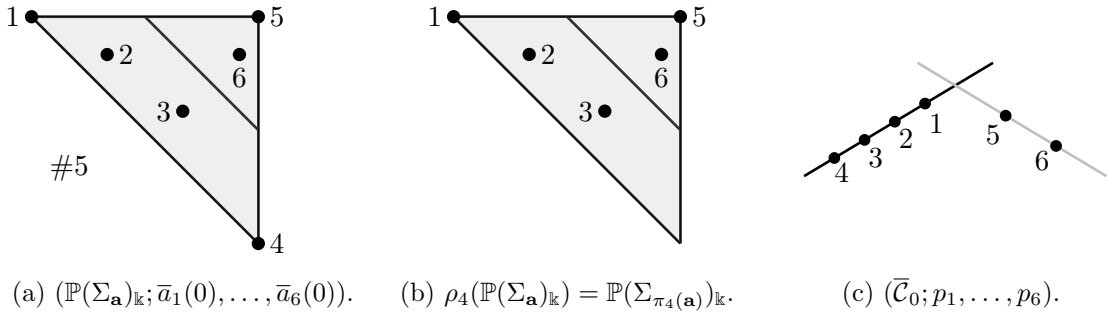

\textbf{Type \#5.} We have that $(\mathbb{P}(\Sigma_\mathbf{a})_\Bbbk;\overline{a}_1(0),\ldots,\overline{a}_6(0))$ is as in Figure~\ref{fig:5} (up to a permutation of the labels). If we let $i=4$, then $\Sigma_{\mathbf{a}}=\Sigma_{\pi_4(\mathbf{a})}$. Therefore, $\rho_4$ is the identity, and hence $\overline{\mathcal{C}}_0=\overline{\rho_4(\mathcal{C})}_0$. We find that $\overline{\rho_4(\mathcal{C})}_0$ is described in Proposition~\ref{prop:degenerate-5-pointed-conics-1-strata}, and it is a reduced curve. This allows us to conclude the same for $\overline{\mathcal{C}}_0$. We summarize these considerations in Figure~\ref{fig:conic-in-case-5}.

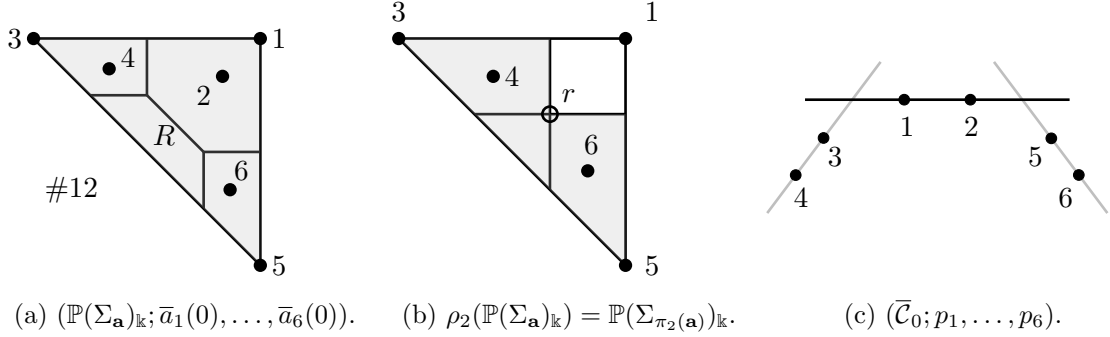
\begin{figure}[h]
    \centering
    \begin{subfigure}{.3\textwidth}
            \begin{tikzpicture}[scale=0.5]

	\draw[line width=1pt] (0,6) -- (6,6);
	\draw[line width=1pt] (6,6) -- (6,0);
	\draw[line width=1pt] (6,0) -- (0,6);
	\draw[line width=1pt] (12/4,6) -- (12/4,6-6/4);
	\draw[line width=1pt] (12/4,6-6/4) -- (18/4,6-12/4);
	\draw[line width=1pt] (18/4,6-12/4) -- (6,6-12/4);
	\draw[line width=1pt] (6-6/4,6-12/4) -- (6-6/4,6/4);
	\draw[line width=1pt] (6-12/4,6-6/4) -- (6/4,6-6/4);
	\fill[gray!50,nearly transparent] (0,6) -- (6,6) -- (6,0) -- cycle;
	\fill (6,0) circle (5pt);
	\fill (0,6) circle (5pt);
	\fill (6,6) circle (5pt);
	\fill (2,5.2) circle (5pt);
	\fill (5,5) circle (5pt);
	\fill (5.2,2) circle (5pt);

	\node at (6.5,6) {$1$};
	\node at (4.5,4.5) {$2$};
	\node at (-.5,6) {$3$};
	\node at (2.5,5.5) {$4$};
	\node at (6.5,0) {$5$};
	\node at (5.5,2.5) {$6$};
\node at (3.45,3.45) {$R$};
	\node at (1,2) {\#12};

\end{tikzpicture}
\caption{$(\mathbb{P}(\Sigma_\mathbf{a})_\Bbbk;\overline{a}_1(0),\ldots,\overline{a}_6(0))$.}\label{fig:12}
    \end{subfigure}
    \begin{subfigure}{.3\textwidth}
            \begin{tikzpicture}[scale=0.5]

	\draw[line width=1pt] (0,6) -- (6,6);
	\draw[line width=1pt] (6,0) -- (6,6);
	\draw[line width=1pt] (0,6) -- (6,0);
	\draw[line width=1pt] (4,6) -- (4,2);
	\draw[line width=1pt] (6,4) -- (2,4);
    \fill[gray!50,nearly transparent] (0,6) -- (4,6) -- (4,4) -- (2,4) -- cycle;
    \fill[gray!50,nearly transparent] (6,0) -- (6,4) -- (4,4) -- (4,2) -- cycle;
    \fill[gray!50,nearly transparent] (2,4) -- (4,4) -- (4,2) -- cycle;
	\fill (0,6) circle (5pt);
	\fill (2.5,5) circle (5pt);
	\fill (6,0) circle (5pt);
	\fill (5,2.5) circle (5pt);
	\fill (6,6) circle (5pt);

	\node at (0,6.7) {$3$};
	\node at (3,5) {$4$};
	\node at (6.7,0) {$5$};
	\node at (5.1,3.2) {$6$};
	\node at (6.7,6.7) {$1$};
	\node at (4.5,4.5) {$r$};
\draw[line width=1pt] (4,4) circle (0.2cm);
	\node at (1,2) {};

\end{tikzpicture}
\caption{$\rho_2(\mathbb{P}(\Sigma_{\mathbf{a}})_\Bbbk)=\mathbb{P}(\Sigma_{\pi_2(\mathbf{a})})_\Bbbk$.}\label{fig:conic-in-case-12-forget2}
    \end{subfigure}
   \begin{subfigure}{.3\textwidth}

\begin{tikzpicture}[scale=0.5]
\draw (3.9,-5.6) node {$4$};
\draw (6.7,-3.7) node {$1$};
\draw (8.4,-3.7) node {$2$};
\draw (4.8,-4.5) node {$3$};
\draw (10.1,-4.5) node {$5$};
\draw (10.9,-5.6) node {$6$};

\draw (10.9,-7.6) node {};

\draw [line width=1.pt,color=lightgray] (3.,-6.)-- (6.,-2.);
\draw [line width=1.pt,color=lightgray] (9.,-2.)-- (12.,-6.);
\draw [line width=1.pt] (4.,-3.)-- (11.,-3.);

\draw [fill=black] (6.625,-3.) circle (4.0pt);
\draw [fill=black] (8.375,-3.) circle (4.0pt);
\draw [fill=black] (4.4904,-4.0128) circle (4.0pt);
\draw [fill=black] (3.7536,-4.9952) circle (4.0pt);
\draw [fill=black] (10.5168,-4.0224) circle (4.0pt);
\draw [fill=black] (11.244,-4.992) circle (4.0pt);

\end{tikzpicture}
\caption{$(\overline{\mathcal{C}}_0;p_1,\ldots,p_6)$.}\label{fig:curve-case-12}
   \end{subfigure}

    \caption{Computation of the limit curve $\overline{\mathcal{C}}_0$ in type \#12.}
    \label{fig:conic-in-case-12}
\end{figure}

\textbf{Type \#12.} We have that $(\mathbb{P}(\Sigma_\mathbf{a})_\Bbbk;\overline{a}_1(0),\ldots,\overline{a}_6(0))$ is as in Figure~\ref{fig:12}. If we let $i=2$, then $\Sigma_{\pi_2(\mathbf{a})}=\Sigma_{\mathbf{a}}\setminus\{[L]\}$, where $L$ is the lattice stabilizing the quadruples
\[
1235,~1236,~1245,~1246,
\]
 which corresponds to the pentagon in Figure \ref{fig:12}. Hence, $\mathbb{P}(\Sigma_{\pi_2(\mathbf{a})})_\Bbbk$ is of type II as in Figure~\ref{fig:degenerations-for-5-points}, also shown in Figure~\ref{fig:conic-in-case-12-forget2}. We observe that no component of $\mathbb{P}(\Sigma_\mathbf{a})_\Bbbk$ is contracted under $\rho_2$ and the component $P\subseteq\mathbb{P}(\Sigma_{\mathbf{a}})_\Bbbk$ containing $p_1$ and $p_2$, which is the blow-up of $\PP^2$ in two points, becomes the secondary component  $Q\subseteq\mathbb{P}(\Sigma_{\pi_2(\mathbf{a})})_\Bbbk$ isomorphic to $\PP^1\times\PP^1$. The birational morphism $P\rightarrow Q$ contracts to a point $r$ a curve, which we denote by $R$ in Figure~\ref{fig:12}. We understand the curve $\overline{\rho_2(\mathcal{C})}_0$ via Proposition~\ref{prop:degenerate-5-pointed-conics-0-strata} and observe that it does not pass through the point $r$. Hence, we have that
\[
\overline{\mathcal{C}}_0=\rho_2^{-1}(\overline{\rho_2(\mathcal{C})}_0)
\]
is a reduced curve. The above considerations are summarized in Figure~\ref{fig:conic-in-case-12}.

\begin{figure}[h]
    \centering
    \begin{subfigure}{.3\textwidth}
            \begin{tikzpicture}[scale=0.5]

	\draw[line width=1pt] (0,6) -- (6,6);
	\draw[line width=1pt] (6,6) -- (6,0);
	\draw[line width=1pt] (6,0) -- (0,6);
	\draw[line width=1pt] (2,6) -- (3,5);
	\draw[line width=1pt] (3,5) -- (4,5);
	\draw[line width=1pt] (4,5) -- (4,6);
	\draw[line width=1pt] (4,5) -- (5,4);
	\draw[line width=1pt] (5,4) -- (6,4);
	\draw[line width=1pt] (5,4) -- (5,3);
	\draw[line width=1pt] (5,3) -- (6,2);
	\draw[line width=1pt] (5,3) -- (4,3);
	\draw[line width=1pt] (4,3) -- (4,2);
	\draw[line width=1pt] (4,3) -- (3,4);
	\draw[line width=1pt] (3,4) -- (3,5);
	\draw[line width=1pt] (3,4) -- (2,4);
	\fill[gray!50,nearly transparent] (2,6) -- (4,6) -- (4,5) -- (3,5) -- cycle;
	\fill[gray!50,nearly transparent] (6,2) -- (6,4) -- (5,4) -- (5,3) -- cycle;
	\fill[gray!50,nearly transparent] (2,4) -- (3,4) -- (4,3) -- (4,2) -- cycle;
	\fill[gray!50,nearly transparent] (0,6) -- (2,6) -- (3,5) -- (3,4) -- (2,4) --cycle;
	\fill[gray!50,nearly transparent] (6,0) -- (6,2) -- (5,3) -- (4,3) -- (4,2) --cycle;
	\fill[gray!50,nearly transparent] (4,6) -- (6,6) -- (6,4) -- (5,4) -- (4,5) --cycle;
	\fill (6,0) circle (5pt);
	\fill (0,6) circle (5pt);
	\fill (6,6) circle (5pt);
	\fill (5,2) circle (5pt);
	\fill (5,5) circle (5pt);
	\fill (2,5) circle (5pt);

	\node at (-.5,6) {$1$};
	\node at (1.5,5.3) {$2$};
	\node at (6.5,6) {$3$};
	\node at (5.5,4.9) {$4$};
	\node at (6.5,0) {$5$};
	\node at (5.3,1.5) {$6$};

	\node at (1,2) {\#20};

\end{tikzpicture}
\caption{$(\mathbb{P}(\Sigma_\mathbf{a})_\Bbbk;\overline{a}_1(0),\ldots,\overline{a}_6(0))$.}\label{fig:20}
    \end{subfigure}
    \begin{subfigure}{.3\textwidth}
            \begin{tikzpicture}[scale=0.5]

	\draw[line width=1pt] (0,6) -- (6,6);
	\draw[line width=1pt] (6,0) -- (6,6);
	\draw[line width=1pt] (0,6) -- (6,0);
	\draw[line width=1pt] (4,6) -- (4,2);
	\draw[line width=1pt] (6,4) -- (2,4);
    \fill[gray!50,nearly transparent] (0,6) -- (4,6) -- (4,4) -- (2,4) -- cycle;
    \fill[gray!50,nearly transparent] (6,0) -- (6,4) -- (4,4) -- (4,2) -- cycle;
    \fill[gray!50,nearly transparent] (2,4) -- (4,4) -- (4,2) -- cycle;
	\fill (0,6) circle (5pt);
	\fill (2.5,5) circle (5pt);
	\fill (6,0) circle (5pt);
	\fill (5,2.5) circle (5pt);
	\fill (6,6) circle (5pt);

	\node at (0,6.7) {$1$};
	\node at (3,5) {$2$};
	\node at (6.7,0) {$5$};
	\node at (5.1,3.2) {$6$};
	\node at (6.7,6.7) {$3$};

\end{tikzpicture}
\caption{$\rho_4(\mathbb{P}(\Sigma_{\mathbf{a}})_\Bbbk)=\mathbb{P}(\Sigma_{\pi_4(\mathbf{a})})_\Bbbk$.}\label{fig:conic-in-case-20_forget4}
    \end{subfigure}
   \begin{subfigure}{.3\textwidth}

\begin{tikzpicture}[scale=0.5]
\draw (8.0,-5.0) node {$4$};
\draw (5.5,-4.0) node {$1$};
\draw (5.5,-5.0) node {$2$};
\draw (8.0,-4.0) node {$3$};
\draw (10.5,-4.0) node {$5$};
\draw (10.5,-5.0) node {$6$};
\draw (10.5,-7.5) node {};
\draw [line width=1.pt, color = black] (4.,-3.)-- (11.,-3.);
\draw [line width=1.pt, color = lightgray] (5.,-2.)-- (5.,-6.);
\draw [line width=1.pt, color = lightgray] (10.,-2.)-- (10.,-6.);
\draw [line width=1.pt, color = lightgray] (7.5,-2.)-- (7.5,-6.);

\draw [fill=black] (5.,-4.013611788132217) circle (4.0pt);
\draw [fill=black] (5.,-4.994667521093667) circle (4.0pt);
\draw [fill=black] (7.5,-4.017594265232974) circle (4.0pt);
\draw [fill=black] (7.5,-4.993599487343801) circle (4.0pt);
\draw [fill=black] (10.,-4.021576742333731) circle (4.0pt);
\draw [fill=black] (10.,-4.9925314535939345) circle (4.0pt);

\end{tikzpicture}
\caption{$(\overline{\mathcal{C}}_0;p_1,\ldots,p_6)$.}\label{fig:curve-case-20}
   \end{subfigure}

    \caption{Computation of the limit curve $\overline{\mathcal{C}}_0$ in type \#20.}
    \label{fig:conic-in-case-20}
\end{figure}
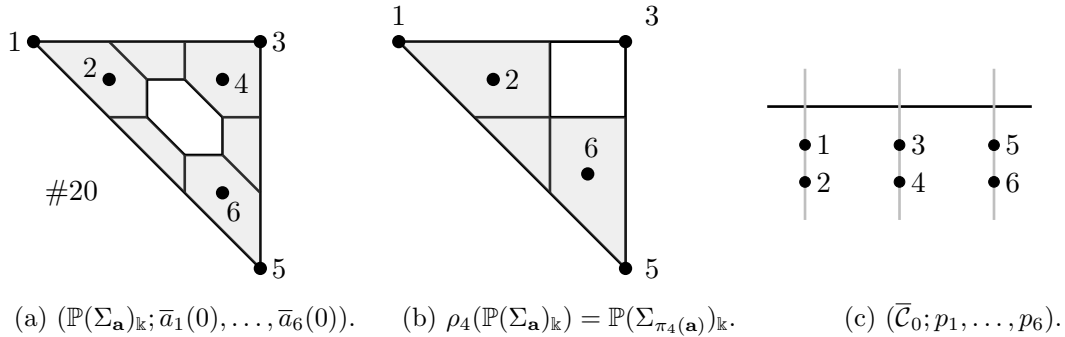

\textbf{Type \#20.} Let us consider $(\mathbb{P}(\Sigma_\mathbf{a})_\Bbbk;\overline{a}_1(0),\ldots,\overline{a}_6(0))$ as in Figure~\ref{fig:20}. Let $H$ be the secondary component of $\mathbb{P}(\Sigma_\mathbf{a})_\Bbbk$ isomorphic to the blow up of $\mathbb{P}^2$ at three not aligned points. If we consider $\rho_4$, then $\Sigma_{\pi_4(\mathbf{a})}=\Sigma_{\mathbf{a}}\setminus\{[L_1],[L_2],[L_3]\}$, where $L_1$ is the lattice stabilizing the quadruple $1234$, $L_2$ stabilizes the quadruple $3456$, and $L_3$ stabilizes the quadruples
\[
1345,~1346,~2345,~2346.
\]
Hence, $\mathbb{P}(\Sigma_{\pi_4(\mathbf{a})})_\Bbbk$ is as shown in Figure~\ref{fig:conic-in-case-20_forget4}, which is type II in Figure~\ref{fig:degenerations-for-5-points}. We have that three primary components of $\mathbb{P}(\Sigma_\mathbf{a})_\Bbbk$ are contracted under $\rho_4$, while $\rho_4(H)$ is the unique secondary component of $\mathbb{P}(\Sigma_{\pi_4(\mathbf{a})})_\Bbbk$, which is isomorphic to $\PP^1\times \PP^1$. Then, the curve $\overline{\rho_4(\mathcal{C})}_0$ is of type II as described by Proposition~\ref{prop:degenerate-5-pointed-conics-0-strata} with $i=1,j=2,k=5,\ell=6,m=3$. In an analogous way, we can describe the $5$-pointed curves $\overline{\rho_2(\mathcal{C})}_0$ (with markings $i=3,j=4,k=5,\ell=6,m=1$) and $\overline{\rho_6(\mathcal{C})}_0$ (with markings $i=1,j=2,k=3,\ell=4,m=5$). These three curves will be used in the considerations that follow.

Let $F\subseteq\PP(\Sigma_\mathbf{a})_\Bbbk$ be the irreducible component isomorphic to $\mathbb{F}_1$ mapping birationally to $T\cong\mathbb{P}^2$ under $\rho_4$. Since $\overline{\rho_4(\mathcal{C})}_0$ does not intersect $T$, we have that $\overline{\mathcal{C}}_0$ has empty intersection with $F$. 
The same argument using $\rho_2$ and $\rho_6$ shows that $\overline{\mathcal{C}}_0$ has empty intersection with each of the three primary components of $\PP(\Sigma_\mathbf{a})_\Bbbk$ isomorphic to $\mathbb{F}_1.$

Next, we describe $\overline{\mathcal{C}}_0|_H$. First, we prove that $\overline{\mathcal{C}}_0$ does not contain any of the six boundary curves $L$ of $H$. If $L$ corresponds to the $(-1)$-curve of an $\mathbb{F}_1$ component, then $L$ cannot be an irreducible component of $\overline{\mathcal{C}}_0$ as we already proved that $\overline{\mathcal{C}}_0$ has empty intersection with the $\mathbb{F}_1$ component. So, suppose that $L$ corresponds to the $(-1)$-curve of a $\mathrm{Bl}_2\mathbb{P}^2$ component. We denote these components by $P_{ij}$ and by $L_{ij}$ the $(-1)$-curve $H\cap P_{ij}$, where $p_i,p_j\in P_{ij}$. If $L_{56}\subseteq\overline{\mathcal{C}}_0$, then $\overline{\rho_4(\mathcal{C})}_0$ has to contain the curve $\rho_4(L_{56})$, which it does not. So, $L_{56}$ is not contained in $\overline{\mathcal{C}}_0$, and by an analogous argument also $L_{12}\nsubseteq\overline{\mathcal{C}}_0$. Considering $\rho_2$ instead, we also see that $L_{34}\nsubseteq\overline{\mathcal{C}}_0$. Now, let us consider the blow up $\rho_4|_H\colon H\rightarrow Q$, where $Q\cong\mathbb{P}^1\times\mathbb{P}^1$. Then, $\overline{\mathcal{C}}_0|_H$ coincides with the strict transform of the $(1,1)$-curve $\overline{\rho_4(\mathcal{C})}_0$ in $Q$. Hence, $\overline{\mathcal{C}}_0|_H$ is isomorphic to $\mathbb{P}^1$ and it intersects each $L_{12},L_{34},L_{56}$ transversely at one single point.

Lastly, we describe $\overline{\mathcal{C}}_0|_{P_{ij}}$. From the argument so far, we know that $\overline{\mathcal{C}}_0|_{P_{12}}$ does not have the three boundary curves of $P_{12}$ that are in the double locus as irreducible components. So, $\overline{\mathcal{C}}_0|_{P_{12}}$ coincides with the strict transform under the blow up $\rho_4|_{P_{12}}\colon P_{12}\rightarrow F$ of the curve $\overline{\rho_4(\mathcal{C})}_0|_F$, where $F$ is the irreducible component of $\mathbb{P}(\Sigma_{\pi_4(\mathbf{a})})_\Bbbk$ isomorphic to $\mathbb{F}_1$ containing the markings $\rho_4(p_1),\rho_4(p_2)$. As $\overline{\rho_4(\mathcal{C})}_0|_F$ is a $1$-section of the $\mathbb{P}^1$-fibration, then $\overline{\mathcal{C}}_0|_{P_{12}}$ is also isomorphic to $\mathbb{P}^1$ and passes through $p_1,p_2$. An analogous description holds for $\overline{\mathcal{C}}_0|_{P_{34}}$ and $\overline{\mathcal{C}}_0|_{P_{56}}$.

The discussion above gives the marked curve $(\overline{\mathcal{C}}_0;p_1,\ldots,p_6)$, which we picture in Figure~\ref{fig:curve-case-20}.

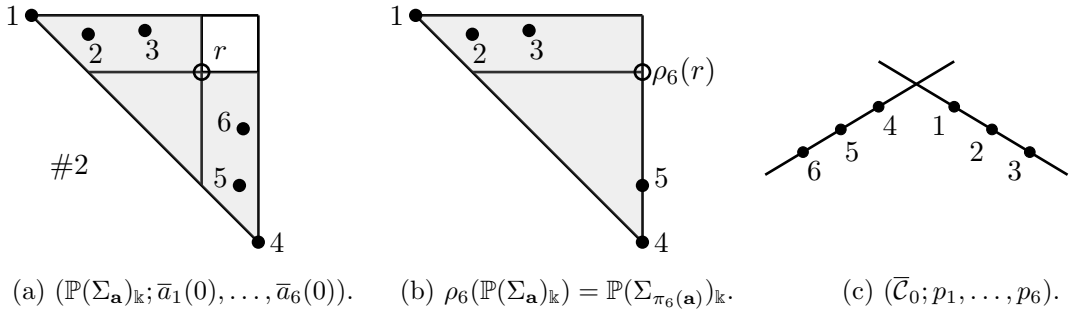
\begin{figure}[h]
    \centering
    \begin{subfigure}{.3\textwidth}
            \begin{tikzpicture}[scale=0.5]

	\draw[line width=1pt] (0,6) -- (6,6);
	\draw[line width=1pt] (6,6) -- (6,0);
	\draw[line width=1pt] (6,0) -- (0,6);
	\draw[line width=1pt] (4.5,6) -- (4.5,1.5);
	\draw[line width=1pt] (1.5,4.5) -- (6,4.5);
	\fill[gray!50,nearly transparent] (1.5,4.5) -- (4.5,4.5) -- (4.5,1.5) -- cycle;
	\fill[gray!50,nearly transparent] (1.5,4.5) -- (4.5,4.5) -- (4.5,6) -- (0,6) -- cycle;
	\fill[gray!50,nearly transparent] (4.5,1.5) -- (4.5,4.5) -- (6,4.5) -- (6,0) -- cycle;
	\fill (6,0) circle (5pt);
	\fill (0,6) circle (5pt);
	\fill (5.5,1.5) circle (5pt);
	\fill (5.6,3) circle (5pt);
	\fill (1.5,5.5) circle (5pt);
	\fill (3,5.6) circle (5pt);

	\node at (-.5,6) {$1$};
	\node at (1.7,5) {$2$};
	\node at (3.2,5.1) {$3$};
	\node at (6.5,0) {$4$};
	\node at (5,1.7) {$5$};
	\node at (5.1,3.2) {$6$};

	\node at (5,5) {$r$};
    \draw[line width=1pt] (4.5,4.5) circle (0.2cm);

	\node at (1,2) {\#2};

\end{tikzpicture}
\caption{$(\mathbb{P}(\Sigma_\mathbf{a})_\Bbbk;\overline{a}_1(0),\ldots,\overline{a}_6(0))$.}\label{fig:2}
    \end{subfigure}
    \begin{subfigure}{.3\textwidth}
            \begin{tikzpicture}[scale=0.5]

	\draw[line width=1pt] (0,6) -- (6,6);
	\draw[line width=1pt] (6,6) -- (6,0);
	\draw[line width=1pt] (6,0) -- (0,6);
\draw[line width=1pt] (1.5,4.5) -- (6,4.5);
	\fill[gray!50,nearly transparent] (0,6) -- (6,6) -- (6,0) -- cycle;
	\fill (6,0) circle (5pt);
	\fill (0,6) circle (5pt);
	\fill (6,1.5) circle (5pt);
	\fill (1.5,5.5) circle (5pt);
	\fill (3,5.6) circle (5pt);
    
	\node at (-.5,6) {$1$};
	\node at (1.7,5) {$2$};
	\node at (3.2,5.1) {$3$};
	\node at (6.5,0) {$4$};
	\node at (6.5,1.7) {$5$};

	\node at (7.1,4.5) {$\rho_6(r)$};
    \draw[line width=1pt] (6,4.5) circle (0.2cm);

\end{tikzpicture}
\caption{$\rho_6(\mathbb{P}(\Sigma_{\mathbf{a}})_\Bbbk)=\mathbb{P}(\Sigma_{\pi_6(\mathbf{a})})_\Bbbk$.}\label{fig:conic-in-case-2_forget6}
    \end{subfigure}
   \begin{subfigure}{.3\textwidth}
\begin{tikzpicture}[scale=0.5]

\draw [line width=1.pt] (0.,0.)-- (5.,3.);
\draw [line width=1.pt] (3.,3.)-- (8.,0.);
\draw (3.3,1.3) node {$4$};
\draw (4.6,1.3) node {$1$};
\draw (5.6,0.6) node {$2$};
\draw (6.6,0.1) node {$3$};
\draw (2.3,0.6) node {$5$};
\draw (1.3,0.1) node {$6$};

\draw (1.3,-2) node {};

\draw [fill=black] (2.,1.2) circle (4.0pt);
\draw [fill=black] (3.,1.8) circle (4.0pt);
\draw [fill=black] (1.,0.6) circle (4.0pt);
\draw [fill=black] (6.,1.2) circle (4.0pt);
\draw [fill=black] (5.,1.8) circle (4.0pt);
\draw [fill=black] (7.,0.6) circle (4.0pt);

\end{tikzpicture}
\caption{$(\overline{\mathcal{C}}_0;p_1,\ldots,p_6)$.}\label{fig:curve-case-2}
   \end{subfigure}

    \caption{Computation of the limit curve $\overline{\mathcal{C}}_0$ in type \#2.}
    \label{fig:conic-in-case-2}
\end{figure}

\textbf{Type \#2.} Consider $(\mathbb{P}(\Sigma_\mathbf{a})_\Bbbk;\overline{a}_1(0),\ldots,\overline{a}_6(0))$ as in Figure~\ref{fig:2}. We describe $\overline{\mathcal{C}}_0$ by computing its restriction to the four irreducible components of $\mathbb{P}(\Sigma_\mathbf{a})_\Bbbk=F_1\cup F_2\cup T\cup Q$, where $F_1,F_2\cong\mathbb{F}_1$ (with $p_1=\overline{a}_1(0)$ being on $F_1$), $T\cong\mathbb{P}^2$, and $Q\cong\mathbb{P}^1\times\mathbb{P}^1$. The common intersection point of the four components is denoted by $r$.

Let us consider $\rho_6\colon\mathbb{P}(\Sigma_\mathbf{a})\rightarrow\mathbb{P}(\Sigma_{\pi_6(\mathbf{a})})$. We have that $\Sigma_{\pi_6(\mathbf{a})}=\Sigma_{\mathbf{a}}\setminus\{[L]\}$, where $L$ is the lattice stabilizing the quadruples $1456$, $2456$, $3456$. Hence, $\mathbb{P}(\Sigma_{\pi_6(\mathbf{a})})_\Bbbk$ is of type I as in Figure~\ref{fig:degenerations-for-5-points}. In this case, the components $F_2,Q$ are each contracted to a curve, while $F_1$ and $P$ are mapped isomorphically. See Figure \ref{fig:conic-in-case-2_forget6}.

By Proposition~\ref{prop:degenerate-5-pointed-conics-0-strata}, we know that the curve $\overline{\rho_6(\mathcal{C})}_0$ can be described as follows: $\overline{\rho_6(\mathcal{C})}_0|_{\rho_6(F_1)}$ is the strict transform of a conic through the blow up $F_1\rightarrow\mathbb{P}^2$ passing through $\rho_6(p_1),\rho_6(p_2),\rho_6(p_3),\rho_6(r)$, and $\overline{\rho_6(\mathcal{C})}_0|_{\rho_6(T)}$ is the line passing through $\rho_6(r),\rho_6(p_4),\rho_6(p_5)$. As $\rho_6|_{F_1}$ is an isomorphism, this describes $\overline{\mathcal{C}}_0|_{F_1}$. Symmetrically we also obtain $\overline{\mathcal{C}}_0|_{F_2}$. From this it also follows that $\overline{\mathcal{C}}_0|_T=\{r\}$ because, for $i=1,2$, $\overline{\mathcal{C}}|_{F_i}$ intersects the exceptional curve in $F_i$ only at the point $r$ and $\overline{\mathcal{C}}|_T$ is entirely contained in $(T\cap F_1)\cup(T\cap F_2)$.

Lastly, let us prove that $\overline{\mathcal{C}}_0|_{Q}=\{r\}$. Suppose by contradiction that $\overline{\mathcal{C}}_0|_{Q}$ consists of at least two points. As we already know that one of these two points is $r$, let us denote the second one by $r'$. The forgetful morphisms $\rho_6$ (resp. $\rho_3$) contracts $Q$ along one of its two rulings, mapping $Q$ to a fiber of $\rho_6(F_1)\cong\mathbb{F}_1$ (resp. $\rho_3(F_2)\cong\mathbb{F}_1$). As $r\neq r'$,  we have $\rho_6(r)\neq\rho_6(r')$ or $\rho_3(r)\neq\rho_3(r')$. This is impossible by Remark~\ref{rmk:typeI-intersection-fiber-strict-transform-conic} because, in type I, a ruling of the $\mathbb{F}_1$ component intersects the strict transform of the conic in a unique point. So, $\overline{\mathcal{C}}_0|_{Q}$ only consists of the point $r$. This ultimately describes $(\overline{\mathcal{C}}_0;p_1,\ldots,p_6)$, which we picture in Figure~\ref{fig:curve-case-2}.

\begin{figure}[h]
    \centering
    \begin{subfigure}{.3\textwidth}
   \begin{tikzpicture}[scale=0.5]

	\draw[line width=1pt] (0,6) -- (6,6);
	\draw[line width=1pt] (6,6) -- (6,0);
	\draw[line width=1pt] (6,0) -- (0,6);
	\draw[line width=1pt] (18/4,6) -- (18/4,6/4);
	\draw[line width=1pt] (12/4,6) -- (12/4,12/4);
	\draw[line width=1pt] (6,6-6/4) -- (6/4,6-6/4);
	\fill[gray!50,nearly transparent] (6/4,6-6/4) -- (6,6-6/4) -- (6,0) -- cycle;
	\fill[gray!50,nearly transparent] (0,6) -- (12/4,6) -- (12/4,6-6/4) -- (6/4,6-6/4) --cycle;
	\fill (6,0) circle (5pt);
	\fill (0,6) circle (5pt);
	\fill (5.5,3.5) circle (5pt);
	\fill (1.8,5.2) circle (5pt);
	\fill (15/4,5.2) circle (5pt);
	\fill (5.5,1.8) circle (5pt);
    \draw[line width=1pt] (4.5,4.5) circle (0.2cm);

	\node at (5,3.5) {$4$};
	\node at (2.3,5.5) {$5$};
	\node at (-.5,6) {$6$};
	\node at (4.2,5.5) {$1$};
	\node at (5,1.8) {$2$};
	\node at (6.5,0) {$3$};

	\node at (1,2) {\#10};
	\node at (5,5) {$r$};
\end{tikzpicture}
\caption{$(\mathbb{P}(\Sigma_\mathbf{a})_\Bbbk;\overline{a}_1(0),\ldots,\overline{a}_6(0))$.}\label{fig:10}
    \end{subfigure}
    \begin{subfigure}{.3\textwidth}
            \begin{tikzpicture}[scale=0.5]

	\draw[line width=1pt] (0,6) -- (6,6);
	\draw[line width=1pt] (6,0) -- (6,6);
	\draw[line width=1pt] (0,6) -- (6,0);
	\draw[line width=1pt] (4,6) -- (4,2);
	\draw[line width=1pt] (6,4) -- (2,4);
    \fill[gray!50,nearly transparent] (0,6) -- (4,6) -- (4,4) -- (2,4) -- cycle;
    \fill[gray!50,nearly transparent] (6,0) -- (6,4) -- (4,4) -- (4,2) -- cycle;
    \fill[gray!50,nearly transparent] (2,4) -- (4,4) -- (4,2) -- cycle;
	\fill (0,6) circle (5pt);
	\fill (2.5,5) circle (5pt);
	\fill (6,0) circle (5pt);
	\fill (6,2.5) circle (5pt);
	\fill (5,5) circle (5pt);
    \draw[line width=1pt] (6,4) circle (0.2cm);

	\node at (0,6.7) {$6$};
	\node at (3,5) {$5$};
	\node at (6.7,0) {$3$};
	\node at (6.7,2.5) {$4$};
	\node at (5.5,5.5) {$1$};

	\node at (1,2) {};
 \node at (7.2,4) {$\rho_2(r)$};
\end{tikzpicture}
\caption{$\rho_2(\mathbb{P}(\Sigma_{\mathbf{a}})_\Bbbk)=\mathbb{P}(\Sigma_{\pi_2(\mathbf{a})})_\Bbbk$.}\label{fig:conic-in-case-10_forget2}
    \end{subfigure}
    \begin{subfigure}{.3\textwidth}
            \begin{tikzpicture}[scale=0.5]

	\draw[line width=1pt] (0,6) -- (6,6);
	\draw[line width=1pt] (6,6) -- (6,0);
	\draw[line width=1pt] (6,0) -- (0,6);
\draw[line width=1pt] (4.5,1.5) -- (4.5,6);
	\fill[gray!50,nearly transparent] (0,6) -- (6,6) -- (6,0) -- cycle;
	\fill (6,0) circle (5pt);
	\fill (0,6) circle (5pt);
	\fill (5.5,3.5) circle (5pt);
	\fill (1.8,6) circle (5pt);
	\fill (5.5,1.8) circle (5pt);
    \draw[line width=1pt] (4.5,6) circle (0.2cm);

	\node at (5,3.5) {$4$};
	\node at (1.8,6.7) {$1$};
	\node at (-.5,6) {$5$};
	\node at (5,1.8) {$2$};
	\node at (6.5,0) {$3$};

	\node at (4.5,7.0) {$\rho_6(r)$};
\end{tikzpicture}
\caption{$\rho_6(\mathbb{P}(\Sigma_{\mathbf{a}})_\Bbbk)=\mathbb{P}(\Sigma_{\pi_6(\mathbf{a})})_\Bbbk$.}\label{fig:conic-in-case-10_forget6}
    \end{subfigure}
   \begin{subfigure}{.3\textwidth}
\begin{tikzpicture}[scale=0.5]
\node at (7.5,-1) {};
\draw (7.5,-3.6) node {$1$};
\draw (9.86,-4.2) node {$2$};
\draw (10.44,-4.9) node {$3$};
\draw (11.,-5.6) node {$4$};
\draw (4.8,-4.5) node {$5$};
\draw (4.1,-5.4) node {$6$};
\draw [line width=1.pt, color = lightgray] (3.,-6.)-- (6.,-2.);
\draw [line width=1.pt] (9.,-2.)-- (12.,-6.);
\draw [line width=1.pt] (4.,-3.)-- (11.,-3.);

\draw [fill=black] (7.5,-3.) circle (4.0pt);
\draw [fill=black] (10.2888,-3.7184) circle (4.0pt);
\draw [fill=black] (10.875,-4.5) circle (4.0pt);
\draw [fill=black] (11.4375,-5.25) circle (4.0pt);
\draw [fill=black] (4.4832,-4.0224) circle (4.0pt);
\draw [fill=black] (3.756,-4.992) circle (4.0pt);

\end{tikzpicture}
\caption{$(\overline{\mathcal{C}}_0;p_1,\ldots,p_6)$.}\label{fig:curve-case-10}
   \end{subfigure}

    \caption{Computation of the limit curve $\overline{\mathcal{C}}_0$ in type \#10.}
    \label{fig:conic-in-case-10}
\end{figure}
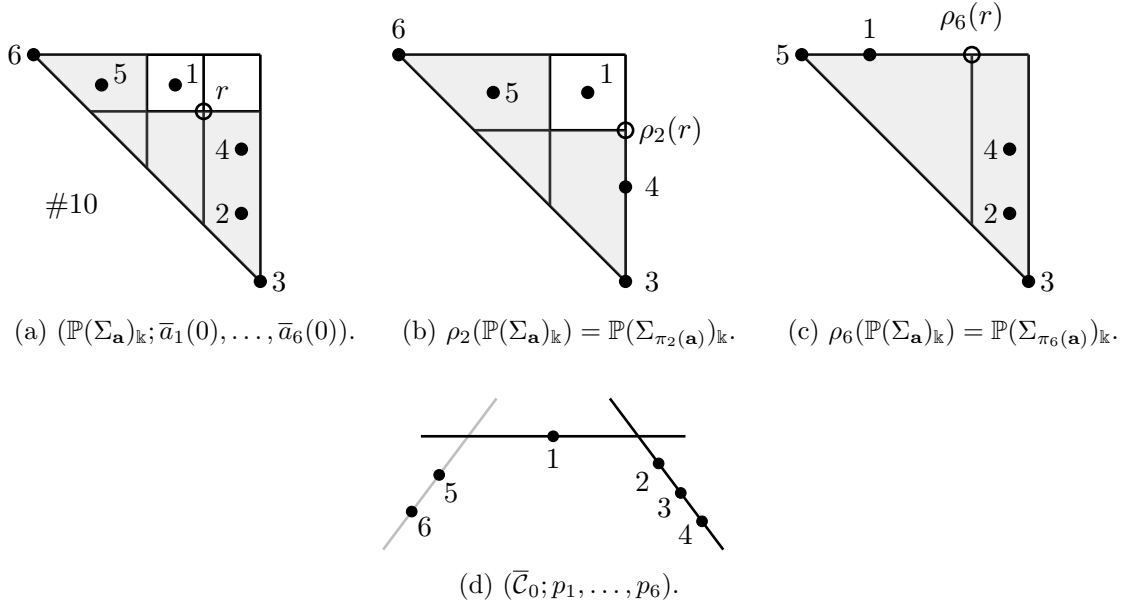

\textbf{Type \#10.} Consider $(\mathbb{P}(\Sigma_\mathbf{a})_\Bbbk;\overline{a}_1(0),\ldots,\overline{a}_6(0))$ as in Figure \ref{fig:10}. We describe $\overline{\mathcal{C}}_0$ by computing its restriction to the six irreducible components of $\mathbb{P}(\Sigma_\mathbf{a})_\Bbbk = T\cup Q_1\cup Q_2\cup F_1\cup F_2\cup F_3,$ where $T\cong\mathbb{P}^2$, $Q_i\cong\mathbb{P}^1\times\mathbb{P}^1$ for $i=1,2$, and $F_j\cong\mathbb{F}_1$, for $j=1,2,3$. Up to relabeling these components, we suppose that $p_1\in Q_1$, $p_5,p_6\in F_1$, $p_2,p_3,p_4\in F_2$, and $T,F_3,Q_2$ contain no markings. We denote the common intersection point of $Q_1,Q_2,F_2,F_3$ by $r$.

Let us consider $\rho_2\colon\mathbb{P}(\Sigma_\mathbf{a})\rightarrow\mathbb{P}(\Sigma_{\pi_2(\mathbf{a})})$. By forgetting $p_2$, $F_2$ is no longer a primary component, and $\rho_2$ maps $F_2$ to a $\mathbb{P}^1$ containing the points $\rho_2(p_3),\rho_2(p_4)$. The component $Q_2$ is contracted along with $F_2$, and the components $T,F_1,F_3,Q_1$ are mapped isomorphically to their images. Then, we obtain Figure~\ref{fig:conic-in-case-10_forget2}, which is of type II as in Figure~\ref{fig:degenerations-for-5-points}. So, by Proposition~\ref{prop:degenerate-5-pointed-conics-0-strata}, $\overline{\rho_2(\mathcal{C})}_0$ can be understood as follows:
\begin{itemize}

\item[(i)] $\overline{\rho_2(\mathcal{C})}_0|_{\rho_2(T)}$ is empty, which implies that also $\overline{\mathcal{C}}_0|_T$ is empty.

\item[(ii)] Under the blow up $\rho_2(F_1)\rightarrow\mathbb{P}^2$, $\overline{\rho_2(\mathcal{C})}_0|_{\rho_2(F_1)}$ is the strict transform of a line away from the blown up point. Also, $\overline{\rho_2(\mathcal{C})}_0|_{\rho_2(F_1)}$ and contains $\rho_2(p_5),\rho_2(p_6)$. This describes $\overline{\mathcal{C}}_0|_{F_1}$ via the isomorphism $F_1\cong\rho_2(F_1)$.

\item[(iii)] Under the blow up $\rho_2(F_3)\rightarrow\mathbb{P}^2$, $\overline{\rho_2(\mathcal{C})}_0|_{\rho_2(F_3)}$ is the strict transform of a line away from the blown up point. Also, $\overline{\rho_2(\mathcal{C})}_0|_{\rho_2(F_3)}$ contains $\rho_2(p_3),\rho_2(p_4),\rho_2(r)$. As $F_2$ is contracted by $\rho_2$ to $\overline{\rho_2(\mathcal{C})}_0|_{\rho_2(F_3)}$, we have that $\overline{\mathcal{C}}_0|_{F_3}$ is contained in $F_3\cap F_2$ (we will later show that $\overline{\mathcal{C}}_0|_{F_3}=\{r\}$.)

\item[(iv)] $\overline{\rho_2(\mathcal{C})}_0|_{\rho_2(Q_1)}$ is a smooth $(1,1)$-curve passing through $\rho_2(p_1)$ and $\rho_2(r)$. This describes $\overline{\mathcal{C}}_0|_{Q_1}$ via the isomorphism $Q_1\cong\rho_2(Q_1)$.

\end{itemize}

To characterize the restriction of $\overline{\mathcal{C}}_0$ to $F_2,F_3,Q_2$ we also need to consider the forgetful morphism $\rho_6\colon\mathbb{P}(\Sigma_\mathbf{a})\rightarrow\mathbb{P}(\Sigma_{\pi_6(\mathbf{a})})$. We have that $\rho_6$ contracts $Q_1,Q_2,T,F_1$, while $F_2$ is mapped isomorphically to its image and $\rho_6(F_3)\cong\mathbb{P}^2$. See Figure~\ref{fig:conic-in-case-10_forget6}. This is type I as in Figure~\ref{fig:degenerations-for-5-points}, and Proposition~\ref{prop:degenerate-5-pointed-conics-1-strata} gives $\overline{\rho_6(\mathcal{C})}_0$.
In particular, $\overline{\rho_6(\mathcal{C})}_0|_{\rho_6(F_2)}$ is the strict transform of a conic passing through $p_2,p_3,p_4$ and intersects the exceptional divisor of $\rho_6(F_2)$ in $\rho_6(r)$. This describes $\overline{\mathcal{C}}_0|_{F_2}$ via the isomorphism $F_2\cong\rho_6(F_2)$, and combined with (iii) above gives that $\overline{\mathcal{C}}_0|_{F_3}=\{r\}$.

Finally, we have that $\overline{\mathcal{C}}_0|_{Q_2}=\{r\}$. The proof of this is analogous to the one for type~\#2. The marked curve $(\overline{\mathcal{C}}_0;p_1,\ldots,p_6)$ is in Figure~\ref{fig:curve-case-10}.

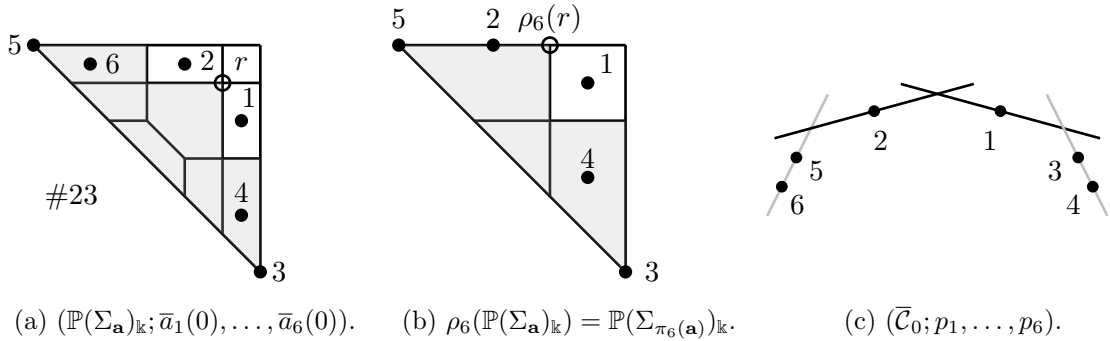
\begin{figure}[h]
    \centering
    \begin{subfigure}{.3\textwidth}
\begin{tikzpicture}[scale=0.5]

	\draw[line width=1pt] (0,6) -- (6,6);
	\draw[line width=1pt] (6,6) -- (6,0);
	\draw[line width=1pt] (6,0) -- (0,6);
	\draw[line width=1pt] (1,5) -- (6,5);
	\draw[line width=1pt] (5,6) -- (5,1);
	\draw[line width=1pt] (2,4) -- (3,4);
	\draw[line width=1pt] (3,6) -- (3,4);
	\draw[line width=1pt] (3,4) -- (4,3);
	\draw[line width=1pt] (4,3) -- (4,2);
	\draw[line width=1pt] (4,3) -- (6,3);
	\fill[gray!50,nearly transparent] (0,6) -- (3,6) -- (3,4) -- (2,4) -- (1,5) -- cycle;
	\fill[gray!50,nearly transparent] (6,0) -- (6,3) -- (4,3) -- (4,2) -- (5,1) -- cycle;
	\fill[gray!50,nearly transparent] (3,4) -- (3,5) -- (5,5) -- (5,3) -- (4,3) -- cycle;
	\fill[gray!50,nearly transparent] (2,4) -- (3,4) -- (4,3) -- (4,2) -- cycle;
	\fill (6,0) circle (5pt);
	\fill (0,6) circle (5pt);
	\fill (5.5,4) circle (5pt);
	\fill (4,5.5) circle (5pt);
	\fill (1.5,5.5) circle (5pt);
	\fill (5.5,1.5) circle (5pt);
    \draw[line width=1pt] (5,5) circle (0.2cm);

	\node at (2.1,5.5) {$6$};
	\node at (-.5,6) {$5$};
	\node at (4.6,5.6) {$2$};
	\node at (6.5,0) {$3$};
	\node at (5.7,4.6) {$1$};
	\node at (5.5,2.1) {$4$};
    \node at (5.5,5.5) {$r$};

	\node at (1,2) {\#23};

\end{tikzpicture}
\caption{$(\mathbb{P}(\Sigma_\mathbf{a})_\Bbbk;\overline{a}_1(0),\ldots,\overline{a}_6(0))$.}\label{fig:23}
    \end{subfigure}
    \begin{subfigure}{.3\textwidth}
            \begin{tikzpicture}[scale=0.5]

	\draw[line width=1pt] (0,6) -- (6,6);
	\draw[line width=1pt] (6,0) -- (6,6);
	\draw[line width=1pt] (0,6) -- (6,0);
	\draw[line width=1pt] (4,6) -- (4,2);
	\draw[line width=1pt] (6,4) -- (2,4);
    \fill[gray!50,nearly transparent] (0,6) -- (4,6) -- (4,4) -- (2,4) -- cycle;
    \fill[gray!50,nearly transparent] (6,0) -- (6,4) -- (4,4) -- (4,2) -- cycle;
    \fill[gray!50,nearly transparent] (2,4) -- (4,4) -- (4,2) -- cycle;
	\fill (0,6) circle (5pt);
	\fill (5,2.5) circle (5pt);
	\fill (6,0) circle (5pt);
	\fill (2.5,6) circle (5pt);
	\fill (5,5) circle (5pt);
    \draw[line width=1pt] (4,6) circle (0.2cm);

	\node at (6.7,0) {$3$};
	\node at (5,3) {$4$};
	\node at (0,6.7) {$5$};
	\node at (2.5,6.7) {$2$};
	\node at (5.5,5.5) {$1$};

	\node at (1,2) {};
 \node at (4,6.7) {$\rho_6(r)$};
\end{tikzpicture}
\caption{$\rho_6(\mathbb{P}(\Sigma_{\mathbf{a}})_\Bbbk)=\mathbb{P}(\Sigma_{\pi_6(\mathbf{a})})_\Bbbk$.}\label{fig:conic-in-case-23_forget6}
    \end{subfigure}
   \begin{subfigure}{.3\textwidth}
\begin{tikzpicture}[scale=0.5]

\draw (11.1,-5.7) node {$4$};
\draw (8.9,-4.0) node {$1$};
\draw (6.0,-4.0) node {$2$};
\draw (10.62,-4.8) node {$3$};
\draw (4.3,-4.8) node {$5$};
\draw (3.8,-5.7) node {$6$};
\draw (3.8,-7.7) node {};
\draw [color = lightgray,line width=1.pt] (3.,-6.)-- (4.605108332342063,-2.7897833353158727);
\draw [line width=1.pt] (3.1931309156043723,-3.947327932107899)-- (8.464763821237732,-2.5096098669351647);
\draw [line width=1.pt] (6.535236178762268,-2.5096098669351647)-- (11.806869084395656,-3.947327932107905);
\draw [color= lightgray,line width=1.pt] (10.394891667657935,-2.7897833353158736)-- (12.,-6.);

\draw [fill] (5.828947368421053,-3.2284688995215314) circle (4.0pt);
\draw [fill] (9.171052631578949,-3.2284688995215314) circle (4.0pt);
\draw [fill] (3.771929824561402,-4.456140350877196) circle (4.0pt);
\draw [fill] (3.3859649122807,-5.2280701754386) circle (4.0pt);
\draw [fill] (11.228070175438598,-4.456140350877197) circle (4.0pt);
\draw [fill] (11.6140350877193,-5.228070175438599) circle (4.0pt);

\end{tikzpicture}
\caption{$(\overline{\mathcal{C}}_0;p_1,\ldots,p_6)$.}\label{fig:curve-case-23}
   \end{subfigure}

    \caption{Computation of the limit curve $\overline{\mathcal{C}}_0$ in type \#23.}
    \label{fig:conic-in-case-23}
\end{figure}

\textbf{Type \#23.} Consider $(\mathbb{P}(\Sigma_\mathbf{a})_\Bbbk;\overline{a}_1(0),\ldots,\overline{a}_6(0))$ as in Figure~\ref{fig:23}. We describe $\overline{\mathcal{C}}_0$ by computing its restriction to the nine irreducible components of $\mathbb{P}(\Sigma_\mathbf{a})_\Bbbk = P\cup\bigcup_{i=1}^3 Q_i\cup\bigcup_{j=1}^5 F_j,$ where $P\cong \mathrm{Bl}_2(\mathbb{P}^2)$, $Q_i\cong\mathbb{P}^1\times\mathbb{P}^1$ for $i=1,2,3$, and $F_j\cong\mathbb{F}_1$, for $j=1,\ldots,5$. Up to relabeling of these components, suppose that $p_1\in Q_1$, $p_2\in Q_2$, $p_5,p_6\in F_1$, $p_3,p_4\in F_5$, and $P,Q_3,F_2,F_3,F_4$ contain no markings. We denote by $r$ the common intersection point of $Q_1,Q_2,Q_3,P$. The component $F_2$ intersects $F_1$, while $F_4$ intersects $F_5$.

Let us consider $\rho_6\colon\mathbb{P}(\Sigma_\mathbf{a})\rightarrow\mathbb{P}(\Sigma_{\pi_6(\mathbf{a})})$. By forgetting $p_6$, we have that $F_1,F_2,F_3$ are no longer primary components, and $\rho_6$ contracts $F_1,F_2$ to a point and $F_3$ to a curve contained in $\rho_6(P)$. $Q_2$ is contracted along a ruling to a $\mathbb{P}^1$ and the same happens for $Q_3$. On the other hand, $Q_1$ and $F_5$ are mapped isomorphically to their images, $\rho_6(P)\cong\mathbb{F}_1$ with $\rho_6(p_2)$ and $\rho_6(p_5)$ in its boundary, and $\rho_6(F_4)\cong\mathbb{P}^2$. The resulting $5$-pointed Mustafin degeneration is shown in  Figure~\ref{fig:conic-in-case-23_forget6}, which is of type II as in Figure~\ref{fig:degenerations-for-5-points}.

In particular, we know $\overline{\rho_6(\mathcal{C})}_0$ by Proposition~\ref{prop:degenerate-5-pointed-conics-0-strata}, from which we can argue the following:
\begin{itemize}

\item[(i)] $\overline{\rho_6(\mathcal{C})}_0|_{\rho_6(F_4)}$ is empty, hence also $\overline{\mathcal{C}}_0|_{F_4}$ is empty.

\item[(ii)] Under the blow up $\rho_6(F_5)\rightarrow\mathbb{P}^2$, $\overline{\rho_6(\mathcal{C})}_0|_{\rho_6(F_5)}$ is the strict transform of a line away from the blown up point. Also, $\overline{\rho_6(\mathcal{C})}_0|_{\rho_6(F_5)}$ contains $\rho_6(p_3),\rho_6(p_4)$. This describes $\overline{\mathcal{C}}_0|_{F_5}$ via the isomorphism $F_5\cong\rho_6(F_5)$.

\item[(iii)] $\overline{\rho_6(\mathcal{C})}_0|_{\rho_6(Q_1)}$ is a smooth $(1,1)$-curve passing through $\rho_6(p_1),\rho_6(r)$. This describes $\overline{C}_0|_{Q_1}$ via the isomorphism $Q_1\cong\rho_6(Q_1)$ as a smooth $(1,1)$ curve passing through $p_1,r$.

\item[(iv)] Under the blow up $\rho_6(P)\rightarrow\mathbb{P}^2$, $\overline{\rho_6(\mathcal{C})}_0|_{\rho_6(P)}$ is the strict transform of a line away from the blown up point, and it passes through $\rho_6(p_2),\rho_6(p_5),\rho_6(r)$.  This curve equals $\rho_6(Q_2)$, hence $\overline{\mathcal{C}}_0|_P\subseteq P\cap Q_2$ (we will later see that $\overline{\mathcal{C}}_0|_{P}=\{r\}$). Moreover, as $\rho_6(F_3)$ intersects $\overline{\rho_6(\mathcal{C})}_0|_{\rho_6(P)}$ only at $\rho_6(p_5) = \rho_6(F_2)$,  we know $\overline{\mathcal{C}}_0|_{F_3}$ is contained in the fiber $F_3\cap F_2$ that is mapped to $\rho_6(p_5)$ under the contraction $F_3\rightarrow\rho_6(F_3)$. 

\end{itemize}
By exploiting the combinatorial symmetry of $(\mathbb{P}(\Sigma_\mathbf{a})_\Bbbk;\overline{a}_1(0),\ldots,\overline{a}_6(0))$, and considering instead the forgetful morphism $\rho_4$, we conclude that
\begin{itemize}

\item[(i)] $\overline{\mathcal{C}}_0|_{F_2}$ is empty.

\item[(ii)] Under the blow up $F_1\rightarrow\mathbb{P}^2$, $\overline{\mathcal{C}}_0|_{F_1}$ is the strict transform of a line away from the blown up point containing $p_5,p_6$.

\item[(iii)] $\overline{\mathcal{C}}_0|_{Q_2}$ is a smooth $(1,1)$ curve passing through $p_2,r$.

\item[(iv)] $\overline{\mathcal{C}}_0|_P\subseteq P\cap Q_1$, which combined with (iv) above implies that $\overline{\mathcal{C}}_0|_P=\{r\}$. We also have that, $\overline{\mathcal{C}}_0|_{F_3}$ is contained in the fiber $F_3\cap F_4$ that is mapped to $\rho_4(p_3)$ under the contraction $F_3\rightarrow\rho_4(F_3)$. But from (iv) above we also know that $\overline{\mathcal{C}}_0|_{F_3}$ is contained in $F_3\cap F_2$, which is disjoint from $F_3\cap F_4$. In conclusion, $\overline{\mathcal{C}}_0|_{F_3}$ is empty.

\end{itemize}
Finally, we have that $\overline{\mathcal{C}}_0|_{Q_3}=\{r\}$. The proof of this is analogous to the one for type~\#2. We represent the marked curve $(\overline{\mathcal{C}}_0;p_1,\ldots,p_6)$ in Figure~\ref{fig:curve-case-23}.

\textbf{Types \#12.1, \#20.2, \#20.3, \#20.6, \#23.1.}
Let us consider $(X;p_1,\ldots,p_6)$ of type \#12.1 (an analogous argument holds for \#20.2, \#20.3, \#20.6, \#23.1). Let $(Y;q_1,\ldots,q_6)$ be a $6$-pointed Mustafin degeneration of type \#12 with the markings chosen so that $(Y;q_1,\ldots,q_6)$ can further degenerate to $(X;p_1,\ldots,p_6)$ (see \#12 and \#12.1 in Tables~\ref{tbl:GPdegenerationsself-dual-1} and \ref{tbl:GPdegenerationsself-dual-2}). We note that, in the specialization of $Y$ to $X$, one irreducible component $F\subseteq Y$ isomorphic to $\mathbb{F}_1$ breaks into $Q\cup T\subseteq X$, where $Q\cong\mathbb{P}^1\times\mathbb{P}^1$ is glued transversely to $T\cong\mathbb{P}^2$ along a ruling of the former and a line of the latter. We call $F$ a \emph{breaking component} of $Y$.

Now, we can reconstruct $\overline{\mathcal{C}}_0\subseteq X$ in the same way we did for \#12 in this proof. In particular, when considering the forgetful morphisms $\rho_i$, $i\in\{1,\ldots,6\}$, we have that $\rho_i(X)\cong\rho_i(Y)$ and $\rho_i(Q\cup T)=\rho_i(F)$, where one component of either $Q$ or $P$ in $X$ is mapped to the image of the breaking component $F$ of $Y$, while the other is mapped into its boundary. We note that $\overline{\mathcal{C}}_0$ has empty intersection with $Q\cup T$, as it was the case for $F$ in $Y$.

\medskip

Now that we know that $\overline{\mathcal{C}}_0$ is reduced and one-dimensional in all cases, we can conclude that $\overline{\mathcal{C}}$ is flat over $\mathrm{Spec}(R)$ by \cite[Lemma~10.12]{HKT09}. Finally, we need to argue that the isomorphism class of the tuple
\[
(\mathbb{P}(\Sigma_{\mathbf{a}})_\Bbbk,\overline{\mathcal{C}}_0;\overline{a}_1(0),\ldots,\overline{a}_6(0))
\]
is independent from the choice of $\mathbf{a}$ and only depends on the point $x\in\overline{\mathbf{S}}_{\mathrm{GP}}(3,6)$. This is true for $(\mathbb{P}(\Sigma_{\mathbf{a}})_\Bbbk;\overline{a}_1(0),\ldots,\overline{a}_6(0))$ as $x$ parametrizes its isomorphism class, so we only have to consider the curve $\overline{\mathcal{C}}_0$. For $i\in\{1,\ldots,6\}$, let $f_i\colon\overline{\mathbf{X}}_{\mathrm{GP}}(3,6)\rightarrow\overline{\mathbf{X}}_{\mathrm{GP}}(3,5)$ be the morphism forgetting the $i$-th label \cite[Proposition~7.10]{ST22}. As seen in the proof above, the conic $\overline{\mathcal{C}}_0$ was constructed from the curves $\overline{\rho_i(\mathcal{C})}_0$ for appropriate choices of $i\in\{1,\ldots,6\}$. We have that $\overline{\rho_i(\mathcal{C})}_0$ only depends on $f_i(x)$ and not on $\pi_i(\mathbf{a})$ by Propositions~\ref{prop:degenerate-5-pointed-conics-1-strata} and \ref{prop:degenerate-5-pointed-conics-0-strata}. Hence, also the curve $\overline{\mathcal{C}}_0$ only depends on $x$ and not on $\mathbf{a}$.
\end{proof}

We finish this section with an example of one-parameter family $\mathbf{a}=(a_1,\ldots,a_6)$ of six self-dual points where we compute $(\mathbb{P}(\Sigma_{\mathbf{a}})_\Bbbk,\overline{\mathcal{C}}_0;\overline{a}_1(0),\ldots,\overline{a}_6(0))$ via its defining polynomial equations, matching the description in the proof of Theorem~\ref{thm:1parameterfamilyconics} for the corresponding case.

\begin{example}\label{ex:deg_conic}
Let us consider the arc
\[
\mathbf{a}(t)=([1:0:0],[0:1:0],[t:1:1],[1:t:1],[4:t:2],[t:4:2]).
\]
For general $t$, these points are in general linear position. Further, they are contained in the smooth conic $\mathcal{C}=V(x_1x_2-tx_3^2)\subseteq\mathbb{P}_K^2$. In what follows, we compute $(\mathbb{P}(\Sigma_{\mathbf{a}})_\Bbbk,\overline{\mathcal{C}}_0;\overline{a}_1(0),\ldots,\overline{a}_6(0))$ explicitly.

We start by computing $\mathbb{P}(\Sigma_{\mathbf{a}})_\Bbbk$. The stable lattice classes $\Sigma_{\mathbf{a}}=\{[L_1],[L_2],[L_3]\}$ are as follows:
\begin{itemize}

\item $L_1=e_1R+e_2R+e_3R$ stabilizes the quadruples
\[
1234,1235,1246,1256,1346,1356,2345,2456,3456.
\]

\item $L_2=te_1R+e_2R+e_3R$ stabilizes $1236,2346,2356$.

\item $L_3=e_1R+te_2R+e_3R$ stabilizes $1245,1345,1456$.

\end{itemize}
We compute the Mustafin variety $\mathbb{P}(\Sigma_{\mathbf{a}})$ following \cite[Example~2.2]{CHSW11}. Let $[x_{1j}:x_{2j}:x_{3j}]$ be the coordinates of the $j$-th factor of $(\mathbb{P}_K^2)^3$ and let $I_2$ be the ideal of generated by the $2\times2$ minors of the matrix
$$ \begin{pmatrix}
    x_{11} & tx_{12} & x_{13}\\
    x_{21} & x_{22} & t x_{23}\\
    x_{31} & x_{32} & x_{33}
\end{pmatrix}.$$
The Mustafin variety $\PP(\Sigma_{\mathbf{a}})$ is the closed subscheme of $ \PP(L_1)\times_R\PP(L_2)\times_R\PP(L_3)$ associated to the ideal $I_2\cap R[x_{ij}\mid i,j=1,2,3]$. The special fiber $\mathbb{P}(\Sigma_{\mathbf{a}})_\Bbbk$ is the subscheme of $(\PP_\Bbbk^2)^3$ defined by the initial ideal of $(I_2:\langle t\rangle)^\infty$ with respect to $t$ (\cite[Section~2.4]{MS15}). Using \texttt{OSCAR} \cite{OSCAR,OSCAR-book}, we compute its primary decomposition:
\begin{align*}
\langle x_{22},x_{32},&x_{13},x_{33}\rangle \cap \langle x_{11},x_{33},x_{13},x_{21}x_{32}-x_{31}x_{22} \rangle\,\cap  \\
&\langle x_{21},x_{32},x_{22},x_{11}x_{33}-x_{31}x_{13} \rangle\cap \langle x_{21},x_{11},x_{22},x_{13}\rangle\cap R[x_{ij}].
\end{align*}
Let $T,F_1,F_2,Q$ be the components of $\mathbb{P}(\Sigma_{\mathbf{a}})_\Bbbk$ corresponding to the above ideals, respectively. Then, $T\cong\mathbb{P}^2$, $F_1,F_2\cong\mathbb{F}_1$, and $Q\cong\mathbb{P}^1\times\mathbb{P}^1$. The surfaces are glued as showed in Figure~\ref{fig:example_marked_conic}, where we also give a parametrization for these surfaces. From the classification of the Mustafin degenerations parametrized by $\overline{\mathbf{S}}_{\mathrm{GP}}(3,6)^\nu$ in Proposition~\ref{prop:strataofselfdualcompactification}, we can conclude that $(\mathbb{P}(\Sigma_{\mathbf{a}});\overline{a}_1(0),\ldots,\overline{a}_6(0))$ is of type \#2. In particular, three limit marked points lie on $F_1$ and the remaining three lie on $F_2$.

Next, we determine which three markings are contained in $F_1$ and $F_2$. From its defining equations, we have that $F_1$ (resp. $F_2$) dominates the $\mathbb{P}^2$ component of coordinates $[x_{12}:x_{22}:x_{32}]$ (resp. $[x_{13}:x_{23}:x_{33}]$), hence $F_1$ corresponds to $[L_2]$ and $F_2$ corresponds to $[L_3]$. By taking into account the quadruples being stabilized by $[L_2]$ and $[L_3]$, we can conclude that $\overline{a}_2(0),\overline{a}_3(0),\overline{a}_6(0)$ lie on $F_1$ while $\overline{a}_1(0),\overline{a}_4(0),\overline{a}_5(0)$ lie on $F_2$.

Finally, we compute the limit in $\mathbb{P}(\Sigma_{\mathbf{a}})_\Bbbk$ of the one-parameter family of conics $x_1x_2-tx_3^2=0$. To do this, we add to $I_2$ the ideal $\langle x_{1j}x_{2j}-tx_{3j}^2\mid j=1,2,3\rangle$ of the diagonally embedded conic, and then saturate and take the initial ideal both with respect to $t$. Using \texttt{OSCAR}, we obtain the following primary decomposition after removing the irrelevant ideal:
\begin{align*}
\langle x_{33}, x_{13}, x_{12}x_{22} - &x_{32}^2, x_{21}x_{32} - x_{31}x_{22}, x_{21}x_{12} - x_{31}x_{32}, x_{11}\rangle\,\cap\\
&\langle x_{13}x_{23} - x_{33}^2, x_{32}, x_{22}, x_{21}, x_{11}x_{33} - x_{31}x_{13}, x_{11}x_{23} - x_{31}x_{33} \rangle.
\end{align*}
The associated closed subset of $\PP(L_1)_\Bbbk\times\PP(L_2)_\Bbbk\times \PP(L_3)_\Bbbk$ in parametric form is given by
\begin{align*}
\{([0:x_{21}:x_{31}],&[x_{31}^2:x_{21}^2:x_{21}x_{31}],[0:1:0]) \mid [x_{21}:x_{31}]\in\PP^1\}\,\cup\\
&\{([x_{11}:0:x_{31}],[1:0:0],[x_{11}^2:x_{31}^2:x_{11}x_{31}])  \mid [x_{11}:x_{31}]\in\PP^1 \}.
\end{align*}
Hence, the limit $\overline{\mathcal{C}}_0$ has two irreducible components contained in $F_1$ and $F_2$, respectively, and with unique intersection point the quadruple point of the surface $([0:0:1],[1:0:0],[0:1:0])$. We sketch this pointed curve on the degenerate surface in Figure~\ref{fig:example_marked_conic} (in general, such pictures become quickly convoluted, which is why we represented them ``outside'' the surface in Figures~\ref{fig:conic-in-case-5}--\ref{fig:conic-in-case-23} and Tables~\ref{tbl:GPdegenerationsself-dual-1} and \ref{tbl:GPdegenerationsself-dual-2}.) A Jupyter notebook showcasing the code used for this specific example is available on GitHub:
\begin{center}
{
\small
\url{https://github.com/lucaschaffler/SelfDuality_MustafinDegenerations}
}
\end{center}
\end{example}

\begin{figure}
    \begin{center}
\begin{tikzpicture}[scale=0.55]
	\draw[line width=1pt] (0,6) -- (6,6);
	\draw[line width=1pt] (6,6) -- (6,0);
	\draw[line width=1pt] (6,0) -- (0,6);
	\draw[line width=1pt] (4.5,6) -- (4.5,1.5);
	\draw[line width=1pt] (1.5,4.5) -- (6,4.5);
	\fill[gray!50,nearly transparent] (1.5,4.5) -- (4.5,4.5) -- (4.5,1.5) -- cycle;
	\fill[gray!50,nearly transparent] (1.5,4.5) -- (4.5,4.5) -- (4.5,6) -- (0,6) -- cycle;
	\fill[gray!50,nearly transparent] (4.5,1.5) -- (4.5,4.5) -- (6,4.5) -- (6,0) -- cycle;
						
	\node at (-0.4,6) {$2$};
	\node at (1.8,5.1) {$3$};
	\node at (3.0,5.1) {$6$};
	\node at (6.4,0) {$1$};
	\node at (5.1,1.6) {$4$};
	\node at (5.2,3.2) {$5$};
						
	\node at (-7.25,5) {$L_2 : \big\{([0:x_{21}:x_{31}],[x_{12}:x_{22}:x_{32}],[0:1:0])\big\}$};
	\node at (-5.2,3.25) {$L_1 :\big\{([x_{11}:x_{21}:x_{31}],[1:0:0],[0:1:0])\big\}$};
	\node at (-3.5,0.75) {$L_3:\big\{([x_{11}:0:x_{31}],[1:0:0],[x_{13}:x_{23}:x_{33}])\big\}$};
	\node at (-3.5,2) {$\big\{([0:0:1],[x_{12}:0:x_{32}],[0:x_{23}:x_{33}])\big\}$};
						
	\draw[lightgray, dashed,line width=0.5pt] (2.4,4.8) -- (0.25,4.8);
	\draw[lightgray, dashed,line width=0.5pt] (1.7,3.05) -- (3.4,3.05); 
	\draw[lightgray, dashed,line width=0.5pt] (4,0.55) -- (5.65,0.55);
						
	\draw[lightgray, dashed,line width=0.5pt] (3.2,1.8) to[out=45, in=245] (5.25,5.25);
						
	\draw[magenta,line width=1.5pt] (0,6) to[out=355, in=180] (1.5,5.5) to[out=45, in=160] (3,5.6) to[out=350, in=180] (4.5,4.5) to[out=90, in=355] (4,5);
	\draw[magenta,line width=1.5pt] (5,4) to[out=160, in=270] (4.5,4.5) to[out=0, in=75] (5.6,3) to[out=255, in=115] (5.5,1.5)  to[out=270, in=90] (6,0);
						
	\fill (6,0) circle (4pt);
	\fill (0,6) circle (4pt);
	\fill (5.5,1.5) circle (4pt);
	\fill (5.6,3) circle (4pt);
	\fill (1.5,5.5) circle (4pt);
	\fill (3,5.6) circle (4pt);
\end{tikzpicture}
\end{center}
\caption{The limit $(\mathbb{P}(\Sigma_{\mathbf{a}}),\overline{\mathcal{C}}_0;\overline{a}_1(0),\ldots,\overline{a}_6(0))$ in Example~\ref{ex:deg_conic}. The two components of the curve $\overline{\mathcal{C}}_0$ intersect transversely at the quadruple point of the surface.}
\label{fig:example_marked_conic}
\end{figure}
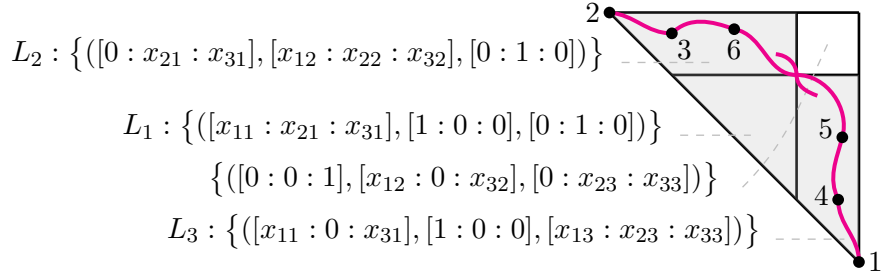

\section{\texorpdfstring{$\overline{\mathbf{S}}_{\mathrm{GP}}(3,6)$}{barSGP(3,6)} and the global family of \texorpdfstring{$6$}{6}-pointed degenerate conics}
\label{sec:construction-of-the-global-family-of-degenerate-conics}

Let us consider the family $\overline{\mathcal{M}}(3,n)\rightarrow\overline{\mathbf{X}}_{\mathrm{GP}}(3,n)$ of Mustafin degenerations of $\mathbb{P}^2$ with $n$ sections as introduced in Theorem~\ref{thm:main-on-XbarGP(3,n)}. For $n=5$, we showed in Section~\ref{sec:Mustafin-comp-for-5-points} the existence of an additional family of conics $\overline{\mathcal{D}}\rightarrow\overline{\mathbf{X}}_{\mathrm{GP}}(3,5)$ with sections $s_1,\ldots,s_5$ such that the fibers of $(\overline{\mathcal{M}}(3,5);\overline{\mathcal{D}})\rightarrow\overline{\mathbf{X}}_{\mathrm{GP}}(3,5)$ are the $5$-pointed Mustafin degenerations of $\mathbb{P}^2$ together with a degenerate conic interpolating them. A similar construction should hold for $\overline{\mathbf{S}}_{\mathrm{GP}}(3,6)$: The goal of this section is to realize a proper flat family of $6$-pointed degenerate conics over the normalization of $\overline{\mathbf{S}}_{\mathrm{GP}}(3,6)$.

\begin{theorem}
\label{thm:main-theorem-about-family-of-degenerate-conics}
There exists a proper flat family of surfaces $\overline{\mathcal{M}}^{\mathrm{sd}}(3,6)\rightarrow\overline{\mathbf{S}}_{\mathrm{GP}}(3,6)^\nu$ with a divisor $\overline{\mathcal{D}}^{\mathrm{sd}}(3,6)\subseteq\overline{\mathcal{M}}^{\mathrm{sd}}(3,6)$ relative over $\overline{\mathbf{S}}_{\mathrm{GP}}(3,6)^\nu$ together with disjoint sections $s_i^{\mathrm{sd}}\colon\overline{\mathbf{S}}_{\mathrm{GP}}(3,6)^\nu\rightarrow\overline{\mathcal{D}}^{\mathrm{sd}}(3,6)$, $i=1,\ldots,6$, such that, over a point $x$, the fiber
\[
(\overline{\mathcal{M}}^{\mathrm{sd}}(3,6)_x,\overline{\mathcal{D}}^{\mathrm{sd}}(3,6)_x;s_1^{\mathrm{sd}}(x),\ldots,s_6^{\mathrm{sd}}(x))=(X,C;p_1,\ldots,p_6)
\]
is described as follows:

\begin{itemize}

\item If $x\in\mathbf{S}(3,6)$, then $X=\mathbb{P}^2$ and $C$ is the smooth conic passing through $p_1,\ldots,p_6$;

\item Otherwise, $(X;p_1,\ldots,p_6)$ is a $6$-pointed Mustafin degeneration of $\mathbb{P}^2$ and $C\subseteq X$ is a reduced curve interpolating $p_1,\ldots,p_6$ such that $(C;p_1,\ldots,p_6)$ is a stable $6$-pointed rational curve. The possibilities for $(X,C;p_1,\ldots,p_6)$ are listed in Tables~\ref{tbl:GPdegenerationsself-dual-1} and \ref{tbl:GPdegenerationsself-dual-2}.
\end{itemize}
\end{theorem}

In order to prove this theorem, we first consider the general case of $n$ points in $\mathbb{P}^2$, and refine the construction of $\overline{\mathbf{X}}_{\mathrm{GP}}(3,n)$ to allow over it a family of degenerate conics containing the first five marked points.

\subsection{The construction of \texorpdfstring{$\overline{\mathbf{X}}_{\mathrm{GP}}(3,6)$}{barXGP(3,6)}}

We recall the construction of $\overline{\mathbf{X}}_{\mathrm{GP}}(3,n)$ from \cite[Section~7]{ST22}. Given $n$ linearly general points $p_1,\ldots,p_n$ in $\mathbb{P}^2$ and indices $\{i_1<\ldots<i_4\}\in\binom{[n]}{4}$, define $\psi_{i_1,\ldots,i_4}\in\mathrm{PGL}_3$ to be the unique projective linear transformation sending $p_{i_1},\ldots,p_{i_4}$ to $[1:0:0],[0:1:0],[0:0:1],[1:1:1]$, respectively. Denoting a point on $\mathbf{B}(3,n)$ as an equivalence class $[p_1,\ldots,p_n]$, we can define the embedding
\begin{align*}
\iota\colon\mathbf{B}(3,n)&\hookrightarrow\prod_{\substack{v~\textrm{ordered}\\\textrm{quintuple in}~[n]}}\mathbb{P}^2\\
[p_1,\ldots,p_n]&\mapsto(\ldots,\psi_{v_1,\ldots,v_4}(p_{v_5}),\ldots).
\end{align*}
Taking the closure of $\mathbf{B}(3,n)$ under this embedding gives a compactification $\overline{\mathbf{B}}(3,n)$. This was introduced in \cite{GP91}, where it was mistakenly thought to be a the moduli space for $n$-pointed Mustafin degenerations of $\mathbb{P}^2$: see the discussion in \cite[Section~7]{ST22}. To construct the correct compactification, we need to use multigraded Hilbert schemes.

\begin{definition}[\cite{HS04}]
For a commutative ring $k$, let $S=k[x_1,\ldots,x_n]$ be a polynomial ring with a grading by an abelian group~$A$ given by a semigroup homomorphism $\deg\colon\mathbb{N}^n\rightarrow A$. Fix a function
$$h\colon A\rightarrow\mathbb{N}.$$
The {\em multigraded Hilbert scheme} $\mathbf{H}^h_S$ parametrizes all 
ideals in $S$ homogeneous with respect to $\deg$ and with Hilbert function $h$, i.e.~such that
$$\dim_k S_a/I_a=h(a)\quad \hbox{\rm for every}\quad a\in A.$$ 
For every $k$-algebra $R$, the set of $R$-points
$\mathbf{H}^h_S(R)$ is the set of homogeneous ideals $I\subseteq R\otimes_k S$
such that $(R\otimes S)_a/I_a$ is a locally free $R$-module of rank $h(a)$  for every~$a\in A$.
The scheme $\mathbf{H}_S^h$ is quasi-projective over $k$, in fact it is
projective if $1\in S$ is the only monomial of degree~$0$.
\end{definition}

\begin{example}
\label{ex:multigraded-Hilbert-scheme-P^N}
We are interested in the multigraded Hilbert scheme of $(\mathbb{P}^2)^N$ with $N=\binom{n}{4}$. For this purpose, let
\[
S=\Bbbk[x_{ij}\mid i=1,2,3,\ j=1,\ldots,N],
\]
where $\Bbbk$ is our algebraically closed base field and $[x_{1j}:x_{2j}:x_{3j}]$ are homogeneous coordinates of the $j$-th copy of~$\mathbb{P}^2$. We view $S$ with the usual multigrading by~$\mathbb{Z}^N$. We will consider the following two Hilbert functions:

\begin{enumerate}

\item The numerical function $h_1$ is the Hilbert function of the diagonally embedded
\[
\mathbb{P}^2\hookrightarrow(\mathbb{P}^2)^N.
\]
A detailed study of this case can be found in \cite{CS10}.

\item Let $h_2$ be the Hilbert function of a conic in the diagonally embedded $\mathbb{P}^2\hookrightarrow(\mathbb{P}^2)^N$.

\end{enumerate}
\end{example}

The compactification $\overline{\mathbf{X}}_{\mathrm{GP}}(3,n)$ is constructed using $\mathbf{H}_S^{h_1}$ as follows. Consider
\begin{align*}
\Psi\colon\mathbf{B}(3,n)\times\mathbb{P}^2&\hookrightarrow\mathbf{B}(3,n)\times(\mathbb{P}^2)^N\\
([p_1,\ldots,p_n],p)&\mapsto([p_1,\ldots,p_n],(\ldots,\psi_{i_1,\ldots,i_4}(p),\ldots)).
\end{align*}
Moreover, we have the following commutative diagram:
\begin{center}
\begin{tikzpicture}[>=angle 90]
\matrix(a)[matrix of math nodes,
row sep=2em, column sep=2em,
text height=2ex, text depth=1ex]
{\mathbf{B}(3,n)\times\mathbb{P}^2&\mathbf{B}(3,n)\times(\mathbb{P}^2)^N\\
\mathbf{B}(3,n).&\\};
\path[right hook->] (a-1-1) edge node[above]{$\Psi$}(a-1-2);
\path[->] (a-1-1) edge node[]{}(a-2-1);
\path[->] (a-1-2) edge node[]{}(a-2-1);
\end{tikzpicture}
\end{center}
This induces an embedding $\iota_1\colon\mathbf{B}(3,n)\hookrightarrow\mathbf{H}_S^{h_1}$ and we consider
\[
\iota\times\iota_1\colon\mathbf{B}(3,n)\hookrightarrow\overline{\mathbf{B}}(3,n)\times\mathbf{H}_S^{h_1}.
\]
Then, $\overline{\mathbf{X}}_{\mathrm{GP}}(3,n)$ denotes the closure of $\mathbf{B}(3,n)$ under the embedding $\iota\times\iota_1$. Let $\mathrm{Hilb}_{h_1}\left((\mathbb{P}^2)^N\right)$ denote the connected component of the (usual, not multigraded) Hilbert scheme of closed subschemes in $(\mathbb{P}^2)^N$, which parametrizes the diagonally embedded $\mathbb{P}^2$ in $(\mathbb{P}^2)^N$. Thus, the family $\overline{\mathcal{M}}(3,n)\rightarrow\overline{\mathbf{X}}_{\mathrm{GP}}(3,n)$ in Theorem~\ref{thm:main-on-XbarGP(3,n)} is the pullback of the universal family over $\mathrm{Hilb}\left((\mathbb{P}^2)^N\right)$ along the composition
\[
\overline{\mathbf{X}}_{\mathrm{GP}}(3,n)\hookrightarrow\overline{\mathbf{B}}(3,n)\times\mathbf{H}_S^{h_1}\rightarrow\mathbf{H}_S^{h_1}\rightarrow\mathrm{Hilb}_{h_1}\left((\mathbb{P}^2)^N\right).
\]
(For the last morphism see \cite[Example~7.3]{ST22}.) The construction of the sections $s_1,\ldots,s_n$ is not recalled here; we refer the reader to \cite[Section~7]{ST22}.

\begin{remark}
In \cite{ST22}, the reason for using the multigraded Hilbert scheme to define $\overline{\mathbf{X}}_{\mathrm{GP}}(3,n)$ and the usual Hilbert scheme to define the family $\overline{\mathcal{M}}\rightarrow\overline{\mathbf{X}}_{\mathrm{GP}}(3,n)$ is the following. The multigraded perspective allows to easily define forgetful morphisms $\overline{\mathbf{X}}_{\mathrm{GP}}(3,n)\rightarrow\overline{\mathbf{X}}_{\mathrm{GP}}(3,n-1)$ \cite[Proposition~7.10]{ST22}. At the same time, we simply want to construct a family of projective varieties over $\overline{\mathbf{X}}_{\mathrm{GP}}(3,n)$ with no extra structure.
\end{remark}

\subsection{Refined versions of \texorpdfstring{$\overline{\mathbf{X}}_{\mathrm{GP}}(3,n)$}{barXGP(3,n)} and \texorpdfstring{$\overline{\mathbf{S}}_{\mathrm{GP}}(3,6)$}{barSGP(3,6)}}

We now twist the above construction of $\overline{\mathbf{X}}_{\mathrm{GP}}(3,n)$ from Example~\ref{ex:multigraded-Hilbert-scheme-P^N} with the Hilbert function $h_2$. Let $\mathcal{D}(3,n)\subseteq\mathbf{B}(3,n)\times\mathbb{P}^2$ be the Cartier divisor defined by $\psi(p_1,\ldots,p_5,p)=0$, with $\psi$ as in \eqref{eq:6-points-on-a-conic-equation}. As this is the condition for $p_1,\ldots,p_5,p$ to lie on a conic, the fiber of $\mathcal{D}(3,n)\rightarrow\mathbf{B}(3,n)$ over $[p_1,\ldots,p_n]$ equals the unique conic passing through $p_1,\ldots,p_5$. Let us consider the following diagram:
\begin{center}
\begin{tikzpicture}[>=angle 90]
\matrix(a)[matrix of math nodes,
row sep=2em, column sep=2em,
text height=2ex, text depth=1ex]
{\mathcal{D}(3,n)&\mathbf{B}(3,n)\times\mathbb{P}^2&\mathbf{B}(3,n)\times(\mathbb{P}^2)^N\\
\mathbf{B}(3,n).&&\\};
\path[right hook->] (a-1-1) edge node[]{}(a-1-2);
\path[right hook->] (a-1-2) edge node[above]{$\Psi$}(a-1-3);
\path[->] (a-1-1) edge node[]{}(a-2-1);
\path[->] (a-1-3) edge node[]{}(a-2-1);
\end{tikzpicture}
\end{center}
This induces a morphism $\iota_2\colon\mathbf{B}(3,n)\rightarrow\mathbf{H}_S^{h_2}$. Note that this is not injective as non-equivalent choices of $p_1,\ldots,p_5$ may yield the same conic interpolating them.

\begin{definition}
With the notation introduced above, let us consider the embedding
\[
\iota\times\iota_1\times\iota_2\colon\mathbf{B}(3,n)\hookrightarrow\overline{\mathbf{B}}(3,n)\times\mathbf{H}_S^{h_1}\times\mathbf{H}_S^{h_2}.
\]
Then, we define $\overline{\mathbf{X}}_{\mathrm{GP}}^{\mathrm{c}}(3,n)$ to be the closure of $\mathbf{B}(3,n)$ under this embedding (the superscript ``$\mathrm{c}$'' refers to ``conic''). Moreover, we denote by $\overline{\mathcal{M}}^{\mathrm{c}}(3,n)\rightarrow\overline{\mathbf{X}}_{\mathrm{GP}}^{\mathrm{c}}(3,n)$ the pullback of the universal family of surfaces on $\mathrm{Hilb}_{h_1}\left((\mathbb{P}^2)^N\right)$ and by $\overline{\mathcal{D}}^{\mathrm{c}}(3,n)\rightarrow\overline{\mathbf{X}}_{\mathrm{GP}}^{\mathrm{c}}(3,n)$ the pullback of the universal family of curves $\mathrm{Hilb}_{h_2}\left((\mathbb{P}^2)^N\right)$.
\end{definition}

\begin{lemma}
\label{lem:properties-of-c-compactification}
Let $(\overline{\mathcal{M}}^{\mathrm{c}}(3,n),\overline{\mathcal{D}}^{\mathrm{c}}(3,n))\rightarrow\overline{\mathbf{X}}_{\mathrm{GP}}^{\mathrm{c}}(3,n)$ as introduced above. The following hold:
\begin{enumerate}

\item There is a birational morphism $f\colon\overline{\mathbf{X}}_{\mathrm{GP}}^{\mathrm{c}}(3,n)\rightarrow\overline{\mathbf{X}}_{\mathrm{GP}}(3,n)$ extending the identity on the common interior $\mathbf{B}(3,n)$ and inducing the Cartesian square
\begin{center}
\begin{tikzpicture}[>=angle 90]
\matrix(a)[matrix of math nodes,
row sep=2em, column sep=2em,
text height=2ex, text depth=1ex]
{\overline{\mathcal{M}}^{\mathrm{c}}(3,n)&\overline{\mathcal{M}}(3,n)\\
\overline{\mathbf{X}}_{\mathrm{GP}}^{\mathrm{c}}(3,n)&\overline{\mathbf{X}}_{\mathrm{GP}}(3,n).\\};
\path[->] (a-1-1) edge node[]{}(a-1-2);
\path[->] (a-1-1) edge node[left]{$\pi^{\mathrm{c}}$}(a-2-1);
\path[->] (a-1-2) edge node[right]{$\pi$}(a-2-2);
\path[->] (a-2-1) edge node[above]{$f$}(a-2-2);
\end{tikzpicture}
\end{center}
Hence, also $\overline{\mathcal{M}}^{\mathrm{c}}(3,n)\rightarrow\overline{\mathbf{X}}_{\mathrm{GP}}^{\mathrm{c}}(3,n)$ is a family of Mustafin degenerations of $\mathbb{P}^2$.

\item Let $\mathbf{a}\colon\mathrm{Spec}(K)\rightarrow\mathbf{B}(3,n)$ be an arc and let us consider its unique extension $\overline{\mathbf{a}}\colon\mathrm{Spec}(R)\rightarrow\overline{\mathbf{X}}_{\mathrm{GP}}^{\mathrm{c}}(3,n)$. Then, $\overline{\mathbf{a}}^*\overline{\mathcal{M}}^{\mathrm{c}}(3,n)\cong\mathbb{P}(\Sigma_\mathbf{a})$.

\item The sections $s_1,\ldots,s_n$ of $\pi$ pullback to sections $s_1^{\mathrm{c}},\ldots,s_n^{\mathrm{c}}$ of $\pi^{\mathrm{c}}$ such that, for all $x\in\overline{\mathbf{X}}_{\mathrm{GP}}^{\mathrm{c}}(3,n)$, the points $s_1^{\mathrm{c}}(x),\ldots,s_5^{\mathrm{c}}(x)$ lie on the curve $\overline{\mathcal{D}}^{\mathrm{c}}(3,n)_x$.

\end{enumerate}
\end{lemma}

\begin{proof}
The birational morphism $f$ is obtained by restricting to $\overline{\mathbf{X}}_{\mathrm{GP}}^{\mathrm{c}}(3,n)$ the projection
\[
\overline{\mathbf{B}}(3,n)\times\mathbf{H}_S^{h_1}\times\mathbf{H}_S^{h_2}\rightarrow\overline{\mathbf{B}}(3,n)\times\mathbf{H}_S^{h_1}.
\]
Since the compositions $c_1\colon\overline{\mathbf{X}}_{\mathrm{GP}}^{\mathrm{c}}(3,n)\hookrightarrow\overline{\mathbf{B}}(3,n)\times\mathbf{H}_S^{h_1}\times\mathbf{H}^{h_2}\rightarrow\mathbf{H}_S^{h_1}$ and $c_2\colon\overline{\mathbf{X}}_{\mathrm{GP}}^{\mathrm{c}}(3,n)\xrightarrow{f}\overline{\mathbf{X}}_{\mathrm{GP}}(3,n)\hookrightarrow\overline{\mathbf{B}}(3,n)\times\mathbf{H}_S^{h_1}\rightarrow\mathbf{H}_S^{h_1}$ are equal, then
\[
\overline{\mathcal{M}}^{\mathrm{c}}(3,n)=c_1^*\mathcal{H}_1=c_2^*\mathcal{H}_1=f^*\overline{\mathcal{M}}(3,n).
\]
The isomorphism $\overline{\mathbf{a}}^*\overline{\mathcal{M}}^{\mathrm{c}}(3,n)\cong\mathbb{P}(\Sigma_\mathbf{a})$ is a formal consequence of the above equality combined with Theorem~\ref{thm:main-on-XbarGP(3,n)}~(1). Finally, the last statement about $s_1^{\mathrm{c}}(x),\ldots,s_5^{\mathrm{c}}(x)$ lying on $\overline{\mathcal{D}}^{\mathrm{c}}(3,n)_x$ follows by continuity after observing that, for all $i\in[n]$, $s_i(x)=s_i^{\mathrm{c}}(x)$ for all $x\in\mathbf{B}(3,n)$ and $s_1(x),\ldots,s_5(x)\in\mathcal{D}(3,n)$ by definition of $\mathcal{D}(3,n)$.
\end{proof}

\begin{remark}
Although we will not need it, we point out that $\overline{\mathbf{X}}_{\mathrm{GP}}(3,5)\cong\overline{\mathbf{X}}_{\mathrm{GP}}^{\mathrm{c}}(3,5)$. Under this isomorphism, the family of curves $\overline{\mathcal{D}}^{\mathrm{c}}(3,5)\rightarrow\overline{\mathbf{X}}_{\mathrm{GP}}^{\mathrm{c}}(3,5)$ corresponds to $\overline{\mathcal{D}}\rightarrow\overline{\mathbf{X}}_{\mathrm{GP}}(3,5)$ constructed in Section~\ref{sec:Mustafin-comp-for-5-points}. The reason is that the latter family, by construction, fits in the following diagram:
\begin{center}
\begin{tikzpicture}[>=angle 90]
\matrix(a)[matrix of math nodes,
row sep=2em, column sep=2em,
text height=2ex, text depth=1ex]
{\overline{\mathcal{D}}&\overline{\mathcal{P}}&\overline{\mathbf{X}}_{\mathrm{GP}}(3,n)\times(\mathbb{P}^2)^{\binom{5}{4}}\\
\overline{\mathbf{X}}_{\mathrm{GP}}(3,n).&&\\};
\path[right hook->] (a-1-1) edge node[]{}(a-1-2);
\path[right hook->] (a-1-2) edge node[]{}(a-1-3);
\path[->] (a-1-1) edge node[]{}(a-2-1);
\path[->] (a-1-3) edge node[]{}(a-2-1);
\end{tikzpicture}
\end{center}
(See \cite[Page~17033]{ST22}, where also the identification $\overline{\mathbf{X}}_{\mathrm{GP}}(3,5)\cong\overline{\mathrm{M}}_{0,5}$ is used.) This gives a morphism $\overline{\mathbf{X}}_{\mathrm{GP}}(3,5)\rightarrow\mathbf{H}_S^{h_2}$, and hence a morphism $\overline{\mathbf{X}}_{\mathrm{GP}}(3,5)\rightarrow\overline{\mathbf{X}}_{\mathrm{GP}}^{\mathrm{c}}(3,5)$ which is the inverse of $f\colon\overline{\mathbf{X}}_{\mathrm{GP}}^{\mathrm{c}}(3,5)\rightarrow\overline{\mathbf{X}}_{\mathrm{GP}}(3,5)$ as they are both the identity on the interior.
\end{remark}

Next, we define an (a priori) enhanced compactification of $\mathbf{S}(3,6)$ that also comes equipped with a family of curves generically giving the unique conic passing through the six marked points.

\begin{definition}
\label{def:families-of-surfaces-and-curves-on-the-enhanced-compactification}
Let $\overline{\mathbf{S}}_{\mathrm{GP}}^{\mathrm{c}}(3,6)$ be the closure of $\mathbf{S}(3,6)\subseteq\mathbf{X}(3,6)\subseteq\overline{\mathbf{X}}_{\mathrm{GP}}^{\mathrm{c}}(3,6)$. We denote by $(\overline{\mathcal{M}}^{\mathrm{c},\mathrm{sd}}(3,n),\overline{\mathcal{D}}^{\mathrm{c},\mathrm{sd}}(3,n))\rightarrow\overline{\mathbf{S}}_{\mathrm{GP}}^{\mathrm{c}}(3,n)$ with the sections $s_1^{\mathrm{c},\mathrm{sd}},\ldots,s_6^{\mathrm{c},\mathrm{sd}}$ the restriction of the families $(\overline{\mathcal{M}}^{\mathrm{c}}(3,n),\overline{\mathcal{D}}^{\mathrm{c}}(3,n))\rightarrow\overline{\mathbf{X}}_{\mathrm{GP}}^{\mathrm{c}}(3,n)$ and the sections $s_1^{\mathrm{c}},\ldots,s_6^{\mathrm{c}}$ to $\overline{\mathbf{S}}_{\mathrm{GP}}^{\mathrm{c}}(3,n)$. Hence, the objects parametrized by $\overline{\mathbf{S}}_{\mathrm{GP}}^{\mathrm{c}}(3,6)$ are of the form $(X,C;p_1,\ldots,p_6)$, where $(X;p_1,\ldots,p_6)$ is a $6$-pointed Mustafin degenerations of $\mathbb{P}^2$, together with a curve $C$ passing through $p_1,\ldots,p_6$. We call $C$ a \emph{degenerate conic}.
\end{definition}

\subsection{Proof of the main result}

\begin{proposition}
\label{prop:barSGPc(3,6)-and-barSGP(3,6)-have-iso-normalizations}
The compactifications $\overline{\mathbf{S}}_{\mathrm{GP}}^{\mathrm{c}}(3,6)$ and $\overline{\mathbf{S}}_{\mathrm{GP}}(3,6)$ have isomorphic normalizations.
\end{proposition}

\begin{proof}
The birational morphism $f\colon\overline{\mathbf{X}}_{\mathrm{GP}}^{\mathrm{c}}(3,6)\rightarrow\overline{\mathbf{X}}_{\mathrm{GP}}(3,6)$ from Lemma~\ref{lem:properties-of-c-compactification}~(1) restricts to a birational morphism $f^{\mathrm{sd}}\colon\overline{\mathbf{S}}_{\mathrm{GP}}^{\mathrm{c}}(3,6)\rightarrow\overline{\mathbf{S}}_{\mathrm{GP}}(3,6)$ which is the identity on the interior $\mathbf{S}(3,6)$. Let us prove that $f^{\mathrm{sd}}$ is injective.

For this purpose, let $y\in\overline{\mathbf{S}}_{\mathrm{GP}}(3,6)$ and denote by $(X;p_1,\ldots,p_6)$ the $6$-pointed Mustafin degenerations of $\mathbb{P}^2$ parametrized by it. Let us consider an arc $\mathbf{a}\colon\mathrm{Spec}(K)\rightarrow\mathbf{S}_{\mathrm{GP}}(3,6)$ such that  its extension $\overline{\mathbf{a}}\colon\mathrm{Spec}(R)\rightarrow\overline{\mathbf{S}}_{\mathrm{GP}}(3,6)$ satisfies $\overline{\mathbf{a}}(0)=y$. In particular, we have that
\[
(X;p_1,\ldots,p_6)\cong(\mathbb{P}(\Sigma_{\mathbf{a}})_\Bbbk;\overline{a}_1(0),\ldots,\overline{a}_6(0)).
\]
Let $\mathcal{C}\subseteq\mathbb{P}(\Sigma_{\mathbf{a}})$ be the one parameter family of smooth conics such that, for any $t\neq0$, $a_1(t),\ldots,a_6(t)\in\mathcal{C}_t$. Then, by Theorem~\ref{thm:1parameterfamilyconics} we have that $\overline{\mathcal{C}}$ is flat over $\mathrm{Spec}(R)$. By \cite[Chapter~III, Proposition~9.8]{Har77}, this implies that $\overline{\mathcal{C}}=\overline{\mathbf{a}}^*\overline{\mathcal{D}}^{\mathrm{c},\mathrm{sd}}(3,6)$ because they are both flat one-parameter families in $\mathbb{P}(\Sigma_{\mathbf{a}})$ which are equal over $\mathrm{Spec}(K)$, where they both give the unique conic passing through the six linearly general marked points.

Hence, the isomorphism class of $(\mathbb{P}(\Sigma_{\mathbf{a}})_\Bbbk,\overline{\mathcal{C}}_0;\overline{a}_1(0),\ldots,\overline{a}_6(0))$ is parametrized by a point $x\in\overline{\mathbf{S}}_{\mathrm{GP}}^{\mathrm{c}}(3,6)$. Moreover, as $f^{\mathrm{sd}}$ forgets the datum of the curve $\overline{\mathcal{C}}_0$, it holds that $f^{\mathrm{sd}}(x)=y$. The point $x$ is the unique point in $\overline{\mathbf{S}}_{\mathrm{GP}}^{\mathrm{c}}(3,6)$ with this property because again by Theorem~\ref{thm:1parameterfamilyconics} the isomorphim class of the tuple $(\mathbb{P}(\Sigma_{\mathbf{a}})_\Bbbk,\overline{\mathcal{C}}_0;\overline{a}_1(0),\ldots,\overline{a}_6(0))$ is independent from the choice of arc $\mathbf{a}$ and only depends on $y$ (in other words, there is not an exceptional locus of $f^{\mathrm{sd}}$ lying above $y$). This shows that $f^{\mathrm{sd}}$ is injective.

In conclusion, as $f^{\mathrm{sd}}$ is a bijective birational morphism, the theorem follows by Zariski's Main Theorem.
\end{proof}

\begin{proof}[Proof of Theorem~\ref{thm:main-theorem-about-family-of-degenerate-conics}]
Let us consider $(\overline{\mathcal{M}}^{\mathrm{c},\mathrm{sd}}(3,6),\overline{\mathcal{D}}^{\mathrm{c},\mathrm{sd}}(3,6))\rightarrow\overline{\mathbf{S}}_{\mathrm{GP}}^{\mathrm{c}}(3,6)$ with the sections $s_1^{\mathrm{c},\mathrm{sd}},\ldots,s_6^{\mathrm{c},\mathrm{sd}}$ from Definition~\ref{def:families-of-surfaces-and-curves-on-the-enhanced-compactification} and the isomorphism $\overline{\mathbf{S}}_{\mathrm{GP}}(3,6)^\nu\cong\overline{\mathbf{S}}^{\mathrm{c}}(3,6)^\nu$ guaranteed by Proposition~\ref{prop:barSGPc(3,6)-and-barSGP(3,6)-have-iso-normalizations}. Then, we define $(\overline{\mathcal{M}}^{\mathrm{sd}}(3,6),\overline{\mathcal{D}}^{\mathrm{sd}}(3,6))\rightarrow\overline{\mathbf{S}}_{\mathrm{GP}}(3,6)^{\nu}$ to be the family of $6$-pointed Mustafin degenerations of $\mathbb{P}^2$ with degenerate interpolating conic obtained via the following Cartesian square:
\begin{center}
\begin{tikzpicture}[>=angle 90]
\matrix(a)[matrix of math nodes,
row sep=2em, column sep=2em,
text height=2ex, text depth=1ex]
{(\overline{\mathcal{M}}^{\mathrm{sd}}(3,6),\overline{\mathcal{D}}^{\mathrm{sd}}(3,6))&&(\overline{\mathcal{M}}^{\mathrm{c},\mathrm{sd}}(3,6),\overline{\mathcal{D}}^{\mathrm{c},\mathrm{sd}}(3,6))\\
\overline{\mathbf{S}}_{\mathrm{GP}}(3,6)^{\nu}&\overline{\mathbf{S}}_{\mathrm{GP}}^{\mathrm{c}}(3,6)^{\nu}&\overline{\mathbf{S}}_{\mathrm{GP}}^{\mathrm{c}}(3,6).\\};
\path[->] (a-1-1) edge node[]{}(a-1-3);
\path[->] (a-1-1) edge node[]{}(a-2-1);
\path[->] (a-1-3) edge node[]{}(a-2-3);
\path[->] (a-2-1) edge node[above]{$\cong$}(a-2-2);
\path[->] (a-2-2) edge node[above]{}(a-2-3);
\end{tikzpicture}
\end{center}
We let $s_1^{\mathrm{sd}},\ldots,s_6^{\mathrm{sd}}$ be the sections of $\overline{\mathcal{M}}^{\mathrm{sd}}(3,6)\rightarrow\overline{\mathbf{S}}_{\mathrm{GP}}(3,6)^{\nu}$, which lie on $\overline{\mathcal{D}}^{\mathrm{sd}}(3,6)$, obtained by pulling back $s_1^{\mathrm{c},\mathrm{sd}},\ldots,s_6^{\mathrm{c},\mathrm{sd}}$ in Definition~\ref{def:families-of-surfaces-and-curves-on-the-enhanced-compactification}. By construction, $(\overline{\mathcal{M}}^{\mathrm{sd}}(3,6),\overline{\mathcal{D}}^{\mathrm{sd}}(3,6))\rightarrow\overline{\mathbf{S}}_{\mathrm{GP}}(3,6)^{\nu}$ with the sections $s_1^{\mathrm{sd}},\ldots,s_6^{\mathrm{sd}}$ have the required properties.

Finally, if $x\in\overline{\mathbf{S}}_{\mathrm{GP}}(3,6)\setminus\mathbf{S}(3,6)$ parametrizes $(X,C;p_1,\ldots,p_6)$, then the possibilities for it in Tables~\ref{tbl:GPdegenerationsself-dual-1} and \ref{tbl:GPdegenerationsself-dual-2} are obtained by combining Proposition~\ref{prop:strataofselfdualcompactification}, which gives the type of $(X;p_1,\ldots,p_6)$, together with Theorem~\ref{thm:1parameterfamilyconics}, which gives the degenerate conic $C$ after we choose an arc $\mathbf{a}\colon\mathrm{Spec}(K)\rightarrow\mathbf{S}(3,6)$ with limit $x$. In Table~\ref{tbl:GPdegenerationsself-dual-2}, the reason why the pointed curve $(C;p_1,\ldots,p_6)$ generically parametrized by the strata \#12.1, \#20.2, \#20.3, \#20.6, \#23.1 are combinatorially the same as the ones generically parametrized by, respectively, \#12, \#20, \#23 is argued in the last part of the proof of Theorem~\ref{thm:1parameterfamilyconics}.
\end{proof}

\begin{remark}
The divisor $\overline{\mathcal{D}}^{\mathrm{sd}}(3,6)\subseteq\overline{\mathcal{M}}^{\mathrm{sd}}(3,6)$ is not Cartier. To argue this, consider a point in $\overline{\mathbf{X}}_{\mathrm{GP}}(3,6)^\nu$ parametrizing a degeneration $(X,C;p_1,\ldots,p_6)$ of type \#12 (see Figure~\ref{fig:example_marked_conic}). If $\overline{\mathcal{D}}^{\mathrm{sd}}(3,6)$ is Cartier, then $\overline{\mathcal{D}}^{\mathrm{sd}}(3,6)\cap X=C$ is a Cartier divisor on $X$. However, this cannot be because $C$ restricts to the irreducible component of $X$ isomorphic to $\mathbb{P}^1\times\mathbb{P}^1$ giving a point, which is not codimension $1$.
\end{remark}

\subsection{Comparison with \texorpdfstring{$\overline{\mathrm{M}}_{0,6}$}{barM06}}

\begin{theorem}
\label{thm:birat-morph-to-M06bar-and-exc-loci}
There exists a birational morphism  $g\colon\overline{\mathbf{S}}_{\mathrm{GP}}(3,6)^\nu\rightarrow\overline{\mathrm{M}}_{0,6}$ extending the natural isomorphism on the interiors $\mathbf{S}(3,6)\cong\mathrm{M}_{0,6}$. The maximal exceptional loci are the boundary strata of $\overline{\mathbf{S}}_{\mathrm{GP}}(3,6)^\nu$ of Types \#12, \#20, \#23, \#20.2, \#20.6 in Table~\ref{tbl:GPdegenerationsself-dual-1}.
\end{theorem}

\begin{proof}
The family $\overline{\mathcal{D}}^{\mathrm{sd}}(3,6)\rightarrow\overline{\mathbf{S}}_{\mathrm{GP}}(3,6)^\nu$ in Theorem~\ref{thm:main-theorem-about-family-of-degenerate-conics} together with the sections $s_1^{\mathrm{sd}},\ldots,s_6^{\mathrm{sd}}$ is flat by its construction in Section~\ref{sec:construction-of-the-global-family-of-degenerate-conics} and its fibers are stable $6$-pointed rational curves by Theorem~\ref{thm:1parameterfamilyconics}. Hence, this family induces a unique morphism $g\colon\overline{\mathbf{S}}_{\mathrm{GP}}(3,6)^\nu\rightarrow\overline{\mathrm{M}}_{0,6}$ whose geometric interpretation is as follows. If $x\in\overline{\mathbf{S}}_{\mathrm{GP}}(3,6)^{\nu}$ parametrizes $(X,C;p_1,\ldots,p_6)$, then $g(x)$ parametrizes $(C;p_1,\ldots,p_6)$.

Next, we check that $g$ induces an isomorphism between $\mathbf{S}(3,6)$ and $\mathrm{M}_{0,6}$ (from the proof of Proposition~\ref{prop:widehatS(3,6)-is-normal-and-tropical} we already know that they are isomorphic, but here we need to see that this is compatible with $g$). For any $y\in\mathrm{M}_{0,6}$ parametrizing $(\mathbb{P}^1;p_1,\ldots,p_6)$, if we denote by $v\colon\mathbb{P}^1\rightarrow\mathbb{P}^2$ the degree $2$ Veronese embedding, then we have that $y=g(x)$ where $x$ parametrizes $(\mathbb{P}^2,v(\mathbb{P}^1);v(p_1),\ldots,v(p_6))$. Hence, $g$ satisfies $g(\mathbf{S}(3,6))=\mathrm{M}_{0,6}$. Moreover, $g$ is injective on $\mathbf{S}(3,6)$. Suppose $x_1,x_2$ parametrizing $(\mathbb{P}^2,C_1;p_1,\ldots,p_6),(\mathbb{P}^2,C_2;q_1,\ldots,q_6)$ satisfy $g(x_1)=g(x_2)$. Hence, $\varphi\colon C_1\xrightarrow{\cong}C_2$ and $\varphi(p_i)=q_i$ for $i=1,\ldots,6$. By \cite[Lemma~3.9]{Cas26}, we have that $\varphi=\psi|_{C_1}$ for some $\psi\in\mathrm{PGL}_3$, and hence $(\mathbb{P}^2,C_1;p_1,\ldots,p_6)\cong(\mathbb{P}^2,C_2;q_1,\ldots,q_6)$, implying that $x_1=x_2$. In conclusion, $g$ is a birational morphism and it induces an isomorphism between $\mathbf{S}(3,6)$ and $\mathrm{M}_{0,6}$.

Finally, to understand which boundary strata $S\subseteq\overline{\mathbf{S}}_{\mathrm{GP}}(3,6)^{\nu}$ are exceptional with respect to $g$, it is enough to go through all the possibilities for the type of $S$, which are illustrated in the proof of Theorem~\ref{thm:1parameterfamilyconics}.

By Table~\ref{tab:correspondencekapranov'sstrataandcones_selfdual} we can compute $\dim(S)$, because the number of rays generating the corresponding simplicial cone, gives the codimension of $S$ in $\overline{\mathbf{S}}_{\mathrm{GP}}(3,6)^{\nu}$ \cite{Tev07}. On the other hand, let $(X,C;p_1,\ldots,p_6)$ be the generic object parametrized by $S$. Then, $(C;p_1,\ldots,p_6)$ is described in Theorem~\ref{thm:1parameterfamilyconics}. Since $g\colon\overline{\mathbf{S}}_{\mathrm{GP}}(3,6)\rightarrow\overline{\mathrm{M}}_{0,6}$ is surjective, we have that $g(S)$ equals the boundary stratum of $\overline{\mathrm{M}}_{0,6}$ generically parametrizing $(C;p_1,\ldots,p_6)$, and hence $\dim(g(S))$ can be computed as follows:
\begin{table}[ht!]
\begin{tabular}{l|c|c|c|c|c|c}
Type&\#2&\#5&\#10&\#12&\#20&\#23\\
\hline
$\dim(S)$&$2$&$2$&$1$&$2$&$2$&$1$\\
\hline
$\dim(g(S))$&$2$&$2$&$1$&$1$&$0$&$0$
\end{tabular}
\end{table}

From this we see that the boundary strata of types \#2, \#5, \#10 are not in the exceptional locus of $g$, while types \#12, \#20, \#23 are exceptional. The boundary strata of types \#20.2 and \#20.6 are also exceptional because they are $1$-dimensional and contained in boundary strata of type \#20, which are each contracted to a point. The boundary strata of types \#20.3 and \#23.1 are $0$-dimensional, so they cannot be exceptional.

Lastly, let us consider a boundary stratum $S$ of type \#12.1. We have that $\dim(S)=1$ and $S$ is contained inside a type \#12 stratum, which is $2$-dimensional and mapped to a curve. As this is not enough to conclude whether $S$ is exceptional, we need more information from the structure of the morphism $g\colon\overline{\mathbf{S}}_{\mathrm{GP}}(3,6)\rightarrow\overline{\mathrm{M}}_{0,6}$.

Let us consider the restriction to the interiors $\gamma:=g|_{\mathbf{S}(3,6)}\colon\mathbf{S}(3,6)\rightarrow\mathrm{M}_{0,6}$, which we proved is an isomorphism. If $(\mathbb{P}^2,C;p_1,\ldots,p_6)$ is parametrized by $x\in\mathbf{S}(3,6)$, then up to $\mathrm{PGL}_3$ we can suppose that $C$ is the conic $x_1x_3-x_1^2=0$ and that $p_1,\ldots,p_6\in C$ are in the form $[a_1^2:a_1b_1:b_1^2],\ldots,[a_6^2:a_6b_6:b_6^2]$. Then, $\gamma(x)\in\mathrm{M}_{0,6}$ parametrizes, up to $\mathrm{PGL}_2$, $(\mathbb{P}^1;[a_1:b_1],\ldots,[a_6:b_6])$. From this we see that $\gamma^{-1}\colon\mathrm{M}_{0,6}\rightarrow\mathbf{S}(3,6)$ equals the isomorphism induced by the Pachter--Speyer map in the proof of Proposition~\ref{prop:widehatS(3,6)-is-normal-and-tropical} (see in particular the matrix form of this map in \cite[Section~3]{PS04}).

By \cite[Theorem~6.5]{GHSV24} (and the discussion afterwards), we know that there is a linear isomorphism between the fans $\Sigma^\mathrm{sd}(3,6)$ and $\trop(\Gr(2,6)^\circ)$ sending cones to cones. In particular, an FF cone of $\Sigma^\mathrm{sd}(3,6)$ is mapped to a cone of the same dimension in $\trop(\Gr(2,6)^\circ)$. From Table~\ref{tab:correspondencekapranov'sstrataandcones_selfdual} we see that a type \#12.1 stratum $S$ of $\overline{\mathbf{S}}_{\mathrm{GP}}(3,6)^\nu$ corresponds to a CF cone in $\widehat{\Sigma}^{\mathrm{sd}}(3,6)$, which is obtained by splitting an FF cone of $\Sigma^{\mathrm{sd}}(3,6)$. So, the stratum $S$ has to map surjectively to the $1$-dimensional stratum of $\overline{\mathrm{M}}_{0,6}$ corresponding to the FF cone in $\Sigma^\mathrm{sd}(3,6)$, which means $S$ is not exceptional.
\end{proof}

\begin{table}[hbtp]
\centering
\caption{Up to permuting the labels, the $6$-pointed Mustafin degenerations parametrized by $\overline{\mathbf{S}}_{\mathrm{GP}}(3,6)^\nu$ with the degenerate conic. Continues in Table~\ref{tbl:GPdegenerationsself-dual-2}.}
\label{tbl:GPdegenerationsself-dual-1}
\begin{tabular}{|>{\centering\arraybackslash}m{4.5cm}|>{\centering\arraybackslash}m{4.5cm}|>{\centering\arraybackslash}m{4.5cm}|}
\hline
\begin{tikzpicture}[scale=0.5]

	\draw[line width=1pt] (0,6) -- (6,6);
	\draw[line width=1pt] (6,6) -- (6,0);
	\draw[line width=1pt] (6,0) -- (0,6);
	\draw[line width=1pt] (4.5,6) -- (4.5,1.5);
	\draw[line width=1pt] (1.5,4.5) -- (6,4.5);
	\fill[gray!50,nearly transparent] (1.5,4.5) -- (4.5,4.5) -- (4.5,1.5) -- cycle;
	\fill[gray!50,nearly transparent] (1.5,4.5) -- (4.5,4.5) -- (4.5,6) -- (0,6) -- cycle;
	\fill[gray!50,nearly transparent] (4.5,1.5) -- (4.5,4.5) -- (6,4.5) -- (6,0) -- cycle;
	\fill (6,0) circle (5pt);
	\fill (0,6) circle (5pt);
	\fill (5.5,1.5) circle (5pt);
	\fill (5.6,3) circle (5pt);
	\fill (1.5,5.5) circle (5pt);
	\fill (3,5.6) circle (5pt);

	\node at (-.5,6) {$1$};
	\node at (1.7,5) {$2$};
	\node at (3.2,5.1) {$3$};
	\node at (6.5,0) {$4$};
	\node at (5,1.7) {$5$};
	\node at (5.1,3.2) {$6$};

	\node at (1,2) {\#2};

\end{tikzpicture}
&
\begin{tikzpicture}[scale=0.5]

	\draw[line width=1pt] (0,6) -- (6,6);
	\draw[line width=1pt] (6,6) -- (6,0);
	\draw[line width=1pt] (6,0) -- (0,6);
	\draw[line width=1pt] (3,6) -- (6,3);
	\fill[gray!50,nearly transparent] (0,6) -- (6,6) -- (6,0) -- cycle;
	\fill (6,0) circle (5pt);
	\fill (6,6) circle (5pt);
	\fill (0,6) circle (5pt);
	\fill (2,5) circle (5pt);
	\fill (5.5,5) circle (5pt);
	\fill (4,3.5) circle (5pt);

	\node at (-.5,6) {$1$};
	\node at (2.5,5) {$2$};
	\node at (3.5,3.5) {$3$};
	\node at (6.5,0) {$4$};
	\node at (6.5,6) {$5$};
	\node at (5.5,4.4) {$6$};

	\node at (1,2) {\#5};

\end{tikzpicture}
&
\begin{tikzpicture}[scale=0.5]

	\draw[line width=1pt] (0,6) -- (6,6);
	\draw[line width=1pt] (6,6) -- (6,0);
	\draw[line width=1pt] (6,0) -- (0,6);
	\draw[line width=1pt] (18/4,6) -- (18/4,6/4);
	\draw[line width=1pt] (12/4,6) -- (12/4,12/4);
	\draw[line width=1pt] (6,6-6/4) -- (6/4,6-6/4);
	\fill[gray!50,nearly transparent] (6/4,6-6/4) -- (6,6-6/4) -- (6,0) -- cycle;
	\fill[gray!50,nearly transparent] (0,6) -- (12/4,6) -- (12/4,6-6/4) -- (6/4,6-6/4) --cycle;
	\fill (6,0) circle (5pt);
	\fill (0,6) circle (5pt);
	\fill (5.5,3.5) circle (5pt);
	\fill (1.8,5.2) circle (5pt);
	\fill (15/4,5.2) circle (5pt);
	\fill (5.5,1.8) circle (5pt);

	\node at (5,3.5) {$4$};
	\node at (2.3,5.5) {$5$};
	\node at (-.5,6) {$6$};
	\node at (4.2,5.5) {$1$};
	\node at (5,1.8) {$2$};
	\node at (6.5,0) {$3$};

	\node at (1,2) {\#10};

\end{tikzpicture}
\\
\hline
\begin{tikzpicture}[scale=0.5]

\draw [line width=1.pt] (0.,0.)-- (5.,3.);
\draw [line width=1.pt] (3.,3.)-- (8.,0.);
\draw (3.3,1.3) node {$4$};
\draw (4.6,1.3) node {$1$};
\draw (5.6,0.6) node {$2$};
\draw (6.6,0.1) node {$3$};
\draw (2.3,0.6) node {$5$};
\draw (1.3,0.1) node {$6$};

\draw [fill=black] (2.,1.2) circle (4.0pt);
\draw [fill=black] (3.,1.8) circle (4.0pt);
\draw [fill=black] (1.,0.6) circle (4.0pt);
\draw [fill=black] (6.,1.2) circle (4.0pt);
\draw [fill=black] (5.,1.8) circle (4.0pt);
\draw [fill=black] (7.,0.6) circle (4.0pt);

\end{tikzpicture}
&
\begin{tikzpicture}[scale=0.5]
\draw [line width=1.pt] (0.,0.)-- (5.,3.);
\draw [line width=1.pt,color=lightgray] (3.,3.)-- (8.,0.);
\draw (3.5,1.4) node {$1$};
\draw (2.7,1.0) node {$2$};
\draw (1.9,0.5) node {$3$};
\draw (1.0,0.1) node {$4$};
\draw (5,1.1) node {$5$};
\draw (6.3,0.4) node {$6$};

\draw [fill=black] (3.2,1.92) circle (4.0pt);
\draw [fill=black] (5.333333333333331,1.6) circle (4.0pt);
\draw [fill=black] (2.4,1.44) circle (4.0pt);
\draw [fill=black] (1.6,0.96) circle (4.0pt);
\draw [fill=black] (6.666666666666666,0.8) circle (4.0pt);
\draw [fill=black] (0.8,0.48) circle (4.0pt);

\end{tikzpicture}
&
\begin{tikzpicture}[scale=0.5]
\draw (7.5,-3.6) node {$1$};
\draw (9.86,-4.2) node {$2$};
\draw (10.44,-4.9) node {$3$};
\draw (11.,-5.6) node {$4$};
\draw (4.8,-4.5) node {$5$};
\draw (4.1,-5.4) node {$6$};
\draw [line width=1.pt, color = lightgray] (3.,-6.)-- (6.,-2.);
\draw [line width=1.pt] (9.,-2.)-- (12.,-6.);
\draw [line width=1.pt] (4.,-3.)-- (11.,-3.);

\draw [fill=black] (7.5,-3.) circle (4.0pt);
\draw [fill=black] (10.2888,-3.7184) circle (4.0pt);
\draw [fill=black] (10.875,-4.5) circle (4.0pt);
\draw [fill=black] (11.4375,-5.25) circle (4.0pt);
\draw [fill=black] (4.4832,-4.0224) circle (4.0pt);
\draw [fill=black] (3.756,-4.992) circle (4.0pt);

\end{tikzpicture}
\\
\hline
\hline
\begin{tikzpicture}[scale=0.5]

	\draw[line width=1pt] (0,6) -- (6,6);
	\draw[line width=1pt] (6,6) -- (6,0);
	\draw[line width=1pt] (6,0) -- (0,6);
	\draw[line width=1pt] (12/4,6) -- (12/4,6-6/4);
	\draw[line width=1pt] (12/4,6-6/4) -- (18/4,6-12/4);
	\draw[line width=1pt] (18/4,6-12/4) -- (6,6-12/4);
	\draw[line width=1pt] (6-6/4,6-12/4) -- (6-6/4,6/4);
	\draw[line width=1pt] (6-12/4,6-6/4) -- (6/4,6-6/4);
	\fill[gray!50,nearly transparent] (0,6) -- (6,6) -- (6,0) -- cycle;
	\fill (6,0) circle (5pt);
	\fill (0,6) circle (5pt);
	\fill (6,6) circle (5pt);
	\fill (2,5.2) circle (5pt);
	\fill (5,5) circle (5pt);
	\fill (5.2,2) circle (5pt);

	\node at (6.5,6) {$1$};
	\node at (4.5,4.5) {$2$};
	\node at (-.5,6) {$3$};
	\node at (2.5,5.5) {$4$};
	\node at (6.5,0) {$5$};
	\node at (5.5,2.5) {$6$};

	\node at (1,2) {\#12};

\end{tikzpicture}
&
\begin{tikzpicture}[scale=0.5]

	\draw[line width=1pt] (0,6) -- (6,6);
	\draw[line width=1pt] (6,6) -- (6,0);
	\draw[line width=1pt] (6,0) -- (0,6);
	\draw[line width=1pt] (2,6) -- (3,5);
	\draw[line width=1pt] (3,5) -- (4,5);
	\draw[line width=1pt] (4,5) -- (4,6);
	\draw[line width=1pt] (4,5) -- (5,4);
	\draw[line width=1pt] (5,4) -- (6,4);
	\draw[line width=1pt] (5,4) -- (5,3);
	\draw[line width=1pt] (5,3) -- (6,2);
	\draw[line width=1pt] (5,3) -- (4,3);
	\draw[line width=1pt] (4,3) -- (4,2);
	\draw[line width=1pt] (4,3) -- (3,4);
	\draw[line width=1pt] (3,4) -- (3,5);
	\draw[line width=1pt] (3,4) -- (2,4);
	\fill[gray!50,nearly transparent] (2,6) -- (4,6) -- (4,5) -- (3,5) -- cycle;
	\fill[gray!50,nearly transparent] (6,2) -- (6,4) -- (5,4) -- (5,3) -- cycle;
	\fill[gray!50,nearly transparent] (2,4) -- (3,4) -- (4,3) -- (4,2) -- cycle;
	\fill[gray!50,nearly transparent] (0,6) -- (2,6) -- (3,5) -- (3,4) -- (2,4) --cycle;
	\fill[gray!50,nearly transparent] (6,0) -- (6,2) -- (5,3) -- (4,3) -- (4,2) --cycle;
	\fill[gray!50,nearly transparent] (4,6) -- (6,6) -- (6,4) -- (5,4) -- (4,5) --cycle;
	\fill (6,0) circle (5pt);
	\fill (0,6) circle (5pt);
	\fill (6,6) circle (5pt);
	\fill (5,2) circle (5pt);
	\fill (5,5) circle (5pt);
	\fill (2,5) circle (5pt);

	\node at (-.5,6) {$1$};
	\node at (1.5,5.3) {$2$};
	\node at (6.5,6) {$3$};
	\node at (5.5,4.9) {$4$};
	\node at (6.5,0) {$5$};
	\node at (5.3,1.5) {$6$};

	\node at (1,2) {\#20};

\end{tikzpicture}
&
\begin{tikzpicture}[scale=0.5]

	\draw[line width=1pt] (0,6) -- (6,6);
	\draw[line width=1pt] (6,6) -- (6,0);
	\draw[line width=1pt] (6,0) -- (0,6);
	\draw[line width=1pt] (1,5) -- (6,5);
	\draw[line width=1pt] (5,6) -- (5,1);
	\draw[line width=1pt] (2,4) -- (3,4);
	\draw[line width=1pt] (3,6) -- (3,4);
	\draw[line width=1pt] (3,4) -- (4,3);
	\draw[line width=1pt] (4,3) -- (4,2);
	\draw[line width=1pt] (4,3) -- (6,3);
	\fill[gray!50,nearly transparent] (0,6) -- (3,6) -- (3,4) -- (2,4) -- (1,5) -- cycle;
	\fill[gray!50,nearly transparent] (6,0) -- (6,3) -- (4,3) -- (4,2) -- (5,1) -- cycle;
	\fill[gray!50,nearly transparent] (3,4) -- (3,5) -- (5,5) -- (5,3) -- (4,3) -- cycle;
	\fill[gray!50,nearly transparent] (2,4) -- (3,4) -- (4,3) -- (4,2) -- cycle;
	\fill (6,0) circle (5pt);
	\fill (0,6) circle (5pt);
	\fill (5.5,4) circle (5pt);
	\fill (4,5.5) circle (5pt);
	\fill (1.5,5.5) circle (5pt);
	\fill (5.5,1.5) circle (5pt);

	\node at (2.1,5.5) {$6$};
	\node at (-.5,6) {$5$};
	\node at (4.6,5.6) {$2$};
	\node at (6.5,0) {$3$};
	\node at (5.7,4.6) {$1$};
	\node at (5.5,2.1) {$4$};

	\node at (1,2) {\#23};

\end{tikzpicture}
\\
\hline
\begin{tikzpicture}[scale=0.5]
\draw (3.9,-5.6) node {$4$};
\draw (6.7,-3.7) node {$1$};
\draw (8.4,-3.7) node {$2$};
\draw (4.8,-4.5) node {$3$};
\draw (10.1,-4.5) node {$5$};
\draw (10.9,-5.6) node {$6$};
\draw [line width=1.pt,color=lightgray] (3.,-6.)-- (6.,-2.);
\draw [line width=1.pt,color=lightgray] (9.,-2.)-- (12.,-6.);
\draw [line width=1.pt] (4.,-3.)-- (11.,-3.);

\draw [fill=black] (6.625,-3.) circle (4.0pt);
\draw [fill=black] (8.375,-3.) circle (4.0pt);
\draw [fill=black] (4.4904,-4.0128) circle (4.0pt);
\draw [fill=black] (3.7536,-4.9952) circle (4.0pt);
\draw [fill=black] (10.5168,-4.0224) circle (4.0pt);
\draw [fill=black] (11.244,-4.992) circle (4.0pt);

\end{tikzpicture}
&
\begin{tikzpicture}[scale=0.5]
\draw (8.0,-5.0) node {$4$};
\draw (5.5,-4.0) node {$1$};
\draw (5.5,-5.0) node {$2$};
\draw (8.0,-4.0) node {$3$};
\draw (10.5,-4.0) node {$5$};
\draw (10.5,-5.0) node {$6$};
\draw [line width=1.pt, color = black] (4.,-3.)-- (11.,-3.);
\draw [line width=1.pt, color = lightgray] (5.,-2.)-- (5.,-6.);
\draw [line width=1.pt, color = lightgray] (10.,-2.)-- (10.,-6.);
\draw [line width=1.pt, color = lightgray] (7.5,-2.)-- (7.5,-6.);

\draw [fill=black] (5.,-4.013611788132217) circle (4.0pt);
\draw [fill=black] (5.,-4.994667521093667) circle (4.0pt);
\draw [fill=black] (7.5,-4.017594265232974) circle (4.0pt);
\draw [fill=black] (7.5,-4.993599487343801) circle (4.0pt);
\draw [fill=black] (10.,-4.021576742333731) circle (4.0pt);
\draw [fill=black] (10.,-4.9925314535939345) circle (4.0pt);

\end{tikzpicture}
&
\begin{tikzpicture}[scale=0.5]

\draw (11.1,-5.7) node {$4$};
\draw (8.9,-4.0) node {$1$};
\draw (6.0,-4.0) node {$2$};
\draw (10.62,-4.8) node {$3$};
\draw (4.3,-4.8) node {$5$};
\draw (3.8,-5.7) node {$6$};
\draw [color = lightgray,line width=1.pt] (3.,-6.)-- (4.605108332342063,-2.7897833353158727);
\draw [line width=1.pt] (3.1931309156043723,-3.947327932107899)-- (8.464763821237732,-2.5096098669351647);
\draw [line width=1.pt] (6.535236178762268,-2.5096098669351647)-- (11.806869084395656,-3.947327932107905);
\draw [color= lightgray,line width=1.pt] (10.394891667657935,-2.7897833353158736)-- (12.,-6.);

\draw [fill] (5.828947368421053,-3.2284688995215314) circle (4.0pt);
\draw [fill] (9.171052631578949,-3.2284688995215314) circle (4.0pt);
\draw [fill] (3.771929824561402,-4.456140350877196) circle (4.0pt);
\draw [fill] (3.3859649122807,-5.2280701754386) circle (4.0pt);
\draw [fill] (11.228070175438598,-4.456140350877197) circle (4.0pt);
\draw [fill] (11.6140350877193,-5.228070175438599) circle (4.0pt);

\end{tikzpicture}
\\
\hline
\end{tabular}
\end{table}

\begin{table}[hbtp]
\centering
\caption{Continuation of Table~\ref{tbl:GPdegenerationsself-dual-1}. For \#12.1, \#20.$i$, \#23.1 the $6$-pointed degenerate conic is equal to \#12, \#20 and \#23, respectively.}
\label{tbl:GPdegenerationsself-dual-2}
\begin{tabular}{|>{\centering\arraybackslash}m{4cm}|>{\centering\arraybackslash}m{4cm}|>{\centering\arraybackslash}m{4cm}|}
\hline
\begin{tikzpicture}[scale=0.5]

	\draw[line width=1pt] (0,6) -- (6,6);
	\draw[line width=1pt] (6,6) -- (6,0);
	\draw[line width=1pt] (6,0) -- (0,6);
	\draw[line width=1pt] (12/4,6) -- (12/4,6-6/4);
	\draw[line width=1pt] (12/4,6-6/4) -- (18/4,6-12/4);
	\draw[line width=1pt] (18/4,6-12/4) -- (6,6-12/4);
	\draw[line width=1pt] (6-6/4,6-12/4) -- (6-6/4,6/4);
	\draw[line width=1pt] (6-12/4,6-6/4) -- (6/4,6-6/4);
	\draw[line width=1pt] (3,3) -- (3+6/4,3);
	\fill[gray!50,nearly transparent] (0,6) -- (12/4,6) -- (12/4,6-6/4) -- (6/4,6-6/4) -- cycle;
	\fill[gray!50,nearly transparent] (6,0) -- (6,12/4) -- (3,12/4) -- cycle;
	\fill[gray!50,nearly transparent] (12/4,6) -- (6,6) -- (6,3) -- (6-6/4,3) -- (12/4,6-6/4) -- cycle;
	\fill (6,0) circle (5pt);
	\fill (0,6) circle (5pt);
	\fill (6,6) circle (5pt);
	\fill (2,5.2) circle (5pt);
	\fill (5,5) circle (5pt);
	\fill (5.2,2) circle (5pt);

	\node at (6.5,6) {$1$};
	\node at (4.5,4.5) {$2$};
	\node at (-.5,6) {$3$};
	\node at (2.5,5.5) {$4$};
	\node at (6.5,0) {$5$};
	\node at (5.5,2.5) {$6$};

	\node at (1,2) {\#12.1};

\end{tikzpicture}
&
\begin{tikzpicture}[scale=0.5]

	\draw[line width=1pt] (0,6) -- (6,6);
	\draw[line width=1pt] (6,6) -- (6,0);
	\draw[line width=1pt] (6,0) -- (0,6);
	\draw[line width=1pt] (2,6) -- (3,5);
	\draw[line width=1pt] (3,5) -- (4,5);
	\draw[line width=1pt] (4,5) -- (4,6);
	\draw[line width=1pt] (4,5) -- (5,4);
	\draw[line width=1pt] (5,4) -- (6,4);
	\draw[line width=1pt] (5,4) -- (5,3);
	\draw[line width=1pt] (5,3) -- (6,2);
	\draw[line width=1pt] (5,3) -- (4,3);
	\draw[line width=1pt] (4,3) -- (4,2);
	\draw[line width=1pt] (4,3) -- (3,4);
	\draw[line width=1pt] (3,4) -- (3,5);
	\draw[line width=1pt] (3,4) -- (2,4);
	\draw[line width=1pt] (3,3) -- (4,3);
	\draw[line width=1pt] (5,3) -- (6,3);
	\fill[gray!50,nearly transparent] (2,6) -- (4,6) -- (4,5) -- (3,5) -- cycle;
	\fill[gray!50,nearly transparent] (5,3) -- (6,3) -- (6,2) -- cycle;
	\fill[gray!50,nearly transparent] (3,3) -- (4,3) -- (4,2) -- cycle;
	\fill[gray!50,nearly transparent] (0,6) -- (2,6) -- (3,5) -- (3,4) -- (2,4) --cycle;
	\fill[gray!50,nearly transparent] (6,0) -- (6,2) -- (5,3) -- (4,3) -- (4,2) --cycle;
	\fill[gray!50,nearly transparent] (4,6) -- (6,6) -- (6,4) -- (5,4) -- (4,5) --cycle;
	\fill (6,0) circle (5pt);
	\fill (0,6) circle (5pt);
	\fill (6,6) circle (5pt);
	\fill (5,2) circle (5pt);
	\fill (5,5) circle (5pt);
	\fill (2,5) circle (5pt);

	\node at (-.5,6) {$1$};
	\node at (1.5,5.3) {$2$};
	\node at (6.5,6) {$3$};
	\node at (5.5,4.9) {$4$};
	\node at (6.5,0) {$5$};
	\node at (5.3,1.5) {$6$};

	\node at (1,2) {\#20.2};

\end{tikzpicture}
&
\begin{tikzpicture}[scale=0.5]

	\draw[line width=1pt] (0,6) -- (6,6);
	\draw[line width=1pt] (6,6) -- (6,0);
	\draw[line width=1pt] (6,0) -- (0,6);
	\draw[line width=1pt] (2,6) -- (3,5);
	\draw[line width=1pt] (3,5) -- (4,5);
	\draw[line width=1pt] (4,5) -- (4,6);
	\draw[line width=1pt] (4,5) -- (5,4);
	\draw[line width=1pt] (5,4) -- (6,4);
	\draw[line width=1pt] (5,4) -- (5,3);
	\draw[line width=1pt] (5,3) -- (6,2);
	\draw[line width=1pt] (5,3) -- (4,3);
	\draw[line width=1pt] (4,3) -- (4,2);
	\draw[line width=1pt] (4,3) -- (3,4);
	\draw[line width=1pt] (3,4) -- (3,5);
	\draw[line width=1pt] (3,4) -- (2,4);
	\draw[line width=1pt] (3,3) -- (4,3);
	\draw[line width=1pt] (5,3) -- (6,3);
	\draw[line width=1pt] (3,5) -- (3,6);
	\fill[gray!50,nearly transparent] (3,5) -- (3,6) -- (2,6) -- cycle;
	\fill[gray!50,nearly transparent] (5,3) -- (6,3) -- (6,2) -- cycle;
	\fill[gray!50,nearly transparent] (3,3) -- (4,3) -- (4,2) -- cycle;
	\fill[gray!50,nearly transparent] (0,6) -- (2,6) -- (3,5) -- (3,4) -- (2,4) --cycle;
	\fill[gray!50,nearly transparent] (6,0) -- (6,2) -- (5,3) -- (4,3) -- (4,2) --cycle;
	\fill[gray!50,nearly transparent] (4,6) -- (6,6) -- (6,4) -- (5,4) -- (4,5) --cycle;
	\fill (6,0) circle (5pt);
	\fill (0,6) circle (5pt);
	\fill (6,6) circle (5pt);
	\fill (5,2) circle (5pt);
	\fill (5,5) circle (5pt);
	\fill (2,5) circle (5pt);

	\node at (-.5,6) {$1$};
	\node at (1.5,5.3) {$2$};
	\node at (6.5,6) {$3$};
	\node at (5.5,4.9) {$4$};
	\node at (6.5,0) {$5$};
	\node at (5.3,1.5) {$6$};

	\node at (1,2) {\#20.3};

\end{tikzpicture}
\\
\hline
\begin{tikzpicture}[scale=0.5]

	\draw[line width=1pt] (0,6) -- (6,6);
	\draw[line width=1pt] (6,6) -- (6,0);
	\draw[line width=1pt] (6,0) -- (0,6);
	\draw[line width=1pt] (2,6) -- (3,5);
	\draw[line width=1pt] (3,5) -- (4,5);
	\draw[line width=1pt] (4,5) -- (4,6);
	\draw[line width=1pt] (4,5) -- (5,4);
	\draw[line width=1pt] (5,4) -- (6,4);
	\draw[line width=1pt] (5,4) -- (5,3);
	\draw[line width=1pt] (5,3) -- (6,2);
	\draw[line width=1pt] (5,3) -- (4,3);
	\draw[line width=1pt] (4,3) -- (4,2);
	\draw[line width=1pt] (4,3) -- (3,4);
	\draw[line width=1pt] (3,4) -- (3,5);
	\draw[line width=1pt] (3,4) -- (2,4);
	\draw[line width=1pt] (5,3) -- (6,3);
	\draw[line width=1pt] (3,5) -- (3,6);
	\fill[gray!50,nearly transparent] (3,5) -- (3,6) -- (2,6) -- cycle;
	\fill[gray!50,nearly transparent] (5,3) -- (6,3) -- (6,2) -- cycle;
	\fill[gray!50,nearly transparent] (2,4) -- (3,4) -- (4,3) -- (4,2) -- cycle;
	\fill[gray!50,nearly transparent] (0,6) -- (2,6) -- (3,5) -- (3,4) -- (2,4) --cycle;
	\fill[gray!50,nearly transparent] (6,0) -- (6,2) -- (5,3) -- (4,3) -- (4,2) --cycle;
	\fill[gray!50,nearly transparent] (4,6) -- (6,6) -- (6,4) -- (5,4) -- (4,5) --cycle;
	\fill (6,0) circle (5pt);
	\fill (0,6) circle (5pt);
	\fill (6,6) circle (5pt);
	\fill (5,2) circle (5pt);
	\fill (5,5) circle (5pt);
	\fill (2,5) circle (5pt);

	\node at (-.5,6) {$1$};
	\node at (1.5,5.3) {$2$};
	\node at (6.5,6) {$3$};
	\node at (5.5,4.9) {$4$};
	\node at (6.5,0) {$5$};
	\node at (5.3,1.5) {$6$};

	\node at (1,2) {\#20.6};

\end{tikzpicture}
&
\begin{tikzpicture}[scale=0.5]

	\draw[line width=1pt] (0,6) -- (6,6);
	\draw[line width=1pt] (6,6) -- (6,0);
	\draw[line width=1pt] (6,0) -- (0,6);
	\draw[line width=1pt] (1,5) -- (6,5);
	\draw[line width=1pt] (5,6) -- (5,1);
	\draw[line width=1pt] (2,4) -- (3,4);
	\draw[line width=1pt] (3,6) -- (3,4);
	\draw[line width=1pt] (3,4) -- (4,3);
	\draw[line width=1pt] (4,3) -- (4,2);
	\draw[line width=1pt] (4,3) -- (6,3);
	\draw[line width=1pt] (3,3) -- (4,3);
	\fill[gray!50,nearly transparent] (0,6) -- (3,6) -- (3,4) -- (2,4) -- (1,5) -- cycle;
	\fill[gray!50,nearly transparent] (6,0) -- (6,3) -- (4,3) -- (4,2) -- (5,1) -- cycle;
	\fill[gray!50,nearly transparent] (3,4) -- (3,5) -- (5,5) -- (5,3) -- (4,3) -- cycle;
	\fill[gray!50,nearly transparent] (3,3) -- (4,3) -- (4,2) -- cycle;
	\fill (6,0) circle (5pt);
	\fill (0,6) circle (5pt);
	\fill (5.5,4) circle (5pt);
	\fill (4,5.5) circle (5pt);
	\fill (1.5,5.5) circle (5pt);
	\fill (5.5,1.5) circle (5pt);

	\node at (2.1,5.5) {$6$};
	\node at (-.5,6) {$5$};
	\node at (4.6,5.6) {$2$};
	\node at (6.5,0) {$3$};
	\node at (5.7,4.6) {$1$};
	\node at (5.5,2.1) {$4$};

	\node at (1,2) {\#23.1};

\end{tikzpicture}
&
\\
\hline
\end{tabular}
\end{table}

\newcommand{\etalchar}[1]{$^{#1}$}

\end{document}